\documentclass[reqno]{amsart}

\usepackage[utf8]{inputenc}
\usepackage[T1]{fontenc}
\usepackage[british]{babel,isodate}
\cleanlookdateon
\usepackage{lmodern}
\usepackage{enumerate}
\usepackage{tikz}
\usepackage{tikz-cd}
\usetikzlibrary{arrows}
\usepackage{cancel}
\usetikzlibrary{babel}
\usepackage{lscape}
\usepackage{url}

\usepackage{hyperref}
\hypersetup{linktocpage,colorlinks=true,citecolor=purple,linkcolor=blue,urlcolor=blue,filecolor=black}

\usepackage{float}
\usepackage{graphicx, import}
\usepackage{mathtools}
\usepackage{dsfont}
\usepackage{pinlabel}
\usepackage{yhmath}

\usepackage{xargs}                      
\usepackage{xcolor} 
\definecolor{OliveGreen}{rgb}{0,0.6,0}
 
\usepackage[colorinlistoftodos, prependcaption,textsize=tiny]{todonotes}
\newcommandx{\unsure}[2][1=]{\todo[linecolor=red,backgroundcolor=red!25,bordercolor=red,#1]{#2}}
\newcommandx{\change}[2][1=]{\todo[linecolor=blue,backgroundcolor=blue!25,bordercolor=blue,#1]{#2}}
\newcommandx{\info}[2][1=]{\todo[linecolor=OliveGreen,backgroundcolor=OliveGreen!25,bordercolor=OliveGreen,#1]{#2}}

\newcommand{\centre}[1]{\begin{array}{c} #1 \end{array}}

\usepackage{stackengine}

\usepackage{amsmath,amsthm,amssymb, amsfonts, amscd}
\usepackage{rotating,subcaption}
\usepackage{url}
\usepackage{graphicx,bm}
\usepackage{mathtools}
\usepackage[mathscr]{euscript}
\allowdisplaybreaks
\usepackage{lipsum}
\usepackage{bm}
\usepackage[capitalize]{cleveref} 

\usepackage{xpatch}
\makeatletter
\AtBeginDocument{\xpatchcmd{\@thm}{\thm@headpunct{.}}{\thm@headpunct{}}{}{}}
\makeatother

\newtheorem{theorem}{Theorem}[section]
\newtheorem{proposition}[theorem]{Proposition}
\newtheorem{lemma}[theorem]{Lemma}
\newtheorem{corollary}[theorem]{Corollary}
\newtheorem{conjecture}[theorem]{Conjecture}
\newtheorem{problem}[theorem]{Problem}

\theoremstyle{definition}

\newtheorem{example}[theorem]{Example}
\newtheorem{examples}[theorem]{Examples}

\theoremstyle{remark}

\newtheorem{remark}[theorem]{Remark}
\newtheorem{notation}[theorem]{Notation}
\newtheorem{warning}[theorem]{Warning}

\numberwithin{equation}{section}

\newcommand{\hooklongrightarrow}{\lhook\joinrel\longrightarrow}
\newcommand{\longtwoheadrightarrow}{\relbar\joinrel\twoheadrightarrow}

\newcommand{\Z}{\mathbb{Z}}
\newcommand{\Q}{\mathbb{Q}}
\newcommand{\Qe}{\mathbb{Q}_\varepsilon}
\newcommand{\Qeh}{\mathbb{Q}_\varepsilon [[h]]}
\newcommand{\R}{\mathbb{R}}

\newcommand{\A}{\mathcal{A}}

\newcommand{\D}{\mathbb{D}}
\newcommand{\T}{\mathcal{T}}
\newcommand{\id}{\mathrm{Id}}

\let\hom\relax
\newcommand{\hom}[3]{\mathrm{Hom}_{#1}(#2,#3)}
\renewcommand{\to}{\longrightarrow}

\DeclareMathOperator{\im}{\mathrm{Im}}

\newcommand{\toiso}{\overset{\cong}{\to}}
\newcommand{\otimeshat}{\mathbin{\hat\otimes}}

\usepackage{accents}
\newcommand{\uhat}{\underaccent{\check}}

\makeatletter
\newcommand{\cupr@tip}{\text{\raisebox{-0.1ex}{$\m@th\hat{}$}}}
\newcommand{\cupr}{\mathbin{\cup\cupr@}}

\newcommand{\cupr@}{%
  \mathchoice
  {\mkern-1.35mu\cupr@tip}
  {\mkern-1.35mu\cupr@tip}
  {\mkern-1.55mu\cupr@tip}
  {\mkern-1.875mu\cupr@tip}
}
\makeatother

\makeatletter
\newcommand{\capr@tip}{\text{\raisebox{0.47ex}{$\m@th\uhat{}$}}}
\newcommand{\capr}{\mathbin{\capr@\cap}}

\newcommand{\capr@}{%
  \mathchoice
  {\mkern11.6mu\capr@tip\mkern-11.6mu}
  {\mkern11.4mu\capr@tip\mkern-11.4mu}
  {\mkern11.1mu\capr@tip\mkern-11.1mu}
  {\mkern10.2mu\capr@tip\mkern-10.2mu}
}
\makeatother

\makeatletter
\newcommand{\capl@tip}{\text{\raisebox{0.47ex}{$\m@th\uhat{}$}}}
\newcommand{\capl}{\mathbin{\capl@\cap}}

\newcommand{\capl@}{%
  \mathchoice
  {\mkern2.1mu\capl@tip\mkern-2.1mu}
  {\mkern2.1mu\capl@tip\mkern-2.1mu}
  {\mkern2.3mu\capl@tip\mkern-2.3mu}
  {\mkern2.1mu\capl@tip\mkern-2.1mu}
}
\makeatother

\makeatletter
\newcommand{\cupl@tip}{\text{\raisebox{-0.1ex}{$\m@th\hat{}$}}}
\newcommand{\cupl}{\mathbin{\cupl@\cup}}

\newcommand{\cupl@}{%
  \mathchoice
  {\mkern1.35mu\cupl@tip\mkern-1.35mu}
  {\mkern1.35mu\cupl@tip\mkern-1.35mu}
  {\mkern1.55mu\cupl@tip\mkern-1.55mu}
  {\mkern1.875mu\cupl@tip\mkern-1.875mu}
}
\makeatother

\DeclareFontFamily{U}{mathx}{}
\DeclareFontShape{U}{mathx}{m}{n}{ <-> mathx10 }{}
\DeclareSymbolFont{mathx}{U}{mathx}{m}{n}
\DeclareFontSubstitution{U}{mathx}{m}{n}
\DeclareMathAccent{\widecheck}{0}{mathx}{"71}

\begin{document}


\title{On Bar-Natan - van der Veen's perturbed Gaussians}

\date{\today}

\author{Jorge Becerra}
\address{Bernouilli Institute, University of Groningen, Nijenborgh 9, 9747 AG, Groningen, The Netherlands}
\email{\href{mailto:j.becerra@rug.nl}{j.becerra@rug.nl}}
\urladdr{ \href{https://sites.google.com/view/becerra/}{https://sites.google.com/view/becerra/}} 




\begin{abstract}
We elucidate further properties of the novel family of polynomial time knot polynomials devised by Bar-Natan and van der Veen based on the Gaussian calculus of generating series for noncommutative algebras. These polynomials determine all coloured Jones polynomials and the simplest of these is expected to coincide with the one-variable 2-loop polynomial. We prove a conjecture stating that half of these polynomials vanish and give concrete formulas for three of these knot polynomial invariants. We also study the behaviour of these polynomials under the connected sum of knots.

\end{abstract}



\maketitle

\setcounter{tocdepth}{1}
\tableofcontents


\section{Introduction}

Quantum invariants of knots have been widely studied for thirty years now. These are typically constructed from the representation theory of \textit{quantum groups} \cite{jimbo85,chari_pressley,kassel}, which is an umbrella term that includes quantisations (also called deformations) of universal enveloping algebras of complex semisimple Lie algebras. Among those, the quantum algebra $U_h (\mathfrak{sl}_2)$ stands out, as not only it is the most elementary to study but also gives rise to a sequence of knot invariants called the coloured Jones polynomials, the first of them being the celebrated Jones polynomial \cite{jones}.

The main feature of these algebras in constructing quantum invariants is that they are \textit{ribbon}. Roughly speaking, a ribbon algebra $A$ has two preferred, invertible elements $$R = \sum_i \alpha_i \otimes \beta_i \in A \otimes A \qquad , \qquad \kappa \in A,$$ (see \cref{section Topological ribbon Hopf algebras} for precise definitions) which turn the category $\mathsf{Mod}_A$ of finite-dimensional representations into a ribbon or tortile category \cite{ohtsukibook, turaev}. In this paper, we will study a knot invariant which is also produced by a ribbon algebra and contains the invariants produced by all its representations, which gives it the name of \textit{universal knot invariant} associated to $A$ \cite{lawrence, habiro}.

Let us sketch the construction of the universal invariant here, for details see \cref{section The universal tangle invariant}. Given a diagram of a (long, framed, oriented) knot $K$ where all crossings are assumed to point upwards, place copies of $R$ (resp. $R^{-1}$) on the positive (resp. negative) crossings, where the first factor goes in the overpass and the second factor in the underpass. Likewise, place a copy of $\kappa$ (resp. $\kappa^{-1}$) on each counterclockwise full rotation (resp. clockwise full rotation), and multiply the beads along the knot following the orientation to obtain an element $Z_A(K) \in A$. Below we illustrate the construction for the right-handed trefoil $T_{2,3}$:

\vspace{0.3cm}\noindent
 \begin{minipage}{.45\textwidth}
 \begin{equation*} 
\labellist \small  \hair 2pt
\pinlabel{$ \color{magenta} \bullet$} at 61 123
\pinlabel{$ \color{magenta} \alpha_i$} at -12 123
\pinlabel{$ \color{magenta} \bullet$} at 138 123
\pinlabel{$ \color{magenta} \beta_i$} at 206 123
\pinlabel{$ \color{blue} \bullet$} at 56 268
\pinlabel{$ \color{blue} \alpha_j$} at -12 268
\pinlabel{$ \color{blue} \bullet$} at 140 268
\pinlabel{$ \color{blue} \beta_j$} at 206 268
\pinlabel{$ \color{OliveGreen} \bullet$} at 56 419
\pinlabel{$ \color{OliveGreen} \alpha_\ell$} at -12 419
\pinlabel{$ \color{OliveGreen} \bullet$} at 135 419
\pinlabel{$ \color{OliveGreen} \beta_\ell$} at 206 419
\pinlabel{$ \color{orange} \bullet$} at 347 313
\pinlabel{$ \color{orange} \kappa^{-1}$} at 423 325
\endlabellist
\includegraphics[width=0.3\textwidth]{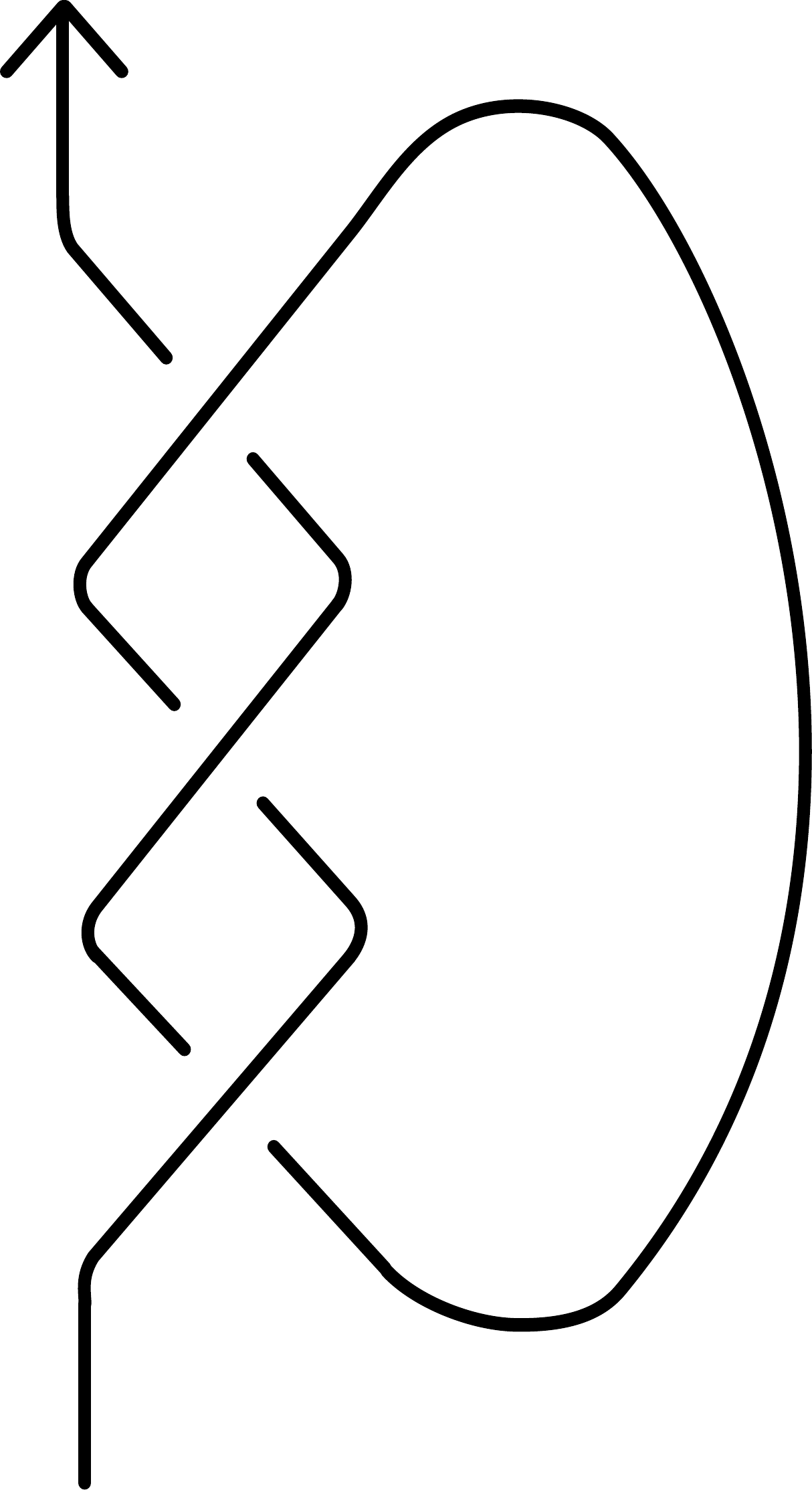} 
\end{equation*}
 \end{minipage}
  \begin{minipage}{.45\textwidth}
 $$Z_A(T_{2,3}) = \sum_{i,j,\ell} \color{magenta}{\alpha_i} \  \color{blue}{\beta_j} \  \color{OliveGreen}{\alpha_\ell} \  \color{orange}{\kappa^{-1}} \  \color{magenta}{\beta_i} \  \color{blue}{\alpha_j} \ \color{OliveGreen}{\beta_\ell}.$$
 \end{minipage}
 \vspace{0.3cm}

In the last two decades the universal invariant has been extensively studied \cite{habiro_WRT, suzuki10,suzuki13,suzuki16,suzuki18}. Although the definition looks harmless, computing this invariant may turn out to be a challenging task. The main difficulty lies in the fact that, to obtain refined knot invariants, one has to take noncommutative, infinite-dimensional ribbon algebras where typically  the elements $R$ and $\kappa$ are expressed by infinite sums\footnote{Such infinite sums make sense as long as the algebra has as base ring a power series ring over an indeterminate $h$ and the algebra is considered as a topological algebra endowed with the $h$-adic topology.} of products of the generators of the algebra, and it is hopeless to obtain a concrete expression out of these products for an arbitrary knot. To the author's knowledge, the values  $Z_{U_h (\mathfrak{sl}_2)} (K)$ are only known for the four  knots $K$ with four or fewer crossings \cite[\S 3.4]{habiro_sl2}.

In a series of publications \cite{barnatanveenpolytime,barnatanveengaussians}, D. Bar-Natan and R. van der Veen developed a novel technique to tackle the problem of performing iterated products in noncommutative algebras as above efficiently. The observation that triggers the so-called \textit{Gaussian calculus} is that the algebras that typically appear in the theory of quantum groups (e.g. the quantisations $U_h (\mathfrak{g})$  of universal enveloping algebras) possess Poincaré-Birkhoff-Witt-type bases and hence they are isomorphic (as topological modules) to  polynomial rings in commuting variables, with an extra parameter $h$ adjoint as power series.


To illustrate this idea, let us omit the parameter $h$ for simplicity. The linear algebra finite-dimensional $k$-vector space isomorphism $\hom{k}{V}{W} \cong V^* \otimes W$ can be upgraded, if $V\cong k[x]$ and $W \cong k[y]$, to a $k$-linear isomorphism $$\hom{k}{V}{W} \toiso k[y][[\xi]].$$ This allows to view $k$-linear maps as power series, and the composite of maps as some pairing between power series. Applying this point of view to an algebra $A$, we can view the multiplication map $\mu: A \otimes A \to A$ as a power series $\bar{\mu}$ (for a more general discussion on Gaussian calculus see \cref{sect mot example} and \ref{sect contraction}).

Bar-Natan and van der Veen use then this approach for a (topological) ribbon algebra $\D$ over the ring $\Qeh$, where $\Qe:= \Q [\varepsilon]$. The key property of this algebra is that, from the perspective above, the series $\bar{\mu}$, as well as the series corresponding to the elements $R$ and $\kappa$, can be more succinctly written as \textit{perturbed Gaussians}, that is,  power series of the form
\begin{equation}\label{eq perturbed introduction}
    e^G (P_0 + P_1 \varepsilon + P_2 \varepsilon^2 + P_3 \varepsilon^3 + \cdots)  
\end{equation}
where $G$ and every $P_i$ are finite expressions, and the pairing of two such power series, which can be performed using a concrete formula (see \cref{thm contraction}), takes the same form.


The consequence of this novel approach is that for any knot $K$, one can compute the value
\begin{equation}\label{eq ZDK modN}
Z_\D (K) \pmod {\varepsilon^N} 
\end{equation}
using truncations of \eqref{eq perturbed introduction}, which are finite expressions. The main computational feature of this approach is that, unlike invariants coming from representation theory, which are computed in exponential time, the value \eqref{eq ZDK modN} can be computed in polynomial time. 

It is worth mentioning that the algebra $\D$ surjects onto $U_h (\mathfrak{sl}_2)$ and hence  the invariant $Z_\D (K)$ determines $Z_{U_h (\mathfrak{sl}_2)} (K)$ and therefore all coloured Jones polynomials. In fact, we expect $Z_\D(K)$ to be equivalent to the collection of coloured Jones polynomials. Moreover, $Z_\D (K)$ is closely related to  Rozansky's rational expansion of the coloured Jones polynomials, see \cref{subsec coloured Jones}.

Using the toolbox sketched above,  it is  shown in \cite{barnatanveengaussians} that for a 0-framed knot $K$, the universal invariant $Z_\D (K)$ is completely determined by a family of knot polynomial invariants
\begin{equation}
    \rho_K^{i,j} \in \Q [t,t^{-1}] \qquad , \qquad i \geq 0, \ 0 \leq j \leq 2i,
\end{equation}
and the celebrated Alexander polynomial $\Delta_K \in \Z [t+t^{-1}]$ of $K$, so that studying features of $Z_\D (K)$ amounts to studying features of the given family of polynomials.

Among this collection of polynomials, the value $\rho_K^{1,0}$ stands out: only the pair $(\Delta_K,\rho_K^{1,0})$ produces the strongest  knot invariant  that can be computed efficiently known up to date \cite{barnatanveenpolytime} : 
\begin{itemize}
\item it is computable in order $O(c^5)$, where $c$ is the number of crossings of the knot diagram,
\item it takes 2883 different values on the set of the 2978 prime knots with twelve or fewer crossings. As a comparison, if $Kh_\Q(K)$ denotes rational Khovanov homology and $P(K)$ denotes the HOMFLY-PT polynomial of $K$, then the pair $(Kh_\Q(K), P(K))$ takes about a hundred  fewer values than the pair $(\Delta_K,\rho_K^{1,0})$.
\end{itemize}
In fact, we expect this polynomial $\rho_K^{1,0}$ to coincide to the one-variable 2-loop polynomial, see \cref{sect relation 2-loop}. For the rest of the family, it is expected that some of these polynomials are weaker. In fact, our first main result is that half of these polynomials are trivial.

\begin{theorem}[\cref{thm rho ij =0 for j>i}, a Conjecture in \cite{barnatanveengaussians}]\label{thm A}
For any $0$-framed knot $K$, $$\rho_K^{i,j}(t)=0$$ for $j>i>0$.
\end{theorem}

Our second main result, which gives evidence to a Conjecture in \cite{barnatanveengaussians}, states that the triple $(\Delta_K,\rho_K^{1,0},\rho_K^{2,0})$ completely determines the value of $ Z_\D (K)$ modulo $\varepsilon^3$. Even more, we are able to give concrete formulas for the rest of the knot polynomial invariants in terms of this triple (a formula for (1) was already given in the aforementioned paper).

\begin{theorem}[\cref{thm rho 2 sth}]\label{thm B}
For any 0-framed knot $K$ we have
\vspace*{0.3cm}
\begin{enumerate}
\setlength\itemsep{0.7em}
\item $\displaystyle \rho_K^{1,1}(t)= \frac{2t}{1-t} \Delta_K ' (t),  $
\item $\displaystyle \rho_K^{2,1}(t)= \frac{2t \Delta_K^4(t)}{t-1} \left( \frac{\rho_K^{1,0}}{ \Delta_K^3}  \right) ' (t),  $
\item $\displaystyle \rho_K^{2,2}(t)= t \left( \frac{\Delta_K(t)}{1-t} \right)^3 \left( (3-t) (\Delta_K^{-1})'(t) + 2t(1-t) (\Delta_K^{-1})''(t)  \right).$
\end{enumerate}
\vspace*{0.1cm}
\end{theorem}

The strategy we followed in these two results is the same, and we briefly sketch it here. Every oriented 0-framed knot $K$ arises from a tangle $L$ with $2g$ open components, $g \geq 1$, by adding a few crossings, doubling each strand, reversing the orientation of some of the strands and merging endpoints, so  that $K$ equals the \textit{thickening} of $L$, denoted $\widecheck{Th}(L)$, as illustrated below:
\begin{equation}
\centre{
\labellist \small \hair 2pt
\pinlabel{$\vdots$}  at 514 880
\pinlabel{$\vdots$}  at 1756 870
\pinlabel{$\vdots$}  at 2057 870
\pinlabel{$\vdots$}  at 2511 870
\pinlabel{ \normalsize $L$}  at 131 1437
\pinlabel{ \LARGE $\overset{\widecheck{Th}}{\longmapsto}$}  at 894 894
\endlabellist
\centering
\includegraphics[width=0.7\textwidth]{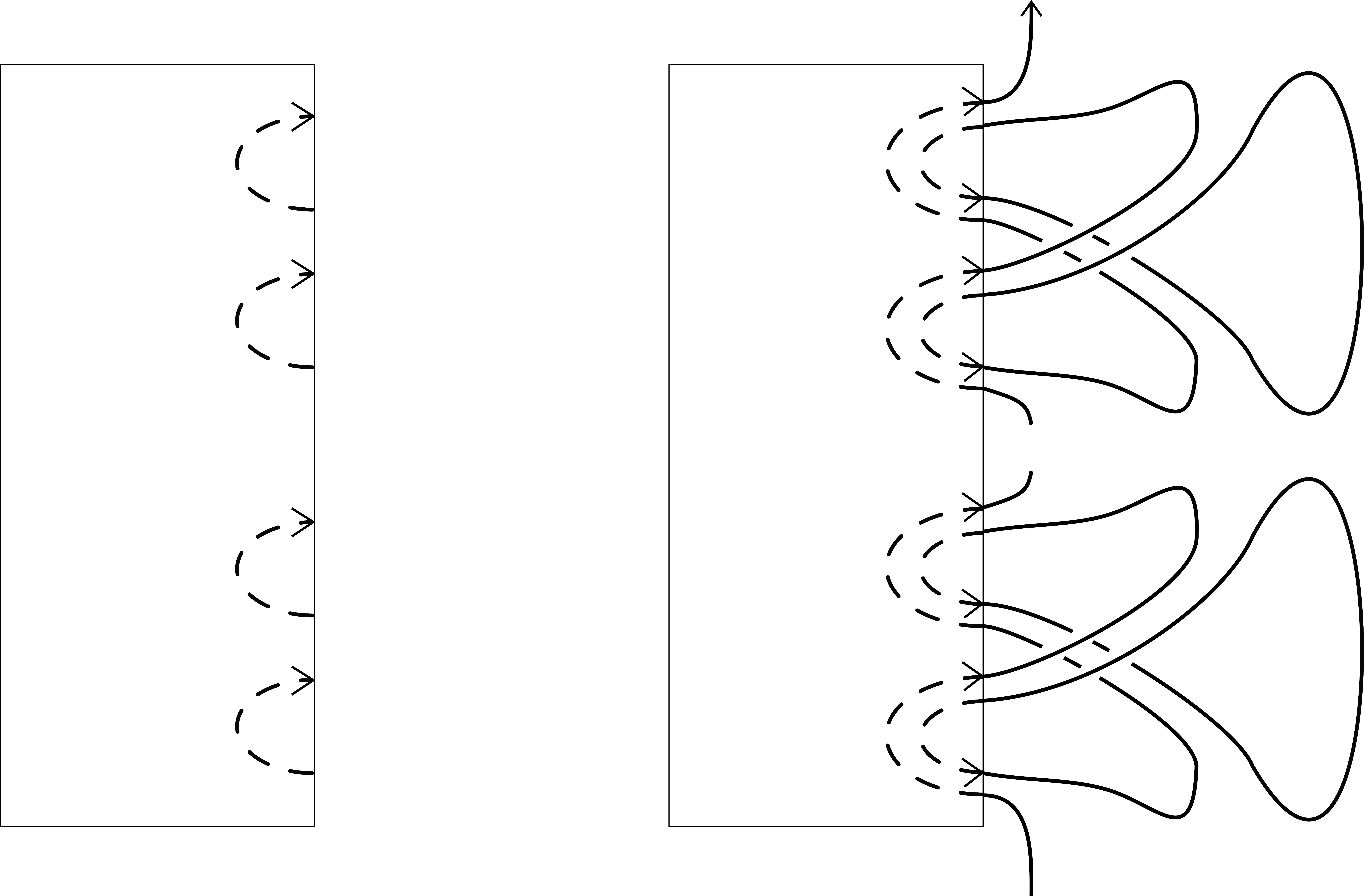}}
\end{equation}
The universal invariant $Z_\D$ is \textit{natural} with respect to the topological operations mentioned before, so that there exists an algebraic thickening map $$ Th: \D^{\otimeshat 2g} \to \D$$
with the property that $$Z_\D (\widecheck{Th}(L)) = Th (Z_\D(L)),$$
where the universal invariant is extended to open tangles in the previous equality. This implies that to study properties of the knot polynomial invariants $\rho_K^{i,j}$, it suffices to study properties of $Z_\D(L)$ and  the series in commuting variables associated to the thickening map.




\subsection*{Organisation of the paper} The rest of the paper is organised as follows: in \cref{sec rot tangles}, we review oriented, framed tangles and introduce rotational tangles as a convenient class. In \cref{sec univ inv}, we recall the notion of topological ribbon Hopf algebra and the universal tangle invariant subject to such ribbon algebra, bringing into focus the naturality of the universal invariant with respect to the Hopf operations. In \cref{sec gauss}, we review Bar-Natan - van der Veen's Gaussian calculus for the ribbon algebra $\D$. We have included extensive motivation to the usage of commutative generating series in noncommutative algebra. In \cref{sect ZD}, we introduce a thickening map defined for vertical bottom tangles, and give an explicit argument for the naturality of with map with respect to the universal tangle invariant. We also describe an algebraic property about the image of the elementary building block of the thickening map, namely the band map, that will lead to an easy proof of \cref{thm A}. In this section we also review how $Z_\D (K)$ mod $\varepsilon$ recovers the Alexander polynomial of the knot and will build upon that to prove \cref{thm B}. We also give explicit formulas of the family of polynomials $\rho_K^{i,j}$ for the connected sum of knots. In \cref{sect band map}, we give proofs of technical results about the structure of the generating series of the band map that will be key to demonstrate   \cref{thm B}. Lastly in \cref{sect perpectives}, we discuss several open problems and conjectures concerning the universal invariant $Z_\D$. \cref{sect appendix} includes a computer implementation to show a couple of properties about the band map that are beyond the scope of manual calculation.

\subsection*{Acknowledgments} The author would like to thank Roland van der Veen for many helpful discussions and suggestions on the content of this paper.

\section{Rotational tangles}\label{sec rot tangles}

In this section, we review oriented, framed tangles in the cube and pay special attention to those whose strands start at the bottom of the cube and end at the top, as they admit a particularly convenient description called rotational.

\subsection{The category \texorpdfstring{$\mathcal{T}$}{T}  of tangles}

 Let $D^1 =[-1,1] \subset \R$ be the one-dimensional disc and let $n,m \geq 0$ be  integers. An  \textit{(oriented, framed)  tangle} is an isotopy class of an embedding $$L:  \left( \coprod_n D^1 \times D^1 \right) \amalg  \left( \coprod_m D^1 \times S^1 \right)  \hooklongrightarrow (D^1)^{\times 3}$$
with the property that it restricts to a orientation-preserving homeomorphism $$\coprod_n   D^1 \times \{-1, 1 \}       \toiso \Big(   \bigcup_{i=1}^{n_1} F_i \Big)  \cup  \Big(   \bigcup_{i=1}^{n_2} H_i \Big)  $$
where $n_1, n_2  \geq 0$, $ n_1+n_2=2n$,  $F_i := \left[ \frac{2i-1}{2n_1+1} , \frac{2i}{2n_1+1} \right] \times 0 \times -1 \subset  (D^1)^{\times 3}$ and $H_i := \left[ \frac{2i-1}{2n_2+1} , \frac{2i}{2n_2+1} \right] \times 0 \times 1 \subset  (D^1)^{\times 3}$, with the orientations on $D^1 \times \pm 1$ induced by the usual one in $D^1 \times D^1$ and the orientations on $F_i$ and $H_i$ are induced by the ones of the positive and negative direction, respectively.  The isotopy is understood to be relative to $\coprod_n (D^1 \times \pm 1  )$. Moreover, the cores of the strips $ 0 \times D^1 $ and annuli $0 \times S^1$ are also endowed with an orientation. If $m=0$, we say that the tangle is \textit{open}.



As it is customary we only depict the cores of the strips of a tangle on the plane in general position with the blackboard framing, and we call these the \textit{strands} of the tangle, see the figure below. We consider these diagrams up to oriented, framed Reidemeister moves of type I, II and III.

\begin{equation*} 
\centering
\includegraphics[width=0.6\textwidth]{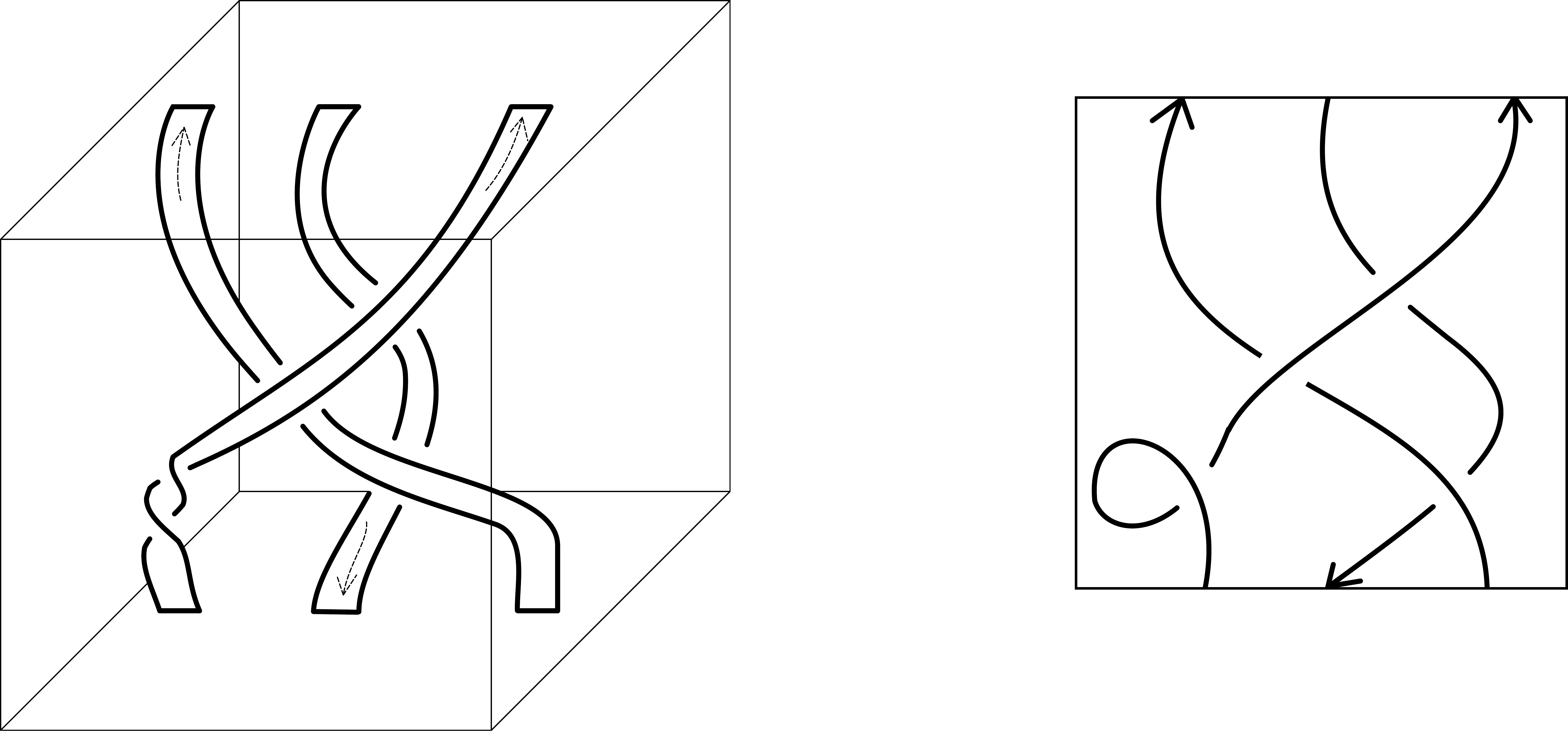}
\end{equation*}
\vspace{0.1cm}\noindent

%

It is well-known that the set of tangles can be organised into a strict monoidal category (in fact ribbon). Let $\mathrm{Mon}(+,-)$ be the free monoid on the set $\{+,- \}$. Given a tangle $L$, assign to every $F_i$ and $H_j$  the symbol $+$ or $-$  depending on whether $L$ points upwards or downwards, respectively. This assignment defines two elements $s(L), t(L) \in \mathrm{Mon}(+,-)$ of lengths $n_1 $ and $n_2$  called the \textit{source} and the \textit{target} of $L$.

The category $\mathcal{T}$ of tangles is defined to have objects $\mathrm{Mon}(+,-)$ and morphisms $\hom{\mathcal{T}}{s}{t}$ the set of tangles $L$ such that $s= s(L)$ and $t=t(L)$. The composite $L_2 \circ L_1$ of tangles $L_1, L_2$ is the tangle resulting from stacking $L_2$ on top of $L_1$, and the identity of a word $w \in \mathrm{Mon}(+,-)$  is the tangle $\uparrow_w$ given by a number of parallel, vertical strands with orientations determined by $s(\uparrow_w)=t(\uparrow_w)=w$. The monoidal product is given by concadenation of words at the level of the object, and at the level of morphisms $L_1 \otimes L_2$ is the tangle resulting from placing $L_2$ to the right of $L_1$ (and normalising the length of the cube). The unit object is the empty word, that is the unit of $\mathrm{Mon}(+,-)$.

It is a classical result (cf. \cite{ohtsukibook, habiro}) that the category $\mathcal{T}$ is monoidally generated by the objects $+,-$ and the morphisms $X, X^-, \cupr , \cupl ,  \capr , \capl $ shown below.


\begin{equation} 
\centre{
\labellist \small \hair 2pt
\pinlabel{$X$}  at 85 -78
\pinlabel{$X^-$}  at 460 -78
 \pinlabel{$ \cupr $}  at 753 -78
 \pinlabel{$ \cupl $}  at 1111 -78
  \pinlabel{$ \capr$}  at 1475 -78
 \pinlabel{$ \capl $}  at 1830 -78
\endlabellist
\centering
\includegraphics[width=0.6\textwidth]{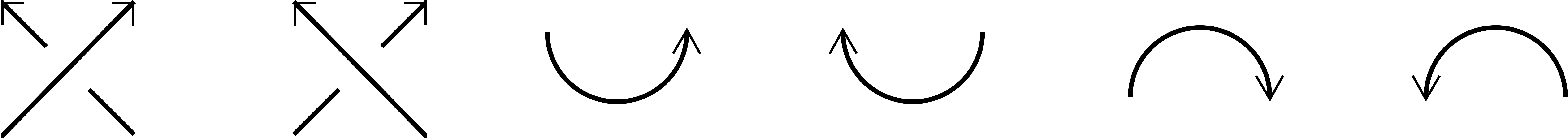}}
\end{equation}
\vspace{0.1cm}\noindent


In what follows, we will mostly restrict our attention to a certain subcategory of $\mathcal{T}$. The category of \textit{upwards tangles} is the monoidal subcategory  $\mathcal{T}^{\mathrm{up }} \subset \mathcal{T}$ on the objects $\mathrm{Mon}(+) \subset \mathrm{Mon}(+,-)$ and arrows open tangles. An upward tangle $K \in \hom{\mathcal{T}^{\mathrm{up }}}{+}{+}$ is called a \textit{long knot} or by simplicity a \textit{knot}. There is a canonical closure operation that establishes a well-know bijection between isotopy classes of  long knots (relative to its endpoints) and isotopy classes of closed knots, so we can do knot theory studying upwards tangles.


\subsection{Rotational tangles}
For our purposes it will be convenient to restrict our\-selves to a particular class of tangles that we introduce now. We say that a tangle diagram $D$ is in \textit{rotational form}  if
\begin{enumerate}[(1)]
\item as a tangle, $D$ is upwards.
\item all crossings in $D$ point upwards and all maxima and minima appear in pairs of the following two forms,
\begin{equation}\label{eq rot tangle}
\centre{
\centering
\includegraphics[width=0.3\textwidth]{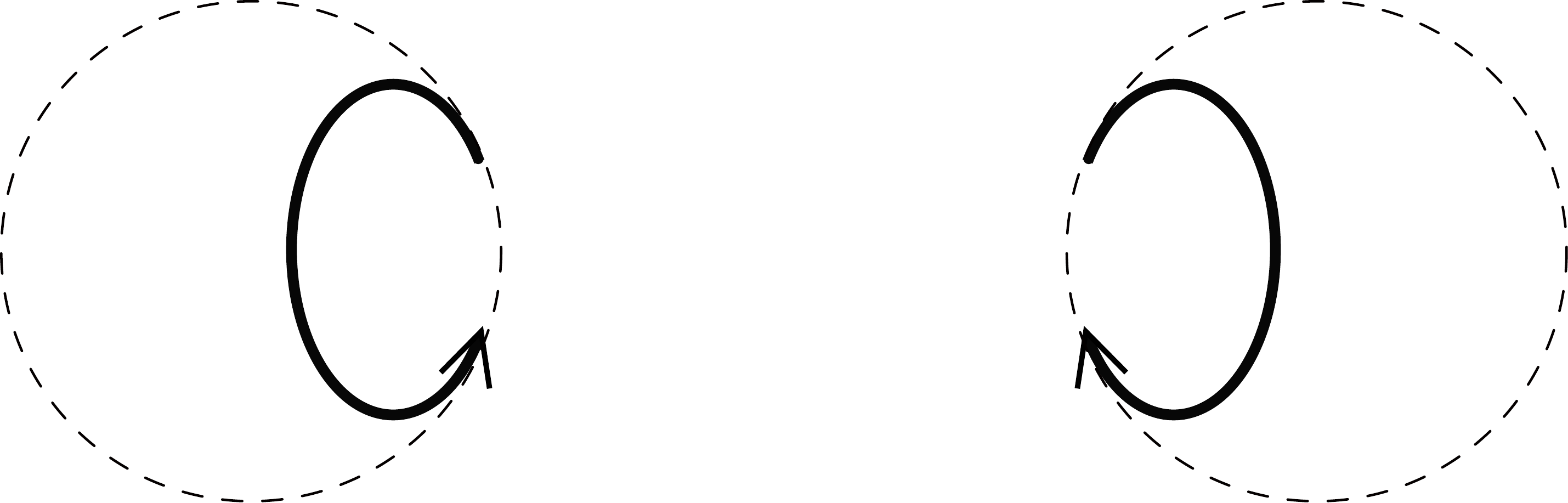}}
\end{equation}
where the dashed discs denote some neighbourhoods of that piece of strand in the tangle diagram.
\end{enumerate}


We regard rotational tangle diagrams up to \textit{Morse isotopy}, that is, planar isotopy that preserve all maxima and minima. We do not allow isolated cups and caps (``half rotations''), instead they must appear in pairs, either  $\capl$ and  $\cupr$ or  $\capr$ and  $\cupl$. This notion of tangle diagram where only ``full rotations'' are allowed was first introduced in \cite{barnatanveenpolytime} and \cite{barnatanveengaussians}.

The following lemma is contained in the computer implementation of the latter reference for knots, but we record here the same argument for tangles .

\begin{lemma}\label{lemma every tangle has rot diag}
Any upwards tangle has a diagram in rotational form.
\end{lemma}
\begin{proof}
Given a tangle diagram $D$ of an upwards tangle, we will construct another diagram in rotational form which is related to the former only by planar isotopy.

For each tangle component, label the edges of the strand according to the orientation. Let the pair $(i,j)$ denote the $j$-th edge of the strand $D_i$. Write $X_{(i,j), (i',j')}$ for the crossing that has the edge $(i,j)$ as the foot of the overstrand and $(i',j')$ as the foot of the understrand. According to the total order $(i,j)<(i',j) $ if $i<i'$ or $i=i'$ and $j<j'$, place the crossings in the bands $\R  \times [k,k+1]$ of the plane  in a upward fashion, placing a cup at the end of the foot of the edge with the greatest pair, see below:

\vspace{0.2cm}
\begin{equation*} 
\labellist \tiny \hair 2pt
\pinlabel{$(i,j)$}  at 44 -35
\pinlabel{$(i',j'+1)$}  at -15 195
\pinlabel{$(i,j+1)$}  at 230 195
\pinlabel{$(i',j')$}  at 345 65
\pinlabel{$(i,j)$}  at 655 -35
\pinlabel{$(i',j'+1)$}  at 620 195
\pinlabel{$(i,j+1)$} [l] at 749 195
\pinlabel{$(i',j')$}  at 960 65
\pinlabel{$(i,j)$}  at 1536 -35
\pinlabel{$(i',j'+1)$}  at 1630 195
\pinlabel{$(i,j+1)$} [r] at 1500 195
\pinlabel{$(i',j')$}  at 1230 65
\pinlabel{$(i,j)$}  at 2150 -35
\pinlabel{$(i',j'+1)$}  at 2250 195
\pinlabel{$(i,j+1)$} [r] at 2120 195
\pinlabel{$(i',j')$}  at 1850 65
\endlabellist
\centering
\includegraphics[width=0.8\textwidth]{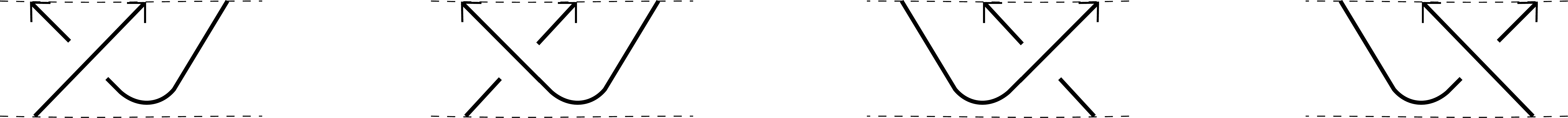}
\end{equation*}
\vspace*{1pt}

\noindent
For each $i$, we connect the edges $(i,j)$ according to the order and extend all edges up if they have not been connected yet.  In doing so, we must place some caps when we have to merge with an edge that is already on the diagram. The resulting diagram has cups and caps appearing in pairs as in  \eqref{eq rot tangle}, but also isolated maxima and zig-zag curves, as shown below to left and right respectively. However these two can be removed by a planar isotopy.

\begin{equation}
\centre{
\centering
\includegraphics[width=0.45\textwidth]{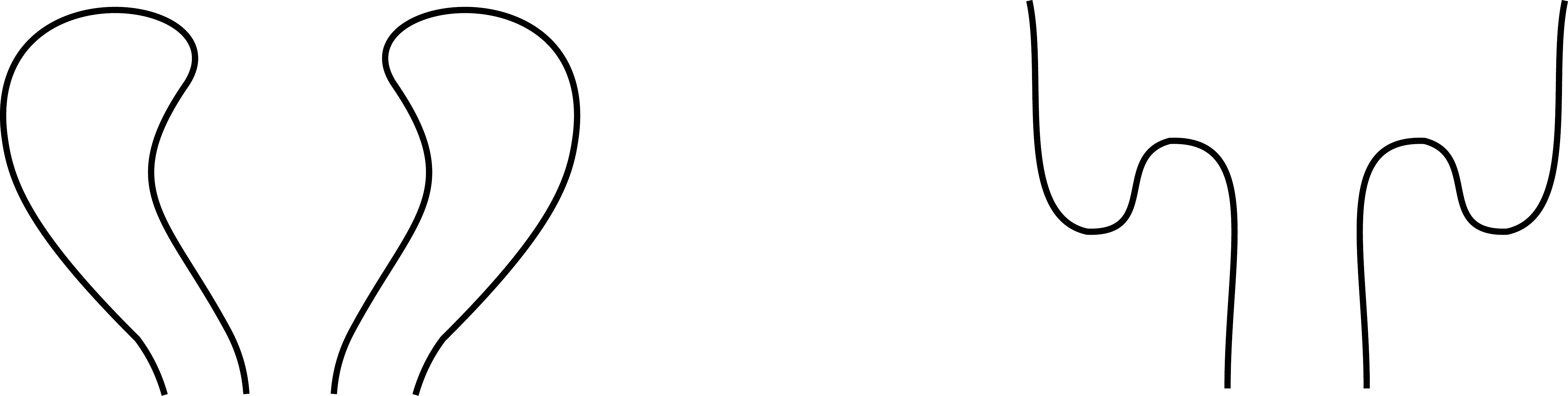}}
\end{equation}

\noindent By construction, the resulting diagram is planar isotopic to the original one.
\end{proof}

\begin{example}
Let us illustrate the proof of the  previous lemma. Suppose we start with the 2-component tangle showed below:
\begin{equation*}
\labellist \tiny \hair 2pt
\pinlabel{$1$}  at 45 50
\pinlabel{$2$}  at 195 321
\pinlabel{$2'$}  at 394 321
\pinlabel{$1'$}  at 538 50
\pinlabel{$3'$}  at 195 58
\pinlabel{$3$}  at 400 58
\pinlabel{$4'$}  at 55 313
\pinlabel{$4$}  at 545 313
\endlabellist
\centering
\includegraphics[width=0.2\textwidth]{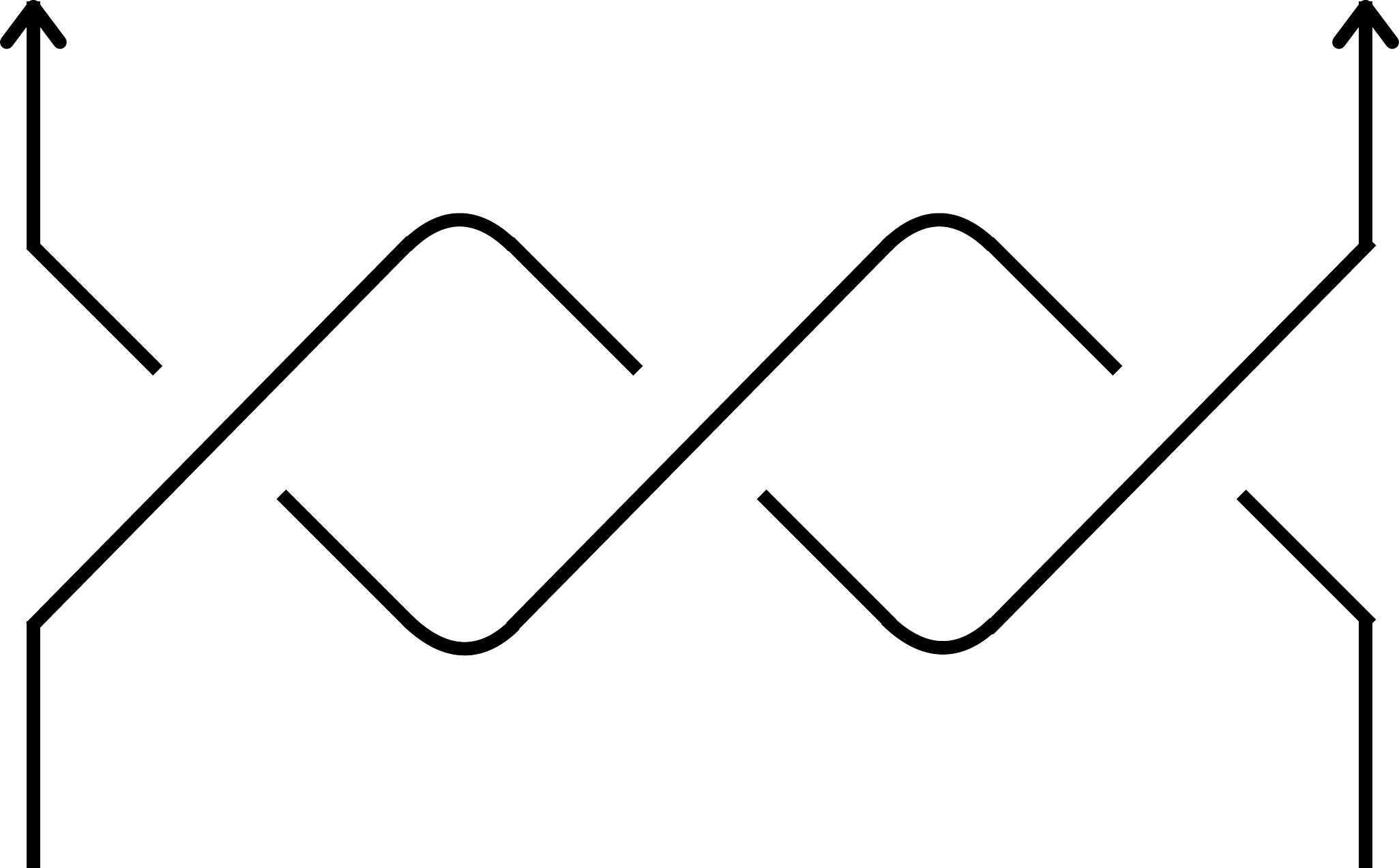}
\end{equation*}
For simplicity we  have labelled the edges  as $j=(1,j)$ and $j'=(2,j)$. The algorithm described produces the following tangle diagram:
\begin{equation*}
\labellist \tiny \hair 2pt
\pinlabel{$1$}  at  228 52
\pinlabel{$2$}  at 406 214
\pinlabel{$3$}  at 236 400
\pinlabel{$4$}  at 355 550
\pinlabel{$1'$}  at 718 64
\pinlabel{$2'$}  at 200 600
\pinlabel{$3'$}  at 555 290
\pinlabel{$4'$}  at 41 432
\endlabellist
\centering
\includegraphics[width=0.26\textwidth]{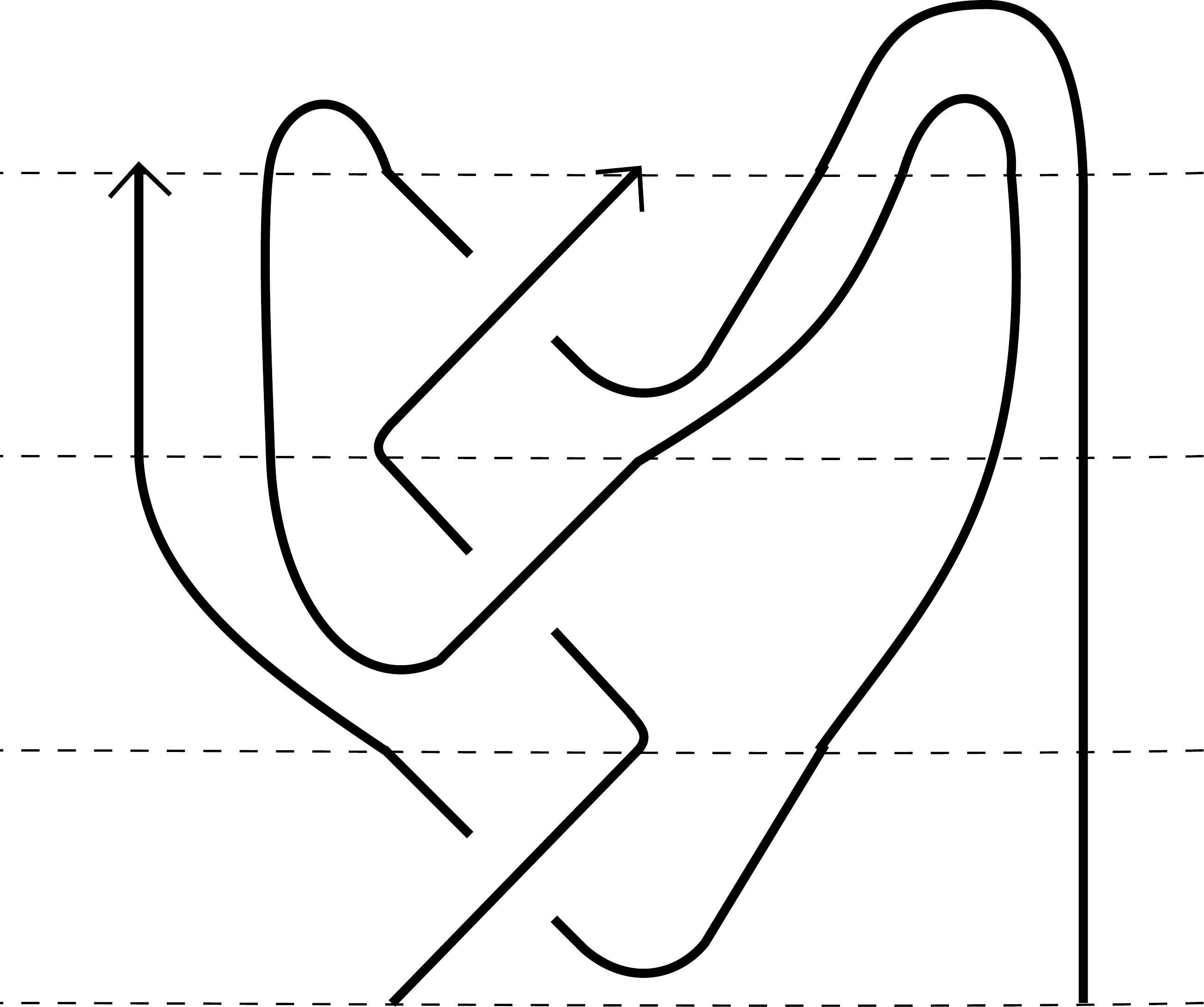}
\end{equation*}
Removing isolated maxima, we obtain the following rotational diagram of the original tangle:
\begin{equation*}
\centering
\includegraphics[width=0.2\textwidth]{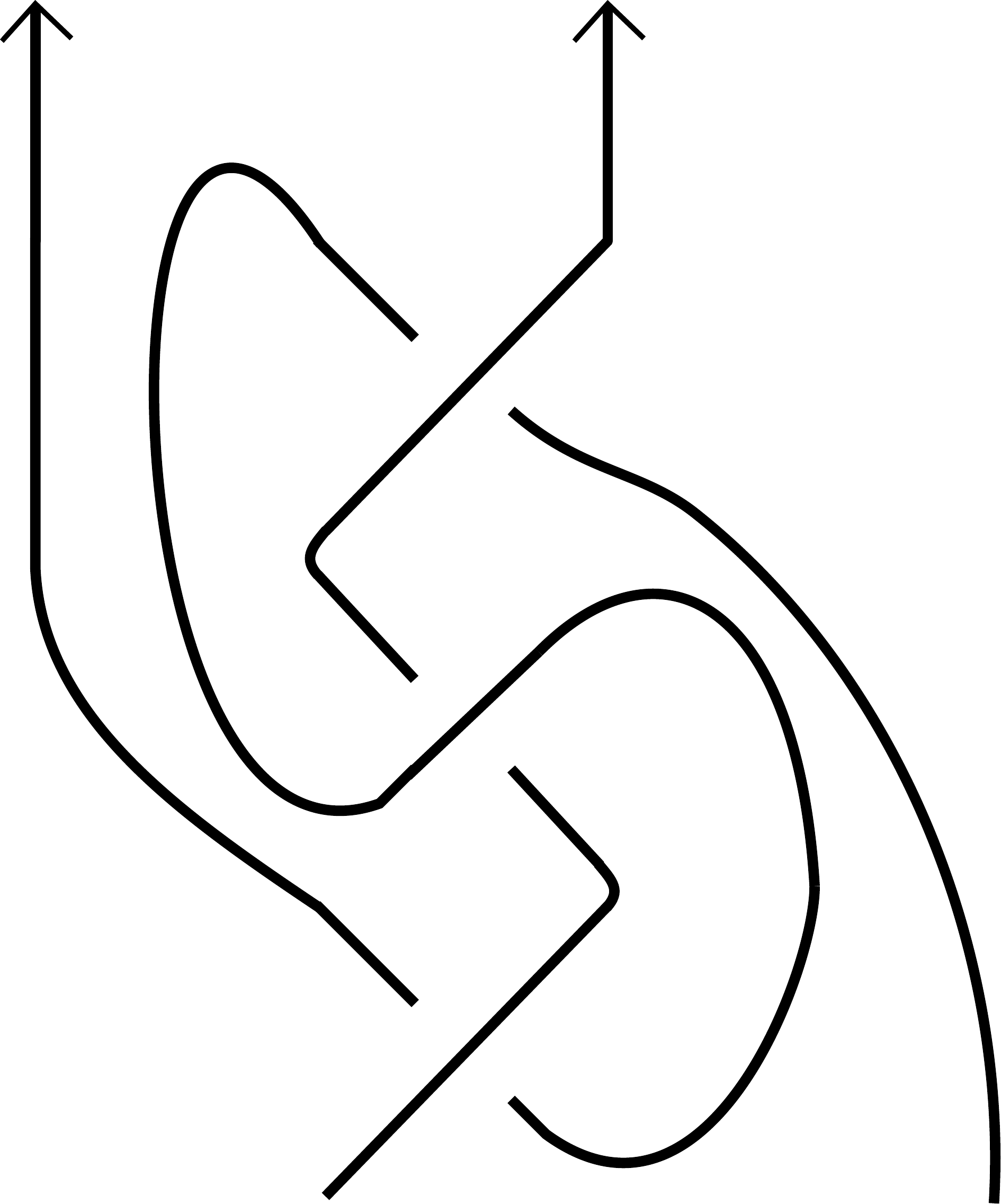}
\end{equation*}
\end{example}


We can hence restrict ourselves to study only tangle diagrams in rotational form that we more succinctly call \textit{rotational tangles}.  Mimicking \cite[Theorem 3.3]{ohtsukibook} and \cite{polyak10}, it is not difficult  to see that the following relations form a set of \textit{rotational Reidemeister moves} for rotational tangles:
\begin{equation}
\centre{
\labellist \small \hair 2pt
\pinlabel{$=$}  at 760 240
\pinlabel{$=$}  at 1830 240
\pinlabel{$=$}  at 2030 240
\pinlabel{$,$}  at 1260 240
\endlabellist
\centering
\includegraphics[width=0.8\textwidth]{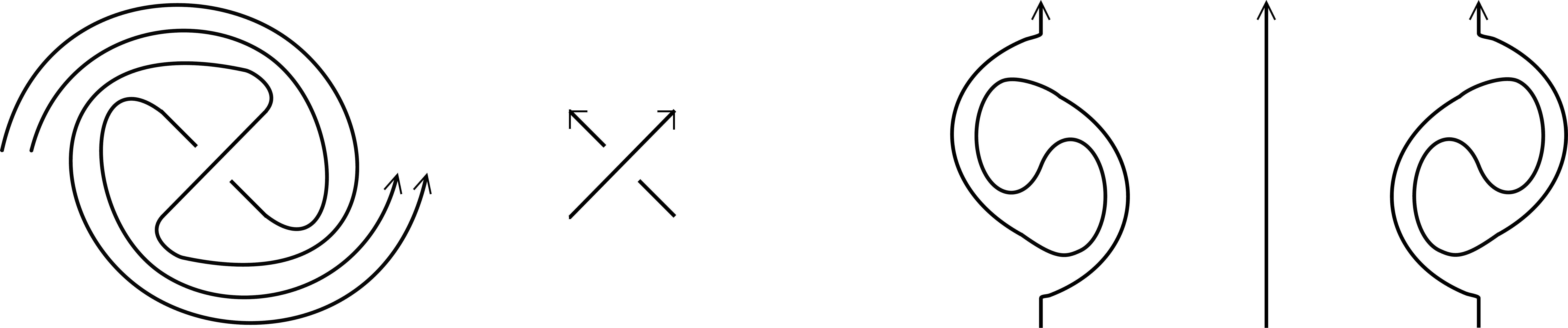}}
\end{equation}
\vspace*{10pt}
\begin{equation}
\centre{
\labellist \small \hair 2pt
\pinlabel{$=$}  at 452 168
\pinlabel{$=$}  at 1570 180
\pinlabel{$,$}  at 1089 168
\endlabellist
\centering
\includegraphics[width=0.65\textwidth]{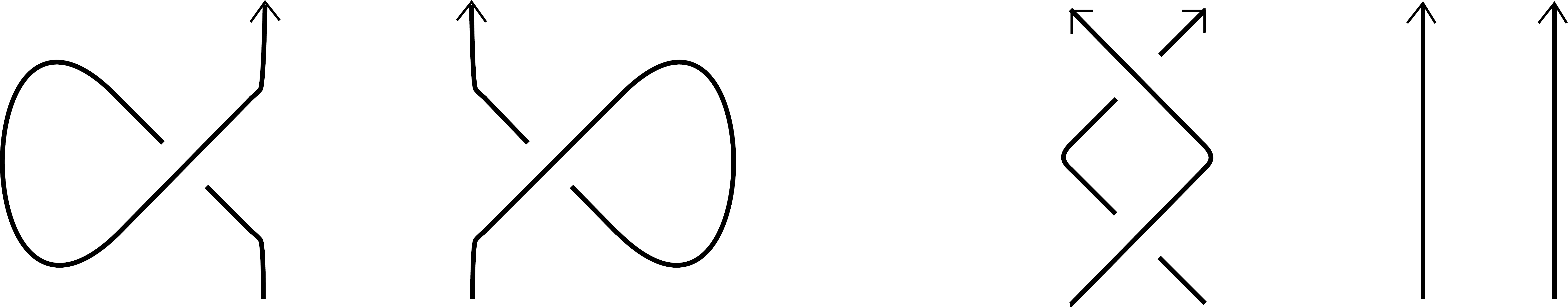}}
\end{equation}
\vspace*{10pt}
\begin{equation}
\centre{
\labellist \small \hair 2pt
\pinlabel{$=$}  at 360 260
\pinlabel{$=$}  at 1590 250
\pinlabel{$,$}  at 960 260
\endlabellist
\centering
\includegraphics[width=0.70\textwidth]{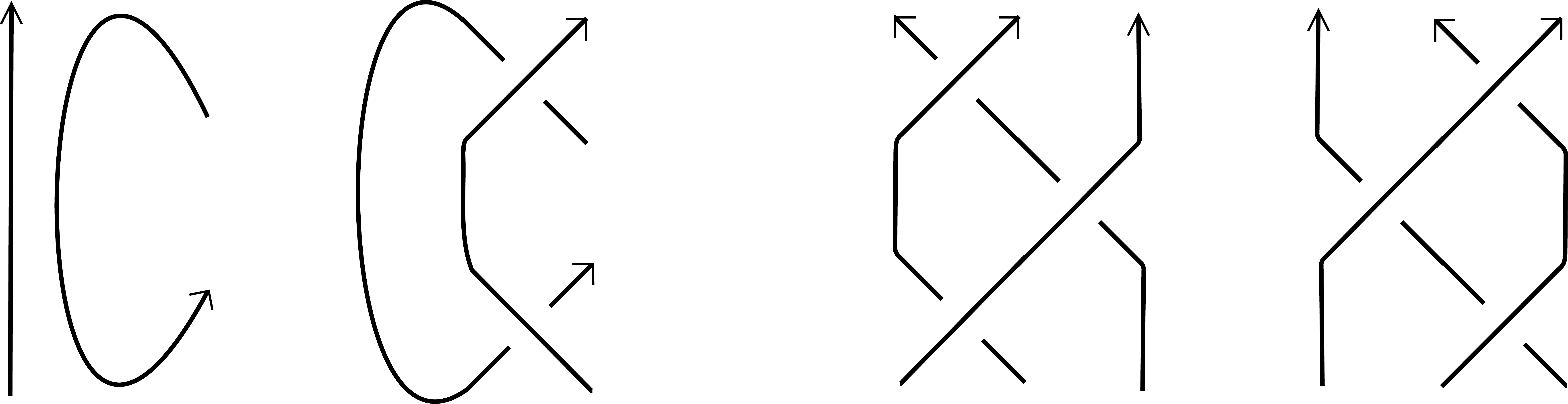}}
\end{equation}
\vspace*{5pt}

\noindent As usual, the above pictures are to be understood as identifying two tangle diagrams that are identical except in an open neighbourhood, where they look like as shown.

From this it follows

\begin{corollary}
There are bijections
\vspace*{5pt}
\begin{center}
\begin{tikzcd}[column sep=2em,row sep=2ex]
 \frac{ \left\{ \parbox[c][2.5em]{7em}{\centering
                      {\small  \textnormal{upwards tangles \\ in $(D^1)^{\times 3}$ }}} \right\} }{  \parbox[c][1.5em]{8em}{\centering  {\small  \textnormal{isotopy} }}} \arrow[equals]{r} &   \frac{ \left\{ \parbox[c][2.5em]{8em}{\centering
                      {\small  \textnormal{upwards  tangle diagrams in $(D^1)^{\times 2}$ }}} \right\} }{ \parbox[c][2.5em]{8em}{\centering  {\small \textnormal{  planar isotopy and Reidemeister moves }}}} \arrow[equals]{r} &   \frac{ \left\{ \parbox[c][2.5em]{8em}{\centering
                      {\small  \textnormal{rotational tangles \\ in $(D^1)^{\times 2}$ }}} \right\} }{ \parbox[c][3.8em]{8em}{\centering  {\small  \textnormal{Morse isotopy and rotational Reidemeister moves }}}}
\end{tikzcd}
\end{center}
\end{corollary}

Note that rotational tangles do not fit in $\mathcal{T}$, as the full rotations of \eqref{eq rot tangle} cannot be seen as morphisms;  yet any rotational tangle can be decomposed as the merging in the plane of the elementary building blocks depicted below: the single, unknotted strand (denoted by $I$), the positive and negative crossings (denoted $X$ and $X^-$), and the anticlockwise and clockwise full rotations (denoted $C$ and $C^-$). We will call the element $C$ the \textit{spinner}.

\begin{equation}\label{eq crossings and spinners}
\centre{
\labellist \small \hair 2pt
\pinlabel{$I$}  at 20 -78
\pinlabel{$X$}  at 350 -78
 \pinlabel{$ X^- $}  at 800 -78
 \pinlabel{$ C$}  at 1140 -78
  \pinlabel{$ C^- $}  at 1560 -78
\endlabellist
\centering
\includegraphics[width=0.6\textwidth]{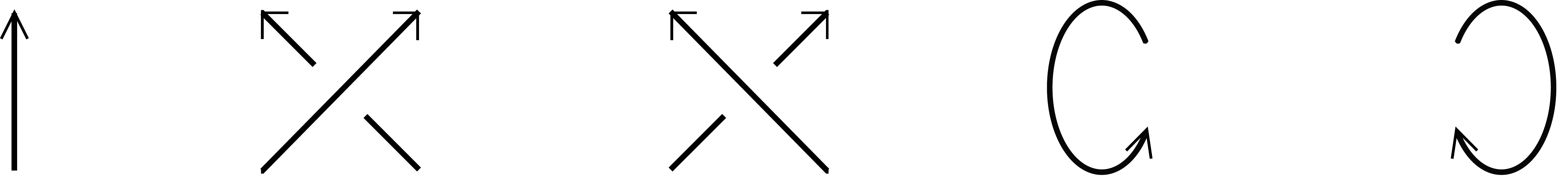}}
\end{equation}
\vspace{10pt}



\section{The universal invariant of tangles} \label{sec univ inv}

In this section, mostly to fix notation and conventions, we review the notions of topological ribbon Hopf algebra and the construction of the universal tangle invariant subject to one such, putting an emphasis on the description for rotational tangles. We also describe naturality and functoriality properties of the universal invariant and how it relates to classical quantum invariants that come from the representation  theory of quantum groups.

\subsection{Topological ribbon Hopf algebras}\label{section Topological ribbon Hopf algebras} We briefly recall the definitions of topological modules and topological ribbon Hopf algebras, highlighting the main features we will use later. For further details we refer the reader to \cite{kassel} or \cite{habiro}. 

Let $k$ be a commutative ring with unit, and consider the ring $k[[h]]$ of formal power  series in the variable $h$ with coefficients in $k$. As usual we topologise $k[[h]]$ with the (inverse) limit topology coming from the ring isomorphism $$k[[h]] \toiso \mathrm{lim}_n k[h]/(h^n),$$ which is known as the \textit{$h$-adic topology}.

Given a $k[[h]]$-module $M$, its \textit{completion} with respect to the $h$-adic topology is $\widehat{M}:=\mathrm{lim}_n M/h^nM$, which comes with the limit topology. When the canonical map $M \to \widehat{M}$ is an isomorphism, we say that $M$ is \textit{complete}, and $M$ is endowed with the $h$-adic topology via this isomorphism, so that $M$ can be seen as a topological $k[[h]]$-module. We will focus on the following class of $k[[h]]$-modules: if $V$ is a free $k$-module, the set $V[[h]]$ of formal power series with coefficients in $V$ is naturally a $k[[h]]$-module. It can be readily seen that $V[[h]]$ is complete, thus it can be realised as a topological $k[[h]]$-module. We call this class of modules \textit{topologically free}. Note that $$V[[h]] \cong \widehat{V\otimes_k k[[h]]}$$ as topological $k[[h]]$-modules.

Given two $k[[h]]$-modules $M$ and $N$, its \textit{topological tensor product} is defined as the completion of the usual tensor product, 
$$M \otimeshat N := \widehat{M \otimes_{k[[h]]} N}.$$ Let us denote by $\mathsf{Mod}_{k[[h]]}^{\mathsf{t}}$  the symmetric monoidal category of topological $k[[h]]$-modules and continuous $k[[h]]$-module maps, and write $\tau_{M,N}: M \otimeshat N \to N \otimeshat M$ for its symmetric braiding. We will also write $\mathsf{Mod}_{k[[h]]}^{\mathsf{tf}}$ for the full subcategory of $\mathsf{Mod}_{k[[h]]}^{\mathsf{t}}$  of topologically free $k[[h]]$-modules. Similarly, we put $\mathsf{Mod}_{k}^{\mathsf{f}}$ for the category of free $k$-modules and $k$-linear maps. For convenience we write $\mathrm{Hom}_{k}$ and $\mathrm{Hom}_{k[[h]]}$ for the hom-sets in the categories  $\mathsf{Mod}_{k}^{\mathsf{f}}$ and $\mathsf{Mod}_{k[[h]]}^{\mathsf{tf}}$, respectively.

The following proposition is a reformulation of well-known properties of topologically free modules (see \cite{kassel}).

\begin{proposition}\label{prop adjuction and top free}
The passage from $V$ to $V[[h]]$ gives rise to a free-forgetful adjuction
$$\begin{tikzcd}[column sep={4em,between origins}]
\mathsf{Mod}_{k}^{\mathsf{f}}
  \arrow[rr, bend left, swap, "{-[[h]]}"' pos=0.55]
& \bot &  \mathsf{Mod}_{k[[h]]}^{\mathsf{tf}}
  \arrow[ll, bend left, swap, "U"' pos=0.45]
\end{tikzcd}$$
Besides, the completed tensor product restricts to topologically free modules via the isomorphism
\begin{equation}
V[[h]] \otimeshat W[[h]] \toiso (V \otimes_{k} W)[[h]],\label{eq top free tensor}
\end{equation}
which makes $\mathsf{Mod}_{k[[h]]}^{\mathsf{tf}}$ a symmetric monoidal category and the functor $-[[h]]$ strong monoidal and faithful (and essentially surjective).

Furthermore, $\mathsf{Mod}_{k[[h]]}^{\mathsf{tf}}$ is enriched\footnote{For a precise definition of  an enriched category, see \cite{riehl2014} } over itself via the isomorphism
\begin{equation}
\hom{k[[h]]}{V[[h]]}{W[[h]]} \toiso \hom{k}{V}{W}[[h]]
\end{equation}
induced by the adjunction, which turns it into a closed symmetric monoidal category, meaning that there are natural isomorphisms
\begin{align}
\begin{split}
\hom{k[[h]]}{M}{\hom{k[[h]]}{N}{P}} &\overset{\cong}{\longleftarrow} \hom{k[[h]]}{M \otimeshat N}{P}\\
&\toiso \hom{k[[h]]}{N}{\hom{k[[h]]}{M}{P}}
\end{split}
\end{align}
where $N,M,P$ are topologically free. 

Lastly, the forgetful functor $U: \mathsf{Mod}_{k[[h]]}^{\mathsf{tf}} \to \mathsf{Set}$ is represented by the unit object $k[[h]]$, that is, there is a natural isomorphism
\begin{equation}\label{eq forgetful represented}
\hom{k[[h]]}{k[[h]]}{V[[h]]} \toiso U(V[[h]])
\end{equation}
\end{proposition}

%
%
%
%
%

A \textit{topological Hopf algebra} is a topological $k[[h]]$-module $A$ endowed with an algebra structure $(A,  \mu, \eta)$, a coalgebra structure $(A,  \Delta, \epsilon)$ (compatible in the sense that the coalgebra structure maps are algebra morphisms), and an isomorphism $S: A \to A$, called the \textit{antipode}, which is the inverse of $\mathrm{Id}_A$ in the convolution monoid $\hom{k[[h]]}{A}{A}$, that is $$ m(S \otimeshat \id)\Delta = \eta \epsilon = m(\id \otimeshat S)\Delta.$$ All structure maps in the definition are required to be continuous (that is, they all must be morphisms in $\mathsf{Mod}_{k[[h]]}^{\mathsf{t}}$).

A \textit{topological quasi-triangular Hopf algebra} is a topological Hopf algebra $A$ together with an invertible element $R \in A \otimeshat A$, called the \textit{universal R-matrix}, satisfying the following properties:
\begin{align}
(\Delta \otimeshat \id)R &= R_{13} \cdot R_{23} \label{eq Rmatrix_axiom1}\\
(\id \otimeshat \Delta)R &= R_{13}\cdot  R_{12} \label{eq Rmatrix_axiom2}\\
\tau_{A,A} \Delta  &=   R \cdot \Delta(-) \cdot  R^{-1} \label{eq Rmatrix_axiom3}
\end{align}
where $R_{12} := R\otimeshat 1 $, $R_{13}:= (\id \otimeshat \tau_{A,A})R_{12}$ and $R_{23}:= 1 \otimeshat R$.

We will write $R= \sum_i \alpha_i \otimeshat \beta_i$ for the universal $R$-matrix and $R^{-1}= \sum_i \overline{\alpha}_i \otimeshat \overline{\beta}_i$ for its inverse. Note that in general these sums will be infinite.

A \textit{topological ribbon Hopf algebra} is a topological quasi-triangular Hopf algebra $A$ together with a preferred element $\kappa \in A$, called the \textit{balancing element}  satisfying 
\begin{align}
\Delta \kappa &= \kappa \otimeshat \kappa \label{eq balancing_axiom1}\\
\epsilon (\kappa) &= 1 \label{eq balancing_axiom2} \\
\kappa^2 &= u \cdot S(u^{-1}) \label{eq balancing_axiom3}\\
S^2 &= \kappa \cdot (-) \cdot \kappa^{-1}\label{eq balancing_axiom4}
\end{align}
where $u:= \mu(S \otimeshat \id_A)R_{21} = \sum_i S(\beta_i) \cdot \alpha_i$ is called the \textit{Drinfeld element} and $u^{-1} := \mu (\id_A \otimeshat S^2)R_{21} = \sum_i \beta_i \cdot S^2 (\alpha_i) $ is its inverse. Here we wrote $R_{21}= \tau_{A,A} (R)$.

The element $v :=  \kappa^{-1} \cdot u$ is classically called the \textit{ribbon element} (hence the name of the algebraic structure). Note that the ribbon is central: for any $x \in A$ we have $vx= \kappa^{-1}ux= \kappa^{-1} S^{2}(x)u=x\kappa^{-1}u=xv$. It can be shown \cite{kauffman_radford} that the set of axioms \eqref{eq balancing_axiom1} -- \eqref{eq balancing_axiom4} for the balancing element are equivalent to the usual set of axioms
$$ v \in \mathcal{Z}(A) \ \ , \ \ v^2 = u \cdot S(u)  \ \ , \ \  \Delta(v)=(R_{21}R)^{-1}(v \otimeshat v)  \ \ , \ \  \epsilon(v)=1  \ \ , \ \  S(v)=v$$ for the ribbon element, where $\mathcal{Z}(A)$ denotes the centre of $A$.

\begin{example}\label{ex quantisation}
For any semisimple complex Lie algebra $\mathfrak{g}$, Drinfeld \cite{drinfeld85} and Jimbo \cite{jimbo85} constructed a topological ribbon Hopf algebra $U_h(\mathfrak{g})$ over $\mathbb{C} [[h]]$ which is a ``quantisation'' of the universal enveloping algebra $U(\mathfrak{g})$ of $\mathfrak{g}$, meaning that $U_h(\mathfrak{g})/ hU_h(\mathfrak{g}) \cong U(\mathfrak{g})$ as bialgebras and $U_h(\mathfrak{g}) \cong U(\mathfrak{g})[[h]]$ as $\mathbb{C}[[h]]$-modules.
\end{example}

In \cref{subsec algebra D} we will introduce the central example of topological ribbon Hopf algebra that we will study.

\subsection{The universal tangle invariant}\label{section The universal tangle invariant}

We are now ready to define the universal tangle invariant, introduced for the first time in \cite{lawrence} (for a brief historical note see \cite[\S 1.2.1]{habiro}).  Let $A$ be a topological ribbon Hopf algebra with unit $1 \in A$, universal $R$-matrix $R= \sum_i \alpha_i \otimeshat \beta_i$, inverse $R^{-1} = \sum_i \bar{\alpha}_i \otimeshat \bar{\beta}_i$ and balancing element $\kappa \in A$. 

For simplicity we will restrict to the case of open tangles (otherwise one has to use a quotient of the ribbon Hopf algebra for the close components, but we will not do that here), and besides we will always assume that the tangle components are ordered.  Given an $n$-component open tangle  $L=L_1 \cup  \cdots \cup L_n$, place beads representing the elements $1, (S^i \otimes S^j)(R^{\pm 1}) , \kappa^{\pm 1} $, $i,j =0,1$, in the lowest part of the trivial tangle, crossings and cups and caps, putting the  ``alpha'' always in the overstrand and the ``beta'' in the understrand,  as depicted below:

\begin{align}\label{eq beads norot}
\begin{split}
\centre{
\labellist \footnotesize \hair 2pt
\pinlabel{$ \color{violet} \bullet$} at 16 30
\pinlabel{$ \color{violet} 1$} at -26 30
\pinlabel{$ \color{violet} \bullet$} at 278 30
\pinlabel{$ \color{violet} \bullet$} at 381 30
\pinlabel{$ \color{violet} \alpha_i$} [r] at 263 30
\pinlabel{$ \color{violet} \beta_i$} [l] at 396 30
\pinlabel{$ \color{violet} \bullet$} at 671 30
\pinlabel{$ \color{violet} \bullet$} at 780 30
\pinlabel{$ \color{violet} \alpha_i$} [r] at 656 30
\pinlabel{$ \color{violet} S \beta_i$} [l] at 795 30
\pinlabel{$ \color{violet} \bullet$} at 1093 30
\pinlabel{$ \color{violet} \bullet$} at 1200 30
\pinlabel{$ \color{violet} S\alpha_i$} [r] at 1078 30
\pinlabel{$ \color{violet}  \beta_i$} [l] at 1215 30
\pinlabel{$ \color{violet} \bullet$} at 1483 30
\pinlabel{$ \color{violet} \bullet$} at 1589 30
\pinlabel{$ \color{violet} S\alpha_i$} [r] at 1468 30
\pinlabel{$ \color{violet} S \beta_i$} [l] at 1604 30
\pinlabel{$ \color{violet} \bullet$} at 1943 45
\pinlabel{$ \color{violet} \kappa$} [t] at 1943 30
\pinlabel{$ \color{violet} \bullet$} at 2349 135
\pinlabel{$ \color{violet} 1$} [b] at 2349 150
\endlabellist
\centering
\includegraphics[width=0.9\textwidth]{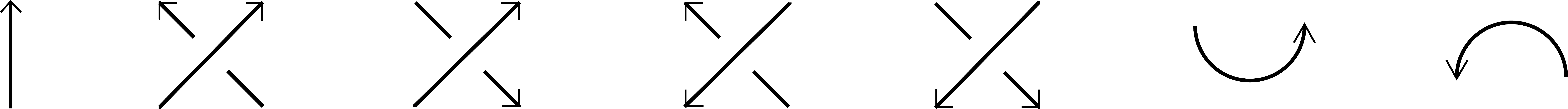}}
\\
\centre{
\labellist \footnotesize \hair 2pt
\pinlabel{$ \color{violet} \bullet$} at 16 30
\pinlabel{$ \color{violet} 1$} at -26 30
\pinlabel{$ \color{violet} \bullet$} at 278 30
\pinlabel{$ \color{violet} \bullet$} at 381 30
\pinlabel{$ \color{violet} \bar{\beta}_i$} [r] at 263 30
\pinlabel{$ \color{violet} \bar{\alpha}_i$} [l] at 396 30
\pinlabel{$ \color{violet} \bullet$} at 671 30
\pinlabel{$ \color{violet} \bullet$} at 780 30
\pinlabel{$ \color{violet} \bar{\beta}_i$} [r] at 656 30
\pinlabel{$ \color{violet} S \bar{\alpha}_i$} [l] at 795 30
\pinlabel{$ \color{violet} \bullet$} at 1093 30
\pinlabel{$ \color{violet} \bullet$} at 1200 30
\pinlabel{$ \color{violet} S\bar{\beta}_i$} [r] at 1078 30
\pinlabel{$ \color{violet}  \bar{\alpha}_i$} [l] at 1215 30
\pinlabel{$ \color{violet} \bullet$} at 1483 30
\pinlabel{$ \color{violet} \bullet$} at 1589 30
\pinlabel{$ \color{violet} S\bar{\beta}_i$} [r] at 1468 30
\pinlabel{$ \color{violet} S \bar{\alpha}_i$} [l] at 1604 30
\pinlabel{$ \color{violet} \bullet$} at 1943 132
\pinlabel{$ \color{violet} \kappa^{-1}$} [b] at 1973 147
\pinlabel{$ \color{violet} \bullet$} at 2349 47
\pinlabel{$ \color{violet} 1$} [t] at 2349 32
\endlabellist
\centering
\includegraphics[width=0.9\textwidth]{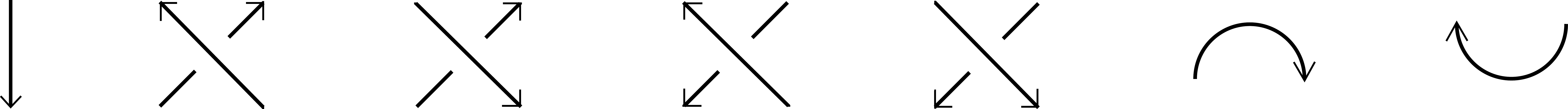}}
\end{split}
\end{align}

The \textit{universal invariant} of $L$ subject to the ribbon  Hopf algebra $A$   is the element  $Z_A(L) \in A \otimeshat \overset{n}{\cdots} \otimeshat A$   defined as follows: for $1 \leq i \leq n$ let $Z_A(L)_{(i)}$ be the (formal) word given by writing from left to right the labels of the beads in the $i$-th component according to the  orientation of the strand. Then put
$$Z_A(L) := \sum Z_A(L)_{(1)} \otimeshat \cdots \otimeshat Z_A(L)_{(n)} \in A \otimeshat \overset{n}{\cdots} \otimeshat A$$ where the summation runs through all subindices in $R^{\pm 1}$ (one for each crossing).  

If $L$ is an upwards tangle, then the description of the universal invariant is particularly simple, as we only have to specify the values for the five (instead of fourteen) building blocks from \eqref{eq crossings and spinners}, once it is in rotational form. It follows from above that the assignments are:

\begin{equation}\label{eq beads rot}
\centre{
\labellist \small \hair 2pt
\pinlabel{$ \color{violet} \bullet$} at 16 30
\pinlabel{$ \color{violet} 1$} at -26 30
\pinlabel{$ \color{violet} \bullet$} at 294 30
\pinlabel{$ \color{violet} \bullet$} at 406 30
\pinlabel{$ \color{violet} \alpha_i$} [r] at 279 30
\pinlabel{$ \color{violet} \beta_i$} [l] at 421 30
\pinlabel{$ \color{violet} \bullet$} at 711 30
\pinlabel{$ \color{violet} \bullet$} at 825 30
\pinlabel{$ \color{violet} \bar{\beta}_i$} [r] at 696 30
\pinlabel{$ \color{violet} \bar{\alpha}_i$} [l] at 850 30
\pinlabel{$ \color{violet} \bullet$} at 1083 92
\pinlabel{$ \color{violet} \kappa$} [r] at 1068 92
\pinlabel{$ \color{violet} \bullet$} at 1599 92
\pinlabel{$ \color{violet} \kappa^{-1}$} [l] at 1614 104
\endlabellist
\centering
\includegraphics[width=0.6\textwidth]{sketch_figures/building_blockss}}
\end{equation}
\vspace{0.01cm}

It is shown in \cite{ohtsukibook} that $Z_A(L)$ is preserved under Reidemeister moves and hence defines an isotopy invariant for tangles.

\begin{example}
Let $L$ be the 2-component  upwards tangle below, that we have already decorated  according to \eqref{eq beads rot}:
\vspace*{5pt}
\begin{equation*}
\centre{
\labellist \footnotesize \hair 2pt
\pinlabel{{\normalsize $L$}} at -134 854
\pinlabel{$ \color{red} \bullet$} at 262 172
\pinlabel{$ \color{red} \bullet$} at 356 172
\pinlabel{$ \color{red} \alpha_i$} [r] at 247 172
\pinlabel{$ \color{red} \beta_i$} [l] at 371 172
\pinlabel{$ \color{OliveGreen} \bullet$} at 262 363
\pinlabel{$ \color{OliveGreen} \bullet$} at 356 363
\pinlabel{$ \color{OliveGreen} \alpha_j$} [r] at 247 363
\pinlabel{$ \color{OliveGreen} \beta_j$} [l] at 371 363
\pinlabel{$ \color{magenta} \bullet$} at 428 535
\pinlabel{$ \color{magenta} \bullet$} at 542 535
\pinlabel{$ \color{magenta} \bar{\beta}_\ell$} [r] at 413 535
\pinlabel{$ \color{magenta} \bar{\alpha}_\ell$} [l] at 557 535
\pinlabel{$ \color{blue} \bullet$} at 250 712
\pinlabel{$ \color{blue} \bullet$} at 364 712
\pinlabel{$ \color{blue}  \bar{\beta}_s$} [r] at 235 712
\pinlabel{$ \color{blue} \bar{\alpha}_s$} [l] at 379 712
\pinlabel{$ \color{orange} \bullet$} at 3 498
\pinlabel{$ \color{orange} \kappa$} [r] at -13 498
\endlabellist
\centering
\includegraphics[width=0.2\textwidth]{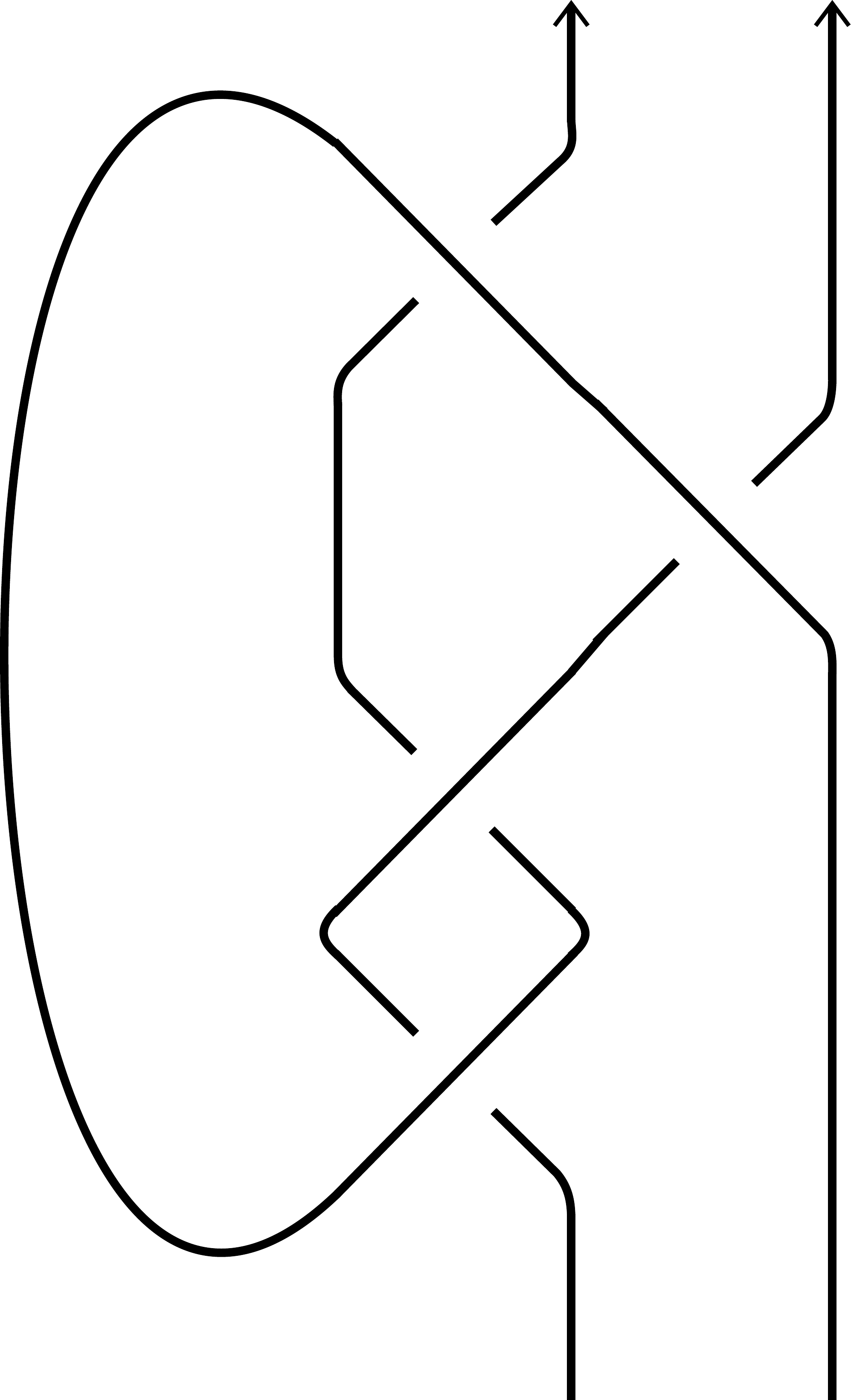}}
\end{equation*}
Its universal invariant is given by
$$Z_A(L)= \sum_{i,j,\ell,s} \beta_i \alpha_j \bar{\beta}_\ell \otimeshat \bar{\alpha}_\ell \bar{\alpha}_s \kappa \alpha_i \beta_j \bar{\beta}_s \in A \otimeshat A.$$
\end{example}

\begin{warning}
We ask the reader to be mindful of the different conventions and notations that exist in the literature about the universal invariant, which is also denoted as $J_L$ or $Q^{A;*}$. 

Our convention is closest to Ohtsuki's  \cite{ohtsukibook}, just up to turning the tangle upside down, since his tangles are preferably  oriented downwards. That is, for a tangle $L$ the value $Z_A(L)$ is Ohtsuki's value of the universal invariant of the tangle $L'$ with is the diagram of $L$ rotated 180 degrees on the plane.

Habiro \cite{habiro} also considers the downwards orientation as the preferred one, and he reads the beads following the inverse orientation of the strand. The value $Z_A(L)$ equals Habiro's value of the universal invariant  of $L'' \subset (D^1)^{\times 3}$ obtained by rotating $L$ 180 degrees about the $z$ axis and reverse the orientation of all strands.

Lastly, a fourth convention is used in \cite{BBG} where tangles are oriented upwards but the universal invariant is written following the inverse orientation of the strand.  The value $Z_A(L)$ equals the authors' value of the universal invariant of $L'''$ obtained by placing $L$ upside down (rotation of 180 degrees on the plane) and reversing the orientation of all strands.

We conclude that our convention and the three presented above are essentially equivalent up to the minor differences between the tangles $L,L', L'', L'''$.
\end{warning}

\subsection{Naturality of the universal invariant}

There are tangle operations which are natural  with respect to the Hopf algebra structure maps. Following \cite{habiro}, given an open tangle $L= L_1 \cup \cdots \cup L_n$ and $1 \leq i \leq n$, define $$S_{L_i}: A  \to A \qquad , \qquad  S_{L_i}(x) := \kappa^{h(i)} S(x) \kappa^{f(i)}$$ where $$ h(i):= \begin{cases} 0, & \text{if the head of $L_i$ points up}\\ -1, & \text{if the head of $L_i$ points down}  \end{cases}  $$ and $$ f(i):= \begin{cases} 0, & \text{if the foot of $L_i$ points up}\\ 1, & \text{if the foot of $L_i$ points down}  \end{cases} . $$

Likewise let $L^1, L^2, L^3,L^4$ be open tangles such that $s(L^1)=t(L^2)=w-+w'$ and  $s(L^3)=t(L^4)=w+-w'$ for some $w, w' \in \mathrm{Mon}(+,-)$, and assume that the tangle component at the first $\pm$ between $w$ and $w'$ is the $i$-th and the one at the second $\pm$ is the $(i+1)$-th (hence different).  Define morphisms $A \otimeshat A \to A$ as 
$$\begin{aligned}[t]
\mu_{L_i^1}^{\cupr} (x \otimes y) &:= x \kappa y\\
\mu_{L_i^3}^{\cupl} (x \otimes y) &:= y \kappa^{-1}x
\end{aligned} \qquad \qquad 
\begin{aligned}[t]
\mu_{L_i^2}^{\capl} (x \otimes y) &:= yx\\
\mu_{L_i^4}^{\capr}(x\otimes y)  &:=xy
\end{aligned}$$
extended linearly. Similarly define 
$$\begin{aligned}[t]
 \check{\mu}_{i}^{\cupr} (L^1) &:= L^1 \circ (\uparrow_w \otimes \cupr \otimes \uparrow_{w'} )\\
 \check{\mu}_{i}^{\cupl} (L^3) &:= L^3 \circ (\uparrow_w \otimes \cupl \otimes \uparrow_{w'} )
\end{aligned} \qquad \qquad 
\begin{aligned}[t]
 \check{\mu}_{i}^{\capl} (L^2) &:=   (\uparrow_w \otimes \capl \otimes \uparrow_{w'} ) \circ L^2\\
\check{\mu}_{i}^{\capr} (L^4) &:=   (\uparrow_w \otimes \capr \otimes \uparrow_{w'} ) \circ L^4
\end{aligned}$$

\begin{notation}
Given a linear map $g: A^{\otimeshat p} \to  A^{\otimeshat q}$, $n \geq 1$ and $1 \leq i \leq n$, write $$g_{(n,i)}=g_{(i)} :=\id^{\otimeshat (i-1)} \otimeshat g \otimeshat \id^{\otimeshat (n-i)}: A^{\otimeshat p+n-1} \to  A^{\otimeshat q+n-1}.$$
\end{notation}

The following result is standard. Here we have adapted to our setup.

\begin{proposition}\label{prop naturality Z_A}
Let $L= L_1 \cup \cdots \cup L_n$ be an $n$-component open tangle, and  let $L^1, L^2, L^3,L^4$ as above also have $n$ components. Let $L'$ be another open tangle. For $1 \leq i \leq n$, let $\check{\Delta}_i (L)$ (resp. $\check{\varepsilon}_i (L), \check{S}_i(L)$) be the tangle obtained from $L$ by doubling\footnote{With the convention that, if the $i$-th component runs from south to north, then the new $i$-th component lies to the east, the new $(i+1)$-th component lies to the west and the order for the rest of the components is shifted accordingly.} (resp. removing, orientation-reversing) the $i$-th component $L_i$.  Then 
\begin{align*}
Z_A ( \check{\mu}_{i}^{\bullet} (L^j))&= (\mu_{L_i^j}^{\bullet})_{(i)}(Z_A(L^j))\\
Z_A (L \otimes L' ) &= Z_A(L) \otimes Z_A(L')\\
Z_A (\check{\Delta}_i (L)) &= \Delta_{(i)}(Z_A(L))\\
Z_A (\check{\varepsilon}_i (L)) &= \varepsilon_{(i)}(Z_A(L))\\
Z_A (\check{S}_i (L)) &= (S_{L_i})_{(i)}(Z_A(L))
\end{align*}
where $\bullet = \cupr, \capl,  \cupl, \capr$ if $j=1,2,3,4$, respectively.
\end{proposition}



\subsection{Functoriality of the universal invariant}

If we restrict our attention to upwards tangles, then the universal tangle invariant can be viewed as a strong monoidal functor, as follows: given a (topological) ribbon Hopf algebra $A \in \mathsf{Mod}_{k[[h]]}$, let $\mathsf{A}$ denote the category whose objects are non-negative integers and 
$$\hom{\mathsf{A}}{n}{m} = \begin{cases}
    \emptyset, & n \neq m, \\ U(A^{\otimeshat n} \times \Sigma_n) , & n=m,
\end{cases}$$
where $U: \mathsf{Mod}_{k[[h]]} \times \mathsf{Grp} \to \mathsf{Set}$ denotes the canonical forgetful functor. The composite law is defined as follows: if $f=(u, \sigma)$ and $g= (v, \tau)$ are arrows of $\mathsf{A}$, then $$g \circ f := ( u \cdot \sigma^{-1}_* v \  , \  \tau \circ \sigma),$$
where $\sigma^{-1}_*: A^{\otimeshat n} \to  A^{\otimeshat n}$ is the map that permutes the factors induced by $\sigma^{-1}$ and the symmetric braiding of $\mathsf{Mod}_{k[[h]]}$, and $ \tau \circ \sigma$ denotes the composite of bijections of elements of the symmetric group. The associativity of this composite law easily follows from the fact that $\sigma^{-1}_*$ is an algebra map and $(\tau \circ \sigma)_* = \tau_* \circ \sigma_*$.  The identity of the object $n$ is the pair $(1 \otimeshat \cdots \otimeshat 1 \ , \ \id)$. This category is strict monoidal with respect to the addition of integers and the product $$f \otimes f' := (u \otimeshat u', \sigma \otimes \sigma' )$$
with $f' = (u', \sigma ') \in \hom{\mathsf{A}}{m}{m}$ and $\sigma \otimes \sigma'\in \Sigma_{n+m}$ the block permutation. The unit for the monoidal structure is the integer zero.

\begin{proposition}\label{prop functoriality}
Let $A$ be a (topological) ribbon Hopf algebra and let $\mathsf{A}$ be the strict monoidal category induced by $A$ as before. Then the universal tangle invariant induces a strong monoidal functor
$$Z_A : \mathcal{T}^{\mathrm{up}} \to \mathsf{A}.$$
\end{proposition}
\begin{proof}
Given  an upwards $n$-component tangle $L$, recall that we always assume that the components are ordered in increasing order according to the source $+^n \in \mathcal{T}$ of $L$. Write $ \sigma_L \in \Sigma_n$ for the permutation induced by $L$, that is $\sigma (i)$ is the label of the component that hits the target $+^n \in \mathcal{T}$ of $L$ at the $i$-th position. 

    The functor is defined as follows: the object $+^n$ is mapped to $n \geq 0$, and an upwards tangle $L$ as before is mapped to the pair $(Z_A (L), \sigma_L)$. That this passage is functorial and strong monoidal follows directly from the construction of $\mathsf{A}$. 
\end{proof}

\begin{remark}
Observe that, when studying the universal invariant, it is much more natural to consider rotational tangles than general tangles in a cube. Indeed the universal $R$-matrix and its inverse $R^{\pm 1}$ as well as the balancing element and its inverse $\kappa^{\pm 1}$ have geometric counterparts $X^{\pm}$ and $C^{\pm}$, and there are no other building blocks to consider. Furthermore, if one relaxes the notion of rotational tangle (roughly, allowing virtual crossings) then one can view the universal invariant as a certain monoidal natural transformation, which generalises both \cref{prop naturality Z_A} and \cref{prop functoriality}. This will be the topic of  a forthcoming paper \cite{becerra}.
\end{remark}

\subsection{Relation with Reshetikhin-Turaev invariants} \label{subsec relation with classical quantum invariants}

It is well-known that the universal invariant $Z_A$ of a  ribbon Hopf algebra $A$ over $k[[h]]$  dominates the Resheti\-khin-Turaev invariants produced with the ribbon category $\mathsf{Mod}_A$ of finite-rank left $A$-modules \cite{ohtsukibook, turaev}, where $V \in \mathsf{Mod}_A$ is of ``finite-rank'' if $V \cong \bigoplus_{\mathrm{finite}} k[[h]]$ as $k[[h]]$-module.

More precisely, let $\mathcal{T}_{\mathsf{Mod}_A}$ be the ribbon category of tangles whose components are coloured by elements of $\mathsf{Mod}_A$ and let $RT: \mathcal{T}_{\mathsf{Mod}_A} \to \mathsf{Mod}_A$ be the Reshetikhin-Turaev functor invariant. Given  an upwards ordered $n$-component tangle $L$, suppose by convenience that the components are ordered in increasing order according to the source $+^n \in \mathcal{T}$ of $L$. Write $\sigma = \sigma_L \in \Sigma_n$ for the permutation induced by $L$, that is $\sigma (i)$ is the label of the component that hits the target $+^n \in \mathcal{T}$ of $L$ at the $i$-th position. 

Given  $(V_i, \rho_{V_i}) \in \mathsf{Mod}_A$, $1 \leq i \leq n$, if we write $$RT(L; V_1, \ldots , V_n) \in \hom{A}{V_1 \otimeshat \cdots \otimeshat V_n}{V_{\sigma(1)} \otimeshat \cdots \otimeshat V_{\sigma(n)}} $$ for the value of $L$ under $RT$ where the $i$-th component has been coloured with $V_i$,  then we have
\begin{equation}
RT(L; V_1, \ldots , V_n) = \sigma_* \circ (\rho_{V_1}\otimeshat \cdots \otimeshat \rho_{V_n})(Z_A(L)),
\end{equation}
where $\sigma_* : V_1 \otimeshat \cdots \otimeshat V_n  \to V_{\sigma(1)} \otimeshat \cdots \otimeshat V_{\sigma(n)}$ is the map induced by $\sigma$ and the symmetric braiding of $\mathsf{Mod}_{k[[h]]}$ that permutes the factors.

In particular, the universal invariant for upwards tangles can also be used to recover the Reshetikhin-Turaev of links.  Indeed any link is the braid closure $\mathrm{cl}(L)$ of a string link $L$, that is, an upwards tangle $L$ with the property that its induced permutation $\sigma_L$ is the identity. If $(V, \rho_{V}) \in \mathsf{Mod}_A$, write $\mathrm{tr}_q^V : A \to k[[h]]$ for the \textit{quantum trace} of $V$, $\mathrm{tr}_q^V (x) := \mathrm{tr}^V (\kappa x)$, where $\mathrm{tr}^V : \mathrm{End}_{k[[h]]}(V) \to  k[[h]]$ is the  (categorical) trace. Then the following equality holds: 
\begin{equation}\label{eq RT ZA}
RT(\mathrm{cl}( L); V_1, \ldots , V_n) = (\mathrm{tr}_q^{V_1} \otimeshat \cdots \otimeshat \mathrm{tr}_q^{V_n})(Z_A(L)),
\end{equation}
where the order of the components of $\mathrm{cl}( L)$ is induced by the order of $L$, and we have made the identification $\hom{k[[h]]}{k[[h]]}{k[[h]]} \cong k[[h]]$.

\section{Gaussian calculus} \label{sec gauss}

In this section we introduce the central algebraic object of this paper, a topological ribbon Hopf algebra $\D$. We also provide an extensive discussion to motivate the usage of generating series and  give concrete examples. After that we recall Bar-Natan - van der Veen's novel technique coined \textit{Gaussian calculus} for the algebra $\D$ \cite{barnatanveengaussians}. 

Generating series and the Gaussian calculus will be essential tools in the study of the universal tangle invariant subject to $\D$. As we saw for this algebra (and for many other topological ribbon Hopf algebras), the universal $R$-matrix  is given by an infinite sum,  so that the expression of $Z_\D(K)$ is unmanageable and the invariant becomes hardly possible to compute. The Gaussian calculus will provide a strategy to effectively compute the universal invariant modulo a power of $\varepsilon$ in polynomial time.

\subsection{The algebra $\mathbb{D}$}\label{subsec algebra D}

Let $X$ be a set, and denote by $k\langle X \rangle$ the free $k$-algebra generated by $X$. The \textit{topologically free $k[[h]]$-algebra generated by $X$} is the $k[[h]]$-algebra $k\langle X \rangle[[h]]$  of formal series with coefficients in $k\langle X \rangle$, endowed with the $h$-adic topology as a topologically free $k[[h]]$-module. If $R$ is a subset of  $k\langle X \rangle[[h]]$, the quotient of $k\langle X \rangle[[h]]$ by the $h$-adic closure of the two-sided ideal generated by $R$ is called the\textit{ $k[[h]]$-algebra topologically generated by $X$ subject to the relations $R$.}

For the rest of the subsection we fix $k$ to be the ring $\Q_\varepsilon:=\Q[\varepsilon]$ of rational polynomials in the unknown $\varepsilon$.

We now define our main algebraic object. The algebra $\D$  is the $\Qeh$-algebra topologically generated by $\bm{y,t,a,x}$ subject to the relations
\begin{equation}
[\bm{t},-]=0 \quad , \quad [\bm{a},\bm{x}]=\bm{x} \quad , \quad [\bm{a},\bm{y}]=-\bm{y} \quad , \quad [\bm{x},\bm{y}]_q=\frac{1-\bm{A}^2\bm{T}}{h},
\end{equation}
where $q:= e^{h \varepsilon}$, $[u,v]:= uv-vu$, $[u,v]_q:= uv-q\cdot vu$, $\bm{A}:= e^{-h \varepsilon \bm{a}}$ and $\bm{T}:= e^{-h \bm{t} }$. The exponential is to be understood as a power series, $e^u= 1+u+u^2/2+ \cdots$.

The following proposition describes the structure of $\D$. Set $[n]_q := \frac{1-q^n}{1-q}$ and $[n]_q! := [n]_q [n-1]_q \cdots [1]_q$, and write $\bm{b}:= \bm{t}+\bm{a}\varepsilon$ (which can obviously replace $\bm{t}$ as a generator).

\begin{proposition}[\cite{barnatanveengaussians}]\label{prop Hopf structure D}
The algebra $\D$ is a topological ribbon Hopf algebra over $\Qeh$. The comultiplication and the antipode are determined by 
\begin{align*}
\Delta (\bm{y})&=  1 \otimeshat \bm{y}+ \bm{y} \otimeshat \bm{TA} \qquad &, \qquad S(\bm{y}) &= -\bm{y}(\bm{TA})^{-1}\\
\Delta (\bm{t})&= \bm{t} \otimeshat 1+ 1\otimeshat \bm{t} \qquad &, \qquad S(\bm{t})&= -\bm{t}\\
\Delta (\bm{a})&= \bm{a} \otimeshat 1+ 1\otimeshat \bm{a} \qquad &, \qquad S(\bm{a})&= -\bm{a}\\
\Delta (\bm{x})&= 1\otimeshat\bm{x}  +\bm{x}\otimeshat \bm{A}  \qquad &, \qquad S(\bm{x})&= -\bm{x}\bm{A}^{-1},
\end{align*}
whereas  the counit $\epsilon$ maps all the generators to $0$. The universal $R$-matrix and the balancing element are given by
$$ \bm{R}=\sum_{n,m=0}^\infty  \frac{h^{n+m}}{[m]_q!n!}  \bm{y}^m \bm{b}^n \otimeshat \bm{a}^n \bm{x}^m \qquad , \qquad \bm{\kappa}= +\sqrt{\bm{A}^2\bm{T}}.$$
\end{proposition}

In the above proposition, the positive square root is again to be understood as a formal power series, $\sqrt{1+u} = 1+u/2-u^2/8+ \cdots$.

\begin{remark}
The algebra $\D$ is the Drinfeld double of a quantum enveloping algebra \cite{chari_pressley}, in the same spirit as the Jimbo-Drinfeld  $U_h(\mathfrak{sl}_2)$ algebra ``almost'' is. Even more, the algebra $\D$ arises from the simplest non-trivial Lie algebra (over $\Q$), namely the unique non-abelian 2-dimensional Lie algebra $\mathfrak{h}$, with bracket $[e_1,e_2]= e_1$.  More precisely, let $\mathbb{B}$ be the $\Qeh$-algebra topologically generated by $\bm{y,b}$  subject to the single relation $[\bm{y}, \bm{b}]= \varepsilon \bm{y}$. This happens to be a topological Hopf algebra by setting
\begin{align*}
\Delta (\bm{y})&= 1 \otimeshat \bm{y} +\bm{y} \otimeshat \bm{B}  \qquad &, \qquad S(\bm{y}) &= -\bm{y}\bm{B}^{-1}\\
\Delta (\bm{b})&= \bm{b} \otimeshat 1+ 1\otimeshat \bm{b} \qquad &, \qquad S(\bm{b})&= -\bm{b}.
\end{align*}
where $\bm{B}:= e^{-h\bm{b}}$. Observe that this is a quantisation of the universal enveloping algebra $U(\mathfrak{h})$ of $\mathfrak{h}$ with its standard Hopf algebra structure (after setting $\varepsilon =1)$. It can readily be seen that $\mathbb{B}$ is topologically free, concretely the \textit{ordering} $\Qeh$-homomorphism $$\mathcal{O}: \Qe [y,b][[h]] \toiso \mathbb{B} \qquad , \qquad y^p b^q \mapsto \bm{y}^p \bm{b}^q $$ is an isomorphism of topological $\Qeh$-modules (note that the obvious isomorphism $ \Qe [y,b] \cong \bigoplus_{p,q\geq 0} \Qe  y^p b^q$ expresses that $\Qe [y,b]$ is a free $\Qe$-module).

In analogy with the $U_h(\mathfrak{sl}_2)$ case \cite{chari_pressley}, let $\widetilde{\mathbb{B}}$ be the $\Qeh$-subalgebra generated by $\widetilde{\bm{y}} := h\bm{y}$ and $\widetilde{\bm{b}} := h\bm{b}$.  Set $\mathbb{A} :=  \hom{\Qe[[h]]}{\widetilde{\mathbb{B}}}{\Qeh} $ for the dual module of $\widetilde{\mathbb{B}}$. A careful analysis shows that $\mathbb{A}$ is topologically generated by two elements $\bm{a,x}$ subject to the single relation $[\bm{a},\bm{x}]= \bm{x}$, and  setting  
\begin{align*}
\Delta (\bm{a})&= \bm{a} \otimeshat 1+ 1\otimeshat \bm{a} \qquad &, \qquad S(\bm{a})&= -\bm{a}\\
\Delta (\bm{x})&= \bm{x} \otimeshat 1 + \bm{A} \otimeshat \bm{x}  \qquad &, \qquad S(\bm{x}) &= -\bm{A}^{-1}\bm{x}
\end{align*}
where $\bm{A}:= e^{-h \varepsilon \bm{a}} $, turns $\mathbb{A}$ into a topological Hopf algebra. Besides, $\mathbb{A}$ is also a topologically free $\Qeh$-module via the ordering map
$$\mathcal{O}: \Qe [a,x][[h]] \toiso \mathbb{B} \qquad , \qquad a^p x^q \mapsto \bm{a}^p \bm{x}^q.$$

Then the algebra $\D$ arises as the Drinfeld  double $\D := \mathcal{D}(\mathbb{B})= \mathbb{A} \otimeshat \mathbb{B}$ with the Hopf algebra structure determined by the Hopf algebra structures of $\mathbb{A}$ and $\mathbb{B}$ and the natural pairing  $\langle -,-\rangle: \mathbb{A} \otimeshat \widetilde{\mathbb{B}} \to \Qeh$.
\end{remark}

The following proposition directly follows from the discussion above and \eqref{eq top free tensor}:

\begin{proposition}\label{prop ordering D}
The algebra $\D$ is topologically free. More precisely, the ordering map
$$\mathcal{O}: \Qe [y,t,a,x][[h]] \toiso \mathbb{D} \qquad , \qquad y^p t^q a^r x^s \mapsto \bm{y}^p \bm{t}^q \bm{a}^r \bm{x}^s $$
is an isomorphism of topological $\Qeh$-modules.
\end{proposition}

\begin{remark}
The topological Hopf algebra $\D_{\varepsilon=1}:= \D / (\varepsilon -1)\D$ can be seen a quantisation of $U(\mathfrak{gl}_2)$, the universal enveloping algebra of the (rational) Lie algebra $\mathfrak{gl}_2$, in the same terms as in  \cref{ex quantisation}. This follows easily using generators $H,X,Y,Z$ for $\mathfrak{gl}_2$, where $H,X,Y$ are a $\mathfrak{sl}_2$-triple and $Z$ is central, and replacing $H$ by the generator $H'=(H+Z)/2$, since the $q$-commutator $[\bm{x},\bm{y}]_q$ becomes  $[\bm{x},\bm{y}]= 2\varepsilon \bm{a} + \bm{t}$ for $h=0$.

Even more, the Lie algebra projection $\mathfrak{gl}_2  \longtwoheadrightarrow  \mathfrak{sl}_2$ can be upgraded to a surjective Hopf algebra homomorphism $\D_{\varepsilon=1} \longtwoheadrightarrow U_h(\mathfrak{sl}_2)$ with kernel the submodule generated by $\bm{t}$.
\end{remark}

The previous remark together with \eqref{eq RT ZA} imply

\begin{corollary}[\cite{barnatanveengaussians}]\label{cor ZD recovers}
The universal invariant $Z_\D$ of $\D$ determines the universal invariant $Z_{U_h(\mathfrak{sl}_2)}$ of $U_h(\mathfrak{sl}_2)$ and hence all coloured Jones polynomials.
\end{corollary}

\subsection{A motivating example for Gaussian expressions}\label{sect mot example}
We now let $k$ be a ring which is \textit{uniquely divisible} as an abelian group, ie  $\Q \subseteq k$. As pointed out above, we are mostly interested in computing the universal invariant $Z_A$ for a choice of topological ribbon Hopf algebra $A$ over some power series ring $k[[h]]$ in an effective way. If we are to compute the universal invariant of a knot from a diagram with $r$ beads, note that we will have to perform the composites 
\begin{equation}\label{eq composites A^tensors}
k[[h]] \to A^{\otimeshat r} \overset{\mu \otimeshat \id}{\to} A^{\otimeshat r-1} \overset{\mu \otimeshat \id}{\to} A^{\otimeshat r-2} \overset{\mu \otimeshat \id}{\to} \cdots \overset{\mu \otimeshat \id}{\to} A
\end{equation}
where the first arrow maps $1 \in k[[h]]$ to the element determined by the $r$ beads. The main issue is that this composite is in general extremely hard to understand, $A$ being noncommutative and of infinite rank. This means that instead we would like to have an alternative description of the sets  $\hom{k[[h]]}{A^{\otimeshat p}}{A^{\otimeshat q}}$ and the composite law $$\hom{k[[h]]}{A^{\otimeshat p}}{A^{\otimeshat q}} \otimeshat \hom{k[[h]]}{A^{\otimeshat q}}{A^{\otimeshat r}} \overset{\circ}{\to} \hom{k[[h]]}{A^{\otimeshat p}}{A^{\otimeshat r}},$$ in which the previous sequence of composites becomes feasible.

As an illustrative example let us consider $V,W$ free $k$-modules of finite rank (for instance two finite-dimensional vector spaces).  The canonical map $$V^* \otimes W \toiso \hom{k}{V}{W} \qquad , \qquad \xi \otimes w \mapsto (v \mapsto \xi (v) w)$$ is an isomorphism with inverse given by sending a linear map $T$ to $\sum_{i=1}^{n} \omega_i \otimes T(e_i)$, where $(e_1, \ldots , e_n)$ is a basis of $V$ and $(\omega_1, \ldots , \omega_n)$ is its dual basis. Under this isomorphism, the composite of maps is precisely given by the dual pairing $\langle -,- \rangle$ between a $k$-module and its dual,
$$\begin{tikzcd}
\hom{k}{V}{W}  \otimes \hom{k}{W}{U}   \rar{\circ} & \hom{k}{V}{U}\\
V^* \otimes W \otimes W^* \otimes U \rar{\id \otimes \langle -,- \rangle \otimes \id} \arrow{u}{\cong} & V^* \otimes U \arrow{u}{\cong}
\end{tikzcd}$$

Now suppose that $V,W$ are free $k$-modules of infinite countable rank, and for convenience write $V= k[x]$ and $W=k[y]$. Note that $V^* \cong k[[\xi]]$, for $$\hom{k}{\bigoplus_i k\cdot x^i}{k} \cong \prod_i \hom{k}{k\cdot x^i}{k} \cong \prod_i k\cdot \xi^i.$$  Now the canonical map $$  k[x]^* \otimes k[y] \cong k[[\xi]] \otimes k[y] \to \hom{k}{k[x]}{k[y]} $$ which is in these terms given by $$ \left( \sum_i  \lambda_i \xi^i \right) \otimes p(y) \mapsto \left( x^i \mapsto \lambda_i p(y) \right)$$ is still injective but fails to be an isomorphism, for the ``renaming'' map $T(p(x)):= p(y)$ cannot lie in the image. To promote this map to an isomorphism we are then forced to extend it to its $\xi$-adic completion $\widehat{k[[\xi]] \otimes k[y]} \cong k[y][[\xi]]$,
$$\begin{tikzcd}
k[y][[\xi]] \arrow[dashed]{dr}  &  \\
k[[\xi]] \otimes k[y] \arrow[hook]{u} \arrow{r} &  \hom{k}{k[x]}{k[y]} 
\end{tikzcd}$$
where the dashed arrow must be given, by the commutativity of the diagram, by 
\begin{equation}\label{eq correspondence hom power series}
\sum_i p_i(y) \xi^i  \mapsto (x^i \mapsto p_i(y)).
\end{equation}
Observe that, even though we passed to the completion of $k[[\xi]] \otimes k[y]$, we still have the duality between $\xi$ and $x$ as in the base case. The dashed arrow is indeed an isomorphism, for 
$$\hom{k}{\bigoplus_i k\cdot x^i}{k[y]} \cong \prod_i \hom{k}{k\cdot x^i}{k[y]} \cong \prod_i k[y] \cdot \xi^i \cong k[y][[\xi]]$$  and it is readily verified that the map is precisely given by \eqref{eq correspondence hom power series}. The composite in the infinite countable case is still given by the duality pairing,
$$\begin{tikzcd}
\hom{k}{k[x]}{k[y]}  \otimes \hom{k}{k[y]}{k[z]}   \rar{\circ} & \hom{k}{k[x]}{k[z]}\\
k[y][[\xi]] \otimes k[z][[\eta]] \rar{ \langle - \rangle_{y, \eta}} \arrow{u}{\cong} & k[z][[\xi]]  \arrow{u}{\cong}
\end{tikzcd}$$
where we put $\eta$ for the variable dual to $y$ and $\langle - \rangle_{y, \eta}$ is the duality paring between these two variables, $\eta^i (y^j)= \delta_{ij}$. More precisely if $p_i (y)= \sum_r^{N_i} \lambda_r^i y^r$ then 
\begin{equation}\label{eq pairing base case}
\left\langle  \left( \sum_i p_i(y) \xi^i \right) \otimes \left( \sum_j q_j(z) \eta^j \right)   \right\rangle_{y, \eta} := \sum_i \left( \sum_r^{N_i} \lambda_r^i q_r(z)    \right) \xi^i .
\end{equation}

What we have learned is that replacing a map by a power series amounts to keeping track of the image of all basis elements under such a map. For instance, the renaming map $T$ from above is given by $T= \sum_i y^i \xi^i$. However keeping track of infinite sums is simply too cumbersome. If we are to deal with power series, it would be most convenient to introduce a factor $1/i!$ in every summand before so that $T$ becomes an exponential, $$T= \sum_i \frac{1}{i!} y^i \xi^i = e^{y \xi}$$
obtaining a compact, easier-to-handle expression instead. We can achieve this by slightly modifying the isomorphism \eqref{eq correspondence hom power series} as
\begin{equation}\label{eq intermm step isomorp}
\sum_i p_i(y) \xi^i  \mapsto (x^i \mapsto i!  p_i(y)).
\end{equation}
In doing so, the pairing between dual variables $\eta$ and $y$ now becomes $\eta^i (y^j)= i! \ \delta_{ij}$ and the formula \eqref{eq pairing base case} is now given by 
\begin{equation}\label{eq composite intermm}
\left\langle  \left( \sum_i p_i(y) \xi^i \right) \otimes \left( \sum_j q_j(z) \eta^j \right)   \right\rangle_{y, \eta} := \sum_i \left( \sum_r^{N_i} r!  \lambda_r^i q_r(z)    \right) \xi^i .
\end{equation}

\begin{remark}
For convenience we will identify $k[y][[\xi]] \otimes k[z][[\eta]]$ with its image under the canonical injection $$k[y][[\xi]] \otimes k[z][[\eta]] \hooklongrightarrow k[y,z][[\xi, \eta]] \qquad , \qquad P(y, \xi) \otimes Q(z, \eta) \mapsto P(y, \xi) Q(z, \eta).$$ In this way we obtain more reduced expressions, for instance the composite of two ``renaming'' maps is just $\langle e^{y \xi + z \eta}  \rangle_{y, \eta} = e^{z \xi}.$

We warn the reader that the pairing $\langle - \rangle_{y, \eta}$ does \textit{not} extend to $k[y,z][[\xi, \eta]]$: for instance, the pairing $\langle \sum_i y^i \eta^i  \rangle_{y, \eta} = \sum_i 1$ is not well-defined.
\end{remark}

\begin{warning}
Since we will mostly be working with algebras given by generators and relations (which are free as modules), it is convenient to write instead $V=k[x_1, \ldots , x_n]$ rather than $V=k[x]$ (both are clearly isomorphic $k$-modules).  In doing so very little changes: we have now an isomorphism of $k$-modules 
\begin{equation}\label{eq THE isomorphism}
k[y_1, \ldots, y_m][[\xi_1, \ldots , \xi_n]] \toiso \hom{k}{k[x_1, \ldots, x_n]}{k[y_1, \ldots, y_m]}
\end{equation}
which is given by (compare with \eqref{eq intermm step isomorp})
\begin{equation}\label{eq almost final step isomorp}
\sum_{i_1, \ldots, i_n \geq 0} p_{i_1, \ldots, i_n}  \xi^{i_1} \cdots \xi^{i_n}  \mapsto (x^{i_1} \cdots x^{i_n} \mapsto i_1! \cdots i_n!  p_{i_1, \ldots, i_n}),
\end{equation}
where $ p_{i_1, \ldots, i_n}= p_{i_1, \ldots, i_n} (y_1, \ldots , y_m)$. As before we must modify the pairing accordingly, namely $(\eta^{j_1}_1 \cdots \eta^{j_m}_m)(y^{i_1}_1 \cdots y^{i_m}_m)= i_1! \cdots i_m!  \delta_{i_1 j_1} \cdots \delta_{i_m j_m}$, as well as the formula \eqref{eq composite intermm}. By doing this, we can write the composite of multivariable renaming maps as $\langle e^{\sum_i( y_i \xi_i +  z_i \eta_i)}  \rangle_{y, \eta} = e^{\sum_i z_i \xi_i}.$ 
\end{warning}

Let $\zeta_1, \ldots , \zeta_n $ be formal variables. Given a $k$-linear map $g: V \to W$ between $k$-modules, there is a natural extension $$\widehat{g}: V[[\zeta_1, \ldots , \zeta_n]] \to W[[\zeta_1, \ldots , \zeta_n]] $$ defined as $$\widehat{g} \left( \sum_{i_1, \ldots , i_n \geq 0} v_{i_1, \ldots , i_n} \zeta_1^{i_1} \cdots \zeta_n^{i_n}  \right) := \sum_{i_1, \ldots , i_n \geq 0} g(v_{i_1, \ldots , i_n}) \zeta_1^{i_1} \cdots \zeta_n^{i_n} $$ that we call the \textit{$(\zeta_1, \ldots , \zeta_n )$-extension} of $g$.  The following lemma will be very useful:

\begin{lemma}\label{lemma extension hat}
Let $V_1, V_2$ be free $k$-modules of infinite countable rank with preferred isomorphisms $\mathcal{O}_1: k[x_1, \ldots , x_n] \toiso V_1$ and $\mathcal{O}_2: k[y_1, \ldots ,y_m] \toiso V_2$. The element $k[y_1, \ldots , y_m][[\xi_1,  \ldots , \xi_n]]$
corresponding to $$\mathcal{O}_2^{-1} \circ f \circ \mathcal{O}_1: k[x_1, \ldots , x_n] \to  k[y_1, \ldots ,y_m] $$ under the isomorphism \eqref{eq THE isomorphism} equals $$(\widehat{\mathcal{O}_2^{-1}} \circ \widehat{f} \circ \widehat{\mathcal{O}_1} )(e^{x_1 \xi_1 + \cdots + x_n \xi_n}) ,$$ where all the hats denote $(\xi_1,  \ldots , \xi_n)$-extensions.
\end{lemma}
\begin{proof}
If $p_{i_1, \ldots , i_n} := (\mathcal{O}_2^{-1} \circ f \circ \mathcal{O}_1)(x_1^{i_1} \cdots x_n^{i_n})$, then by definition the power series corresponding to $\mathcal{O}_2^{-1} \circ f \circ \mathcal{O}_1$ is $$ \sum_{i_1, \ldots , i_n \geq 0} \frac{1}{i_1! \cdots i_n!} p_{i_1, \ldots , i_n} \xi_1^{i_1} \cdots \xi_n^{i_n}.$$ Since for the hat maps the variables $\xi_1,  \ldots , \xi_n$ are treated as constants,  this is precisely the element $(\widehat{\mathcal{O}_2^{-1}} \circ \widehat{f} \circ \widehat{\mathcal{O}_1} )(e^{x_1 \xi_1 + \cdots + x_n \xi_n})$.
\end{proof}

\subsection{Contraction of perturbed Gaussian expressions}\label{sect contraction}

There are many interesting $k$-linear maps between free infinite countable rank modules that can be expressed as an exponential of a finite expression, or as a product of an exponential expression and something else.

\begin{examples}
\begin{enumerate}
\item The identity $\Q[x_1, \ldots , x_n] \to \Q[x_1, \ldots , x_n] $ corresponds to $e^{x_1 \xi_1 + \cdots + x_n \xi_n}$.
\item For $n \geq 0$, the map ``multiplication by $x^n$'',  $\Q[x] \to \Q[x]$, $x^i \mapsto x^{i+n}$ corresponds to $e^{x \xi} x^n$.
\item The map ``derivation''  $\Q[x] \to \Q[x]$, $x^i \mapsto i x^{i-1}$ corresponds to $e^{x \xi} \xi$.
\item The multiplication law of $\Q[x]$, that is $\Q[x_1, x_2] \cong \Q[x]\otimes \Q[x] \to \Q[x]$, $x_1^n x_2^m \mapsto x^{m+n}$, corresponds to the element $e^{x(\xi_1 + \xi_2)}$. Here subindices  have been added to distinguish the two generators.
\end{enumerate}
\end{examples}

Here is a more elaborate example.

\begin{example}[\cite{barnatanveengaussians}]
Let $A$ be the (rational) Weyl algebra, that is, the free $\Q$-algebra generated by $\bm{p}, \bm{x}$ subject to the relation $[\bm{p}, \bm{x}]=1$. Applying recursively $\bm{xp}= \bm{px}-1$, we can write any element of $A$ as a sum of ``ordered monomials'' $\sum_i \bm{p}^{n_i} \bm{x}^{m_i}$, which amounts to saying that there is an isomorphism $$\mathcal{O}: \Q [p,x] \toiso A \qquad , \qquad p^n x^m \mapsto \bm{p}^n \bm{x}^m$$ (compare with \cref{prop ordering D}). Under this isomorphism, the multiplication of $A$ can be seen as a map $\Q[p_1, x_1, p_2 , x_2] \to \Q[x,p]$, which maps the monomial $p_1^{n_1 }x_1^{m_1}p_2^{n_2 }x_2^{m_2}$  to $\mathcal{O}^{-1}(\bm{p}^{n_1 }\bm{x}^{m_1}\bm{p}^{n_2 }\bm{x}^{m_2})$. Via our identification $$\hom{\Q}{\Q[p_1, x_1, p_2 , x_2] }{\Q[x,p]} \cong \Q[p,x][[\pi_1, \xi_1, \pi_2, \xi_2]]$$ described in \eqref{eq THE isomorphism}, the multiplication of the Weyl algebra corresponds with the element 
\begin{equation}\label{eq power series for mu Weyl}
e^{p(\pi_1 + \pi_2) + x(\xi_1 + \xi_2)-\xi_1 \pi_2}.
\end{equation}
For the key observation is the following equality in $A[[\pi_1, \xi_1, \pi_2, \xi_2]]$: 
\begin{equation}\label{eq Weyl commutation relation}
e^{ \bm{x} \xi_1} e^{ \bm{p} \pi_2}= e^{ \bm{p}\pi_2}e^{ \bm{x}\xi_1}e^{-\xi_1 \pi_2}.
\end{equation}
This is easy to see using the Baker-Campbell-Hausdorff formula, since the commutator $[\pi_2 \bm{p}, \xi_1 \bm{x}]$ is central (both sides of the equation equals $e^{\bm{x}\xi_1 + \bm{p\pi_2}-\frac{1}{2}\xi_1 \pi_2}$).

Then, according to \cref{lemma extension hat}, the power series corresponding to the multiplication $\mu: A  \otimes A  \to A$ of the Weyl algebra is
\begin{align*}
(\widehat{\mathcal{O}^{-1}} \circ \widehat{\mu} \circ \widehat{\mathcal{O}\otimes \mathcal{O} } ) &(e^{p_1 \pi_1 + x_1 \xi_1+p_2 \pi_2 + x_2 \xi_2}) =\\  &= (\widehat{\mathcal{O}^{-1}} \circ \widehat{\mu} )(e^{(\bm{p} \otimes 1)\pi_1}e^{(\bm{x} \otimes 1)\xi_1}e^{(1 \otimes \bm{p})\pi_2}e^{(1 \otimes \bm{x})\xi_2}) \\ &=\widehat{\mathcal{O}^{-1}} (e^{\bm{p} \pi_1}e^{\bm{x} \xi_1}e^{\bm{p}\pi_2}e^{\bm{x}\xi_2}) \\
&\overset{\eqref{eq Weyl commutation relation}}{=}  \widehat{\mathcal{O}^{-1}} (e^{\bm{p} \pi_1}e^{ \bm{p}\pi_2}e^{ \bm{x}\xi_1}e^{\bm{x}\xi_2}e^{-\xi_1 \pi_2})\\
&= e^{p(\pi_1 + \pi_2) + x(\xi_1 + \xi_2)-\xi_1 \pi_2}
\end{align*}
as claimed in \eqref{eq power series for mu Weyl}.
\end{example}

\begin{remark}
Abusing notation we will denote with the same symbol a $k$-linear map and its corresponding power series  according to \eqref{eq THE isomorphism}, that we will call \textit{generating series} or \textit{generating function} of the linear map, so that in the previous example we simply write $$\mu = e^{p(\pi_1 + \pi_2) + x(\xi_1 + \xi_2)-\xi_1 \pi_2}.$$
\end{remark}

At this point the reader should be convinced of the importance of exponential expressions, more precisely of those with quadratic exponential, known as \textit{Gaussian expressions}, as well as the product of a Gaussian expresion and some other series, that we call \textit{perturbed Gaussian expressions}. If we are to focus on Gaussians, we should have a way of pairing (the operation corresponding to the composite of $k$-linear maps) Gaussian generating series effectively. In doing so it will be convenient to expand the source of the pairing map as  much as possible.

For $n > 0$ let $r_1, \ldots , r_{2n}$ and $s_1,  \ldots, s_{2n}$ be unknowns, that we will interpret as dual of each other. Define $n$-tuples $r'=(r_1, \ldots , r_{n})$, $r''= (r_{n+1}, \ldots , r_{2n})$, $s'=(s_1, \ldots , s_{n})$, $s''= (s_{n+1}, \ldots , s_{2n})$ and set $r=(r', r'')$ and $s=(s',s'')$. Besides let $u'= (u_1, \ldots , u_n)$, $u''= (u_{n+1}, \ldots , u_{2n})$, $v'=(v_1, \ldots, v_n)$, $v''= (v_{n+1}, \ldots , v_{2n})$ be $n$-tuples of unknowns. Furthermore consider $n^2$ unknowns $w_{ij}$ for $1\leq i,j \leq n$  and form the  matrix $W= (w_{ij})$.

Given an element $P  \in k[r', s'', u'', v'][[r'',s', u', v'', W]]$ (where the variables in the power series ring are to be understood the corresponding unknowns), the \textit{contraction of $P$ along $r,s$}, if it exists, is the power series $\langle P \rangle_{r,s} \in k[u'', v'][[ u', v'', W]]$ resulting from replacing every product $r_k^i s_k^j$ in $P$ by $i! \delta_{ij}$.

Denote by $ \mathscr{C}( k[r', s'', u'', v'][[r'',s', u', v'', W]])$ the $k$-submodule of contractible\footnote{This is analogous to the set of integrable functions among measurable functions.} power series along $r,s$. We then have a well-defined $k$-module map
\begin{equation}
 \langle - \rangle_{r,s} :  \mathscr{C}( k[r', s'', u'', v'][[r'',s', u', v'', W]]) \to  k[u'', v'][[ u', v'', W]]
\end{equation}
that extends the pairing that corresponds to the composite in the following sense:  for $\ell = y, z, \xi$ or $\eta$, set $\ell ' = (\ell_1, \ldots , \ell_{n})$, $\ell '' = (\ell_{n+1}, \ldots , \ell_{2n})$. If $$\varphi :  \mathscr{C}( k[r', s'', u'', v'][[r'',s', u', v'', W]]) \to k[y', y''][[\xi', \xi'']] \otimes k[z',z''][[\eta', \eta'']]$$ is a $k$-algebra morphism such that $$\varphi(r')=y' \quad , \quad \varphi (r'')= \eta ''  \quad , \quad \varphi (s')= \eta'  \quad , \quad \varphi (s'')= y'' $$
and restricts to a $k$-algebra map $$\varphi_{\mathrm{restr}} :  k[u'', v'][[u', v'', W]] \to k[z',z''][[\xi', \xi'']],$$
then we have the following commutative diagram:
$$\begin{tikzcd}
\mathscr{C}( k[r', s'', u'', v'][[r'',s', u', v'', W]])  \rar{\langle - \rangle_{r,s}} \dar{\varphi} &  k[u'', v'][[u', v'', W]] \dar{\varphi_{\mathrm{restr}} }  \\
k[y', y''][[\xi', \xi'']] \otimes k[z',z''][[\eta', \eta'']] \rar{ \langle - \rangle_{y, \eta}}  & k[z',z''][[\xi', \xi'']] 
\end{tikzcd}$$

In other words, if the product of two generating series lies in the image of $\varphi$, then we can use the contraction to compute the pairing of them (hence the composite of their corresponding linear maps). Even though it may seem we have gained little, there is a key difference: under the map $\varphi$, the variables $y', y''$ are now ``duals'' in the sense that $\varphi(r')=y'$ and $\varphi (s'')= y''$. This will allow us to heavily rely on the following theorem for the contraction of perturbed Gaussian expressions to perform composites:

\begin{theorem}[Contraction, \cite{barnatanveengaussians}]\label{thm contraction}
With the notations above, consider the matrix $\widetilde{W} := (I-W)^{-1}$ and set $\bar{r}:= r+v\widetilde{W}$ and $\bar{s}:= \widetilde{W} (s+u)$, where we view $s$ and $u$ as column vectors for the matrix multiplication. Then the following equality holds:
$$ \Big\langle e^{ru+vs+rWs} P(r,s)  \Big\rangle_{r,s}  = \det ( \widetilde{W} ) \  e^{v  \widetilde{W} u} \Big\langle P(\bar{r}, \bar{s})  \Big\rangle_{r,s}. $$
\end{theorem}

Note that the matrix $(I-W)$ is indeed invertible in $k[[W]]$, for its determinant equals 1 mod $(w_{11}, \ldots, w_{nn})$.

\subsection{Generating series for $\mathbb{D}$} We now apply the Gaussian techniques developed in the two last subsections to the ribbon Hopf algebra $\mathbb{D}$ from \cref{subsec algebra D}. According to \cref{prop ordering D}, the algebra $\D$ is a topologically free $\Qe[[h]]$-module. We can easily extend the isomorphism \eqref{eq THE isomorphism} to topologically free $k[[h]]$-modules in view of \cref{prop adjuction and top free}: we have an analogous bijection
\begin{equation}\label{eq THE isomorphism2}
k[y_1, \ldots, y_m][[\xi_1, \ldots , \xi_n]][[h]] \toiso \hom{k[[h]]}{k[x_1, \ldots, x_n][[h]]}{k[y_1, \ldots, y_m][[h]]}
\end{equation}
where the correspondence is still the same as in \eqref{eq almost final step isomorp} for every power of $h$. Now recall from \eqref{eq composites A^tensors} that we are mostly interested in studying maps between $n$-fold tensor products of the Hopf algebra at hand. To keep track of the factors of the tensor product in the passage to the polynomial rings under the isomorphism $\mathcal{O}: \Qe [y,t,a,x][[h]] \toiso \mathbb{D}$, we consider the following: given a finite ordered set $I=(i_1,  \ldots , i_n)$ with $n$ elements, write $\ell_I= (\ell_{i_1}, \ldots , \ell_{i_n})$ for $\ell = y,t,a,x$. Again by \cref{prop adjuction and top free} we have  an isomorphism 
\begin{equation}\label{eq O^otimeshat I}
 \mathcal{O}^{\otimeshat I}: \Qe [y_I,t_I,a_I,x_I][[h]] \toiso \mathbb{D}^{\otimeshat n}.
\end{equation}
Combining this isomorphism with \eqref{eq THE isomorphism2} we obtain the following

\begin{proposition}\label{prop bij Dn -> Dm with polys}
Let $I,J$ be finite ordered sets with $\# I=n$ and $\# J=m$. Then the bijection \eqref{eq THE isomorphism} gives rise to a bijection
$$\Qe [y_J, t_J, a_J, x_J][[\eta_I, \tau_I, \alpha_I, \xi_I]][[h]] \toiso \hom{\Qe[[h]]}{\mathbb{D}^{\otimeshat n}}{\mathbb{D}^{\otimeshat m}}$$
where $\eta, \tau, \alpha,  \xi$ are the dual variables of $y,t,a,x$, respectively.
\end{proposition}

Furthermore, we can use the Gaussian calculus explained in the previous subsection also for the structure maps of the ribbon Hopf algebra $\D$. Their generating series are not perturbed Gaussian expressions, but their exponential part can be divided into two such, in a sense that we make precise now.

A generating series in $\Qe [y_J, t_J, a_J, x_J][[\eta_I, \tau_I, \alpha_I, \xi_I]][[h]]$ is a \textit{two-step Gaussian} if it is of the form $e^G$ with
$$G= \begin{pmatrix}
\tau_I & a_J
\end{pmatrix} 
 \begin{pmatrix}
\Gamma_{11} & \Gamma_{12}\\
\Gamma_{21} & 0
\end{pmatrix} 
 \begin{pmatrix}
t_J \\ \alpha_I
\end{pmatrix} +  \begin{pmatrix}
y_J & \xi_I
\end{pmatrix} \Lambda 
 \begin{pmatrix}
\eta_I \\ x_J
\end{pmatrix} $$
where the matrices $\Gamma_{k\ell}$ have coefficients in $\Z \oplus \Z h$ and $\Lambda$ is a matrix with coefficients finite expressions in $\mathcal{A}_i:= e^{\alpha_i}$,  $T_j := e^{-t_jh}$ and $h$,  for $i\in I, j \in J$

Write $G_1$ and $G_2$ for the first and second summand of $G$, so $G=G_1+G_2$. Given a two-step Gaussian $e^G$, if we view $y,x, \eta, \xi$ as constants, then $e^G$ is a perturbed Gaussian expression $e^{G_1}e^{G_2}$ (where the perturbation is $e^{G_2}$). Alternatively, if we view $t,a,\tau, \alpha$ as constants, then $e^G$ is a perturbed Gaussian expression $e^{G_2}e^{G_1}$ with constant perturbation part $e^{G_1}$. The upshot is that we can still compose two-step Gaussians using the Contraction \cref{thm contraction} twice: first for the variables $a,t, \tau, \alpha$ treating $y,x, \eta, \xi$ as constants, and later for the remaining $y,x, \eta, \xi$.

\begin{notation}
If $f: \D^{\otimeshat n} \to \D^{\otimeshat m}$ is a $\Qe [[h]]$-module map and $I,J$ are choices of finite ordered sets with $\# I=n$, $\# J=m$, then we will write $$f_I^J  :=  (\mathcal{O}^{\otimeshat J})^{-1} \circ  f \circ  \mathcal{O}^{\otimeshat I} :  \Qe [y_I,t_I,a_I,x_I][[h]] \to  \Qe [y_J,t_J,a_J,x_J][[h]]$$ and abusing notation its corresponding generating series will also be denoted by $f_I^J$.
\end{notation}

The following key theorem is highly nontrivial and heavily relies on the construction of $\D$ as a Drinfeld double. Recall that, by \cref{prop adjuction and top free}, the datum of an element of $\D^{\otimeshat m}$ is equivalent to a $\Qe [[h]]$-module map $\Qe [[h]] \to \D^{\otimeshat m}$, and by \cref{prop bij Dn -> Dm with polys} this is determined by a generating series in $\Qe [y_J, t_J, a_J, x_J][[h]]$, where $\# J=m$.

\begin{theorem}[\cite{barnatanveengaussians}]\label{thm DoPeGDO}
The generating series of all structure maps of the ribbon Hopf algebra $\D$ are perturbed two-step Gaussians with perturbation a power series in $\varepsilon$, that is, they are all of the form $$e^G (P_0 + P_1 \varepsilon + P_2 \varepsilon^2 + P_3 \varepsilon^3 + \cdots)$$ where $e^G$ is a two-step Gaussian and every $P_k$ is a finite expression in $y$, $t$, $a$, $x$, $\eta$, $\alpha$, $\xi$, $T$, $\mathcal{A}$ and $h$.

More concretely, we have the following equalities:

\begin{align*}
R^{ij} &= \exp \Big[ (t_i a_j + y_i x_j)h  \Big] \quad , \quad (R^{-1})^{ij} = \exp \Big[ -(t_i a_j + y_i x_j T_i^{-1})h  \Big] \pmod \varepsilon\\
\kappa^i &= T_i^{1/2}  \quad , \quad (\kappa^{-1})^i = T_i^{-1/2} \pmod \varepsilon\\
\mu_{ij}^k &= \begin{aligned}[t]
 \exp \Bigg[  a_k (\alpha_i + \alpha_j) + t_k (\tau_i + \tau_j) &+ y_k(\eta_i + \A_i^{-1} \eta_j) \\ &+   x_k (\xi_j + \A_j^{-1}\xi_i) + \frac{1-T_k}{h}\eta_j \xi_i \Bigg]  \pmod \varepsilon
\end{aligned}\\
\Delta_i^{jk} &=  \begin{aligned}[t]
 \exp \Bigg[  \alpha_i (a_j + a_k) &+ \tau_i (t_j + t_k) + \eta_i (T_k y_j + y_k) +  \xi_i (x_j+x_k)  \Bigg]  \pmod \varepsilon
\end{aligned}\\
S_i^j &=  \begin{aligned}[t]
 \exp \Bigg[- \alpha_i a_j - \tau_i t_j - \A_i T_j^{-1} y_j \eta_i- \A_i x_j \xi_i - \A_i \frac{1-T_j^{-1}}{h} \eta_i \xi_i   \Bigg]  \pmod \varepsilon
\end{aligned}
\end{align*}
\end{theorem}

\begin{remark}\label{remark commutative gen series}
Leaving out for a moment the parameters $h$ and $\varepsilon$, the reader should compare the above generating series for the structure maps of $\D$ with those of a standard Hopf algebra structure of $\Q [y,t,a,x]$. Recall that if $\mathfrak{g}$ is a (rational) Lie algebra, then its universal enveloping algebra $U(\mathfrak{g})$ has a canonical Hopf algebra structure with comultiplication determined by $\Delta (v)= v \otimes 1 + 1\otimes v $, counit  by $\epsilon (v)=0$ and antipode by $S(v)=-v$ for $v \in \mathfrak{g}$. As $\Q [y,t,a,x]$ is isomorphic to the universal enveloping algebra of the abelian Lie algebra $\Q y \oplus \Q t \oplus \Q a \oplus \Q x $ (in fact isomorphic to its symmetric algebra), we get a canonical Hopf algebra structure determined by 
$$\mu_{ij}^k (z_i z'_j)= z_k z'_k \qquad , \qquad   \Delta_i^{jk} (z_i) = z_j + z_k \qquad , \qquad S_i^j (z_i)= -z_j $$ for $z,z'= y,t,a,x$. A straightforward application of \cref{lemma extension hat} yields the following generating series:
\begin{align}\label{eq gen func commutative}
\begin{split}
\mu_{ij}^k &= \begin{aligned}[t]
 \exp \Bigg[  a_k (\alpha_i + \alpha_j) + t_k (\tau_i + \tau_j) &+ y_k(\eta_i +  \eta_j) + x_k (\xi_j + \xi_i)  \Bigg] 
\end{aligned}\\
\Delta_i^{jk} &=  \begin{aligned}[t]
 \exp \Bigg[  \alpha_i (a_j + a_k) &+ \tau_i (t_j + t_k) + \eta_i (y_j + y_k) +  \xi_i (x_j+x_k)  \Bigg]  
\end{aligned}\\
S_i^j &=  \begin{aligned}[t]
 \exp \Bigg[- \alpha_i a_j - \tau_i t_j - y_j \eta_i-  x_j \xi_i  \Bigg] 
\end{aligned}
\end{split}
\end{align}
We can then regard the presence of $T$ and $\A$ in the generating series of \cref{thm DoPeGDO} as consequence of the nontriviality of the bracket of $\D$ and its different structure maps.
\end{remark}

\begin{example}\label{example contraction}
Let us illustrate how to compose $\Qeh$-module maps between $n$-fold tensor products of $\D$ using generating series and the Contraction \cref{thm contraction}.

Consider maps $f: \Qeh \to \D^{\otimeshat 4}$ and  $g: \D^{\otimeshat 4} \to \D^{\otimeshat 3}$ whose generating series are given by $$f^{1,2,3,4} = \exp \Big[ (t_1 a_2 + t_3 a_4 + y_1 x_2 + y_3 x_4)h \Big]  (a_1 a_2 + a_3 a_4 + y_1^2 x_2^2) \varepsilon  \pmod {\varepsilon^2}$$ and $$g_{1,3}^{0} =  \exp \Big[ a_0 (\alpha_1 + \alpha_3) + t_0 (\tau_1 + \tau_3) + y_0 (\eta_1 + \mathcal{A}_1^{-1}\eta_3 ) + x_0 (\xi_3 + \mathcal{A}_3^{-1}\xi_1)  \Big] \pmod {\varepsilon^2} $$ (strictly speaking, the second generating series should be $g_{1,3}^0 \otimes \id_{2,4}^{2,4}$, but since the composite with the identity is trivial we simply omit it in the formulas).  Both expressions are two-step Gaussians and so by the discussion above we can perform the composite
\begin{align*}
g  \circ f &= \Big\langle f^{1,2,3,4} \cdot g_{1,3}^{0}   \Big\rangle_{1,3} \\ &= \Big\langle e^{ a_0 (\alpha_1 + \alpha_3) + t_0 (\tau_1 + \tau_3) + (t_1 a_2 + t_3 a_4 + y_1 x_2 + y_3 x_4)h + y_0 (\eta_1 + \mathcal{A}_1^{-1}\eta_3 ) + x_0 (\xi_3 + \mathcal{A}_3^{-1}\xi_1) } \cdot \\ &\phantom{d} \hspace*{7.5cm} \cdot (a_1 a_2 + a_3 a_4 + y_1^2 x_2^2) \varepsilon  \Big\rangle_{1,3}
\end{align*}
using the contration formula twice as explained above.

First we contract the variables $\alpha , t, a, \tau$ (for the indices $1,3$), taking
\begin{align*}
r &= ( t_1, t_3, \alpha_1, \alpha_3) \quad , \quad s = ( \tau_1, \tau_3, a_1, a_3) \quad , \quad  W=0\\ u &=( a_2 h, a_4 h , a_0, a_0) \ , \ v= ( t_0, t_0, 0,0) 
\end{align*} 
so that $\widetilde{W}=I$ and then
\begin{align*}
 r' & = r+ v = ( t_0 + t_1, t_0 t_3, \alpha_1, \alpha_3), \\ s' & =s+u = ( \tau_1 + a_2 h, \tau_3 + a_4 h, a_0 + a_1 , a_0 + a_3).
\end{align*}
  According to \cref{thm contraction} we have 
\begin{align*}
\Big\langle f^{1,2,3,4} \cdot g_{1,3}^{0}   \Big\rangle_{a_{13}, t_{13}} &= e^{(y_1 x_2 + y_3 x_4)h +  y_0\eta_1 +  x_0 \xi_3 }  \Big\langle e^{ a_0 (\alpha_1 + \alpha_3) + t_0 (\tau_1 + \tau_3) + (t_1 a_2 + t_3 a_4 )h } \cdot \\ &\phantom{==}  e^{ y_0  \mathcal{A}_1^{-1}\eta_3  + x_0  \mathcal{A}_3^{-1}\xi_1} (a_1 a_2 + a_3 a_4 + y_1^2 x_2^2) \varepsilon    \Big\rangle_{a_{13}, t_{13}}  \\ &= e^{(y_1 x_2 + y_3 x_4)h +  y_0\eta_1 +  x_0 \xi_3 }e^{t_0(a_2+a_4)h} \cdot   \\ &  \Big\langle  e^{ y_0  \mathcal{A}_1^{-1}\eta_3  + x_0  \mathcal{A}_3^{-1}\xi_1} ((a_0 +a_1) a_2 + (a_0 + a_3) a_4 + y_1^2 x_2^2) \varepsilon         \Big\rangle_{a_{13}, t_{13}}\\
&= e^{(y_1 x_2 + y_3 x_4)h +  y_0\eta_1 +  x_0 \xi_3 }e^{t_0(a_2+a_4)h} \cdot   \\ &\phantom{=}  e^{ y_0  \eta_3  + x_0 \xi_1} ((a_0 - y_0 \eta_3 )a_2   + (a_0 -x_0\xi_1) a_4 + y_1^2 x_2^2) \varepsilon  
\end{align*}
where for the last equality we have used that $$\left\langle e^{\lambda \mathcal{A}^{\pm 1} } a \right\rangle_a = e^\lambda (\pm \lambda),$$ which can be easily verified expanding the left-hand side in power series and contracting term by term.

The second step is to use once more \cref{thm contraction} to contract the remainder variables $y,\xi, \eta, x$ (at the same indices $1,3$). Now we consider 
\begin{align*}
r &= ( y_1, y_3, \xi_1, \xi_3) \quad , \quad s = ( \eta_1, \eta_3, x_1, x_3) \quad , \quad  W=0\\ u &=( x_2 h, x_4 h , x_0, x_0) \ , \ v= ( y_0, y_0, 0,0) 
\end{align*} 
so that $\widetilde{W}=I$ and then
\begin{align*}
 r' & = r+ v = ( y_0 + y_1, y_0 y_3, \xi_1, \xi_3), \\ s' &= s+u = ( \eta_1 + x_2 h, \eta_3 + x_4 h, x_0 + x_1 , x_0 + x_3).
\end{align*}

We finally obtain
\begin{align*}
g \circ f &=  \Big\langle \Big\langle f^{1,2,3,4} \cdot g_{1,3}^{0}   \Big\rangle_{a_{13}, t_{13}} \Big\rangle_{y_{13}, x_{13}}\\ &= e^{t_0(a_2+a_4)h} \Big\langle    e^{(y_1 x_2 + y_3 x_4)h +  y_0(\eta_1+ \eta_3 ) +  x_0 ( \xi_1+ \xi_3) } \cdot \\ &\phantom{==}  ((a_0 - y_0 \eta_3 )a_2   + (a_0 -x_0\xi_1) a_4 + y_1^2 x_2^2) \varepsilon   \Big\rangle_{y_{13}, x_{13}}\\
&= e^{t_0(a_2+a_4)h} e^{y_0 (x_2 + x_4)h} \cdot \\   &\phantom{=}  \Big\langle      ((a_0 - y_0 (\eta_3 + x_4 h) )a_2   + (a_0 -x_0\xi_1) a_4 + (y_0+y_1)^2 x_2^2)  \Big\rangle_{y_{13}, x_{13}}   \\
&= e^{(t_0(a_2+a_4) + y_0 (x_2 + x_4))h} (a_0 (a_2+a_4) -y_0 a_2 x_4 h +y_0^2x_2^2 ) \varepsilon.
\end{align*}
Let us point out that in concrete examples as this one it is best to use a computer to do the calculation (the previous one only takes a few hundredths of a second), for reasons that most likely the reader will find obvious at the moment after this cumbersome (but conceptually and computationally simple) calculation. Yet in what follows the two-step contraction will play a central role in general arguments and not only on the computational side. 
\end{example}

\section{The universal invariant \texorpdfstring{$Z_\D$}{Z\_D}} \label{sect ZD}

In this section, we will define the thickening and band maps for Habiro's bottom tangles, which will provide us with a powerful strategy to study the universal invariant of 0-framed knots.  We will  use these to show our two main results about the collection of Bar-Natan - van der Veen polynomials $\rho_K^{i,j}$ in \cref{thm rho ij =0 for j>i} and \cref{thm rho 2 sth}.

In order to show these two Theorems we will need a few technical results about the generating series of (part of) the thickening map. We have included these in \cref{sect band map}.

\subsection{The thickening map}\label{sect thickening map}

In this subsection we will construct the key tool of this paper. Recall that the \textit{genus} of a knot $K$ is the minimum among the genera of its \textit{Seifert surfaces}, that is,  compact, connected, oriented surfaces in $S^3$ bounding $\mathrm{cl} (K)$. We will make use of a well-known procedure to create knots of genus $\leq g$ out of thickening of $2g$-component tangles. To that end, we will borrow (and slightly modify) the notion of bottom tangle from \cite{habiro}.

Let $n \geq 0$. An $n$-component \textit{bottom tangle} is an open tangle with $n$ components which is an element of $\hom{\mathcal{T}}{(+-)^n}{\emptyset}$ and such that the $i$-th component starts at the $(2i-1)$-th point on the bottom and ends at the $2i$-th. A \textit{vertical bottom tangle} the image of a bottom tangle $L \subset (D^1)^{\times 3}$ under a $(-\pi/2)$ rotation about the $y$ axis.

Here is an example of a vertical bottom tangle:
\begin{equation*}
\centre{
\centering
\includegraphics[width=0.18\textwidth]{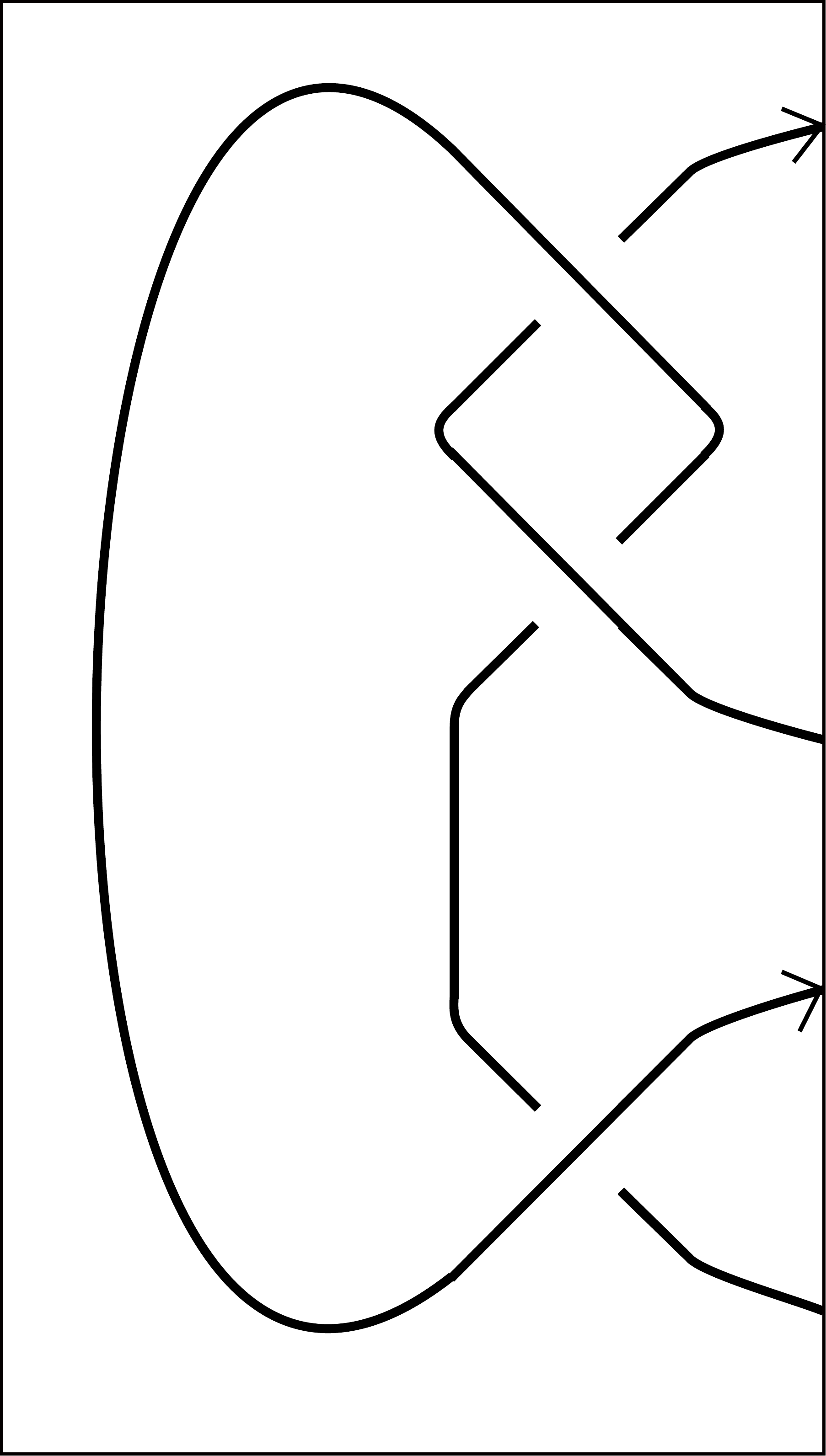}}
\end{equation*}

We endow any vertical bottom tangle with the canonical order starting from the lowest component. Also observe that both head and tail of every component of a vertical bottom tangle point up. In this sense, they behave similarly to upwards tangles.  The same argument used in \cref{lemma every tangle has rot diag} yields the following:

\begin{lemma}\label{lemma vbt rot form}
Every vertical bottom tangle has a diagram in rotational form.
\end{lemma}

Furthermore, given a ribbon Hopf algebra $A$, we can also define the universal invariant for a vertical bottom tangle by using the rules given in \eqref{eq beads norot} (or those in \eqref{eq beads rot}, according to the previous Lemma).

Let $g \geq 1$ and let $L$ be a $2g$-component vertical bottom tangle. The \textit{thickening} of $L$ is the knot $\widecheck{Th}(L)$ given as below,
\begin{equation}
\centre{
\labellist \small \hair 2pt
\pinlabel{$\vdots$}  at 514 880
\pinlabel{$\vdots$}  at 1756 870
\pinlabel{$\vdots$}  at 2057 870
\pinlabel{$\vdots$}  at 2511 870
\pinlabel{ \normalsize $L$}  at 131 1437
\pinlabel{ \normalsize $D(L)$}  at 1430 1440
\pinlabel{ \LARGE $\overset{\widecheck{Th}}{\longmapsto}$}  at 894 894
\endlabellist
\centering
\includegraphics[width=0.7\textwidth]{sketch_figures/thickening}}
\end{equation}
where $D(L)$ is the result of doubling every component of $L$ and reversing the orientation of the inner one for each pair of doubled component, and the dashed components of $L$ indicate that they are possibly knotted. Observe that (the closure of) the knot $\widecheck{Th}(L)$ comes with a choice of Seifert surface of genus $g$. Furthermore, such a knot is always 0-framed, for every crossing gives rise to two positive and two negative crossings, hence the total writhe of the resulting knot will always be zero.

If $\mathsf{vBT}_n$ denotes the set of isotopy classes of $n$-component vertical bottom tangles, then the previous construction defines a surjective\footnote{Given a closed knot of genus $g' \leq g$ in $S^3$, attach $(g-g')$ 2-dimensional 1-handles to a genus $g'$ Seifert surface for the knot, obtaining a genus $g$ Seifert surface $F$. Elementary algebraic topology ensures that there exist $2g$ simple closed curves $\gamma_1, \ldots , \gamma_{2g}$ such that $H_1 (F; \Z) \cong \Z \gamma_1 \oplus \cdots \oplus \Z \gamma_{2g}$. Deformation retract $F$ to the (closure of the) union of tubular neighbourhoods of the curves $\gamma_i$ to obtain a new Seifert surface in band form, from which the vertical bottom tangle can be read off the cores of the bands.} map
\begin{equation}
\widecheck{Th} : \mathsf{vBT}_{2g} \to \left\lbrace \parbox[c]{6.5em}{\centering   {\small  genus $\leq g$ 0-framed knots }}  \right\rbrace \subset \hom{\mathcal{T}}{+}{+}.
\end{equation}

We now aim to construct the algebraic counterpart of this map. Fix a topological ribbon Hopf algebra $A$ with universal $R$-matrix $R= \sum_i \alpha_i \otimeshat \beta_i\in A \otimeshat A$ and balancing element $\kappa \in A$. The \textit{cross map} $Cr $ and the \textit{band map} $B$ are  defined as follows:
$$ Cr: A \otimeshat A \to A \otimeshat A \qquad , \qquad Cr(x \otimeshat y) := \sum_i x \alpha_i \otimeshat \beta_i y$$ and $$ B: A \otimeshat A \to A \qquad , \qquad B(x \otimeshat y) := \sum_{(x), (y)} x_{(2)} S(y_{(1)}) \kappa S( x_{(1)} ) \kappa y_{(2)}  , $$
where  for $z \in A$, we use the Sweedler notation $\Delta (z)= \sum_{(z)} z_{(1)} \otimeshat z_{(2)}$.

For $g \geq 1$, the \textit{algebraic thickening map} $Th$ is the composite
\begin{equation}
\begin{tikzcd}
A^{\otimeshat 2g} \rar{Cr^{\otimeshat g}} & A^{\otimeshat 2g} \rar{B^{\otimeshat g}} & A^{\otimeshat g} \rar{\mu^{[g]}} & A
\end{tikzcd}
\end{equation}
where $\mu^{[g]}$ denotes the $g$-fold multiplication map.

The following result is essentially well-known, but we include a detailed proof for its relevance in this paper.
\begin{proposition}
Let $g \geq 1$. We have the following commutative diagram:
$$\begin{tikzcd}
\mathsf{vBT}_{2g} \arrow{r}{\widecheck{Th}} \arrow{d}[swap]{Z_A} & \left\lbrace \parbox[c]{6.5em}{\centering   {\small  \textnormal{genus $\leq g$ 0-framed knots} }}  \right\rbrace \arrow{d}{Z_A}\\
A^{\otimeshat 2g } \arrow{r}{Th} & A 
\end{tikzcd}  $$
\end{proposition}
\begin{proof}
We will actually show a more general claim. Namely, we aim to construct a commutative diagram 
\begin{equation}\label{eq comm diagram thickening}
\begin{tikzcd}
[column sep={7em,between origins}]
\mathsf{vBT}_{2g} \arrow{r}{\mathrm{cross}} \arrow{d}[swap]{Z_A} & \T^{(1)} \arrow{r}{\check{\Delta}} \arrow{d}[swap]{Z_A} & \T^{(2)}  \arrow{r}{\check{S}} \arrow{d}[swap]{Z_A} & \T^{(3)}  \arrow{r}{\mathrm{cups}} \arrow{d}[swap]{Z_A}  & \left\lbrace \parbox[c]{6.5em}{\centering   {\small  \textnormal{genus $\leq g$ 0-framed knots} }}  \right\rbrace   \arrow{d}[swap]{Z_A} \\
A^{\otimeshat 2g } \arrow{r}{\widetilde{Cr}^{\otimeshat g}} & A^{\otimeshat 2g }  \arrow{r}{\Delta^{\otimeshat 2g}} & A^{\otimeshat 4g } \arrow{r}{(S_{\capr} \otimeshat \id )^{\otimeshat 2g}} & A^{\otimeshat 4g } \arrow{r}{\widetilde{\mu}} & A
\end{tikzcd} 
\end{equation}
such that the top row composite is the geometric thickening map $\widecheck{Th}$ and the bottom row composite is the algebraic thickening map $Th$.

Let $L= L_1 \cup \cdots \cup L_{2g}$ be a vertical bottom tangle. The main observation is that $\widecheck{Th}(L)$ is isotopic to the following knot obtained by dragging the caps to the bottom:
$$
\labellist \small \hair 2pt
\pinlabel{$\widecheck{Th}(L) =$}  at -260 710
\pinlabel{$\vdots$}  at 400 740
\pinlabel{$\cdots$}  at 1150 10
\endlabellist
\centering
\includegraphics[width=0.45\textwidth]{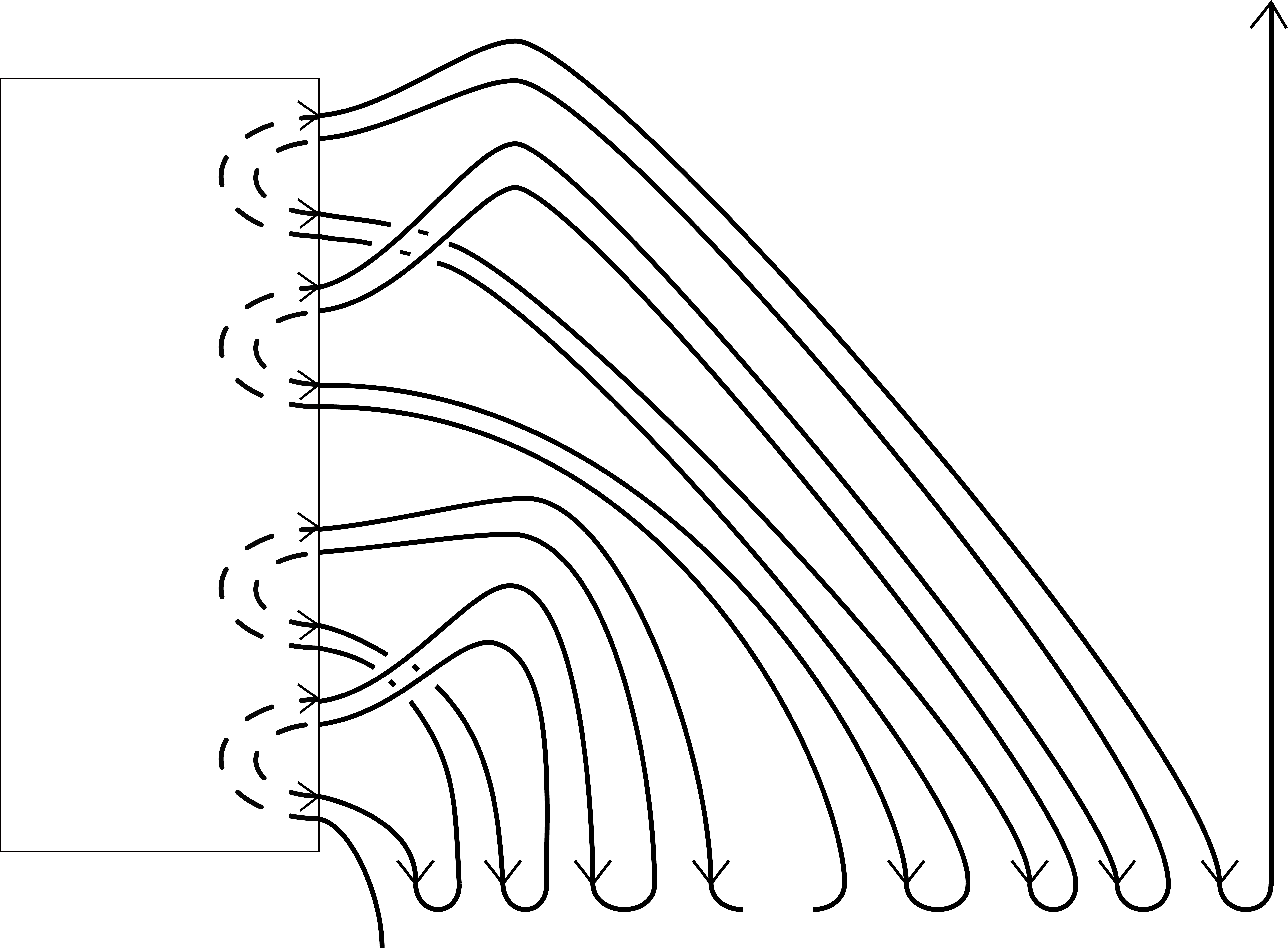}
$$
From this form it can be seen that $\widecheck{Th}(L)$ arises from $L$ as follows: first add crossings and bring the endpoints to the bottom part of the plane to create an element $L^{(1)}$ of $\hom{\mathcal{T}}{(++--)^g}{\emptyset}$, as depicted below:

$$
\labellist \small \hair 2pt
\pinlabel{$L^{(1)} =$}  at -210 640
\pinlabel{$\vdots$}  at 400 645
\pinlabel{$\cdots$}  at 1150 20
\endlabellist
\centering
\includegraphics[width=0.45\textwidth]{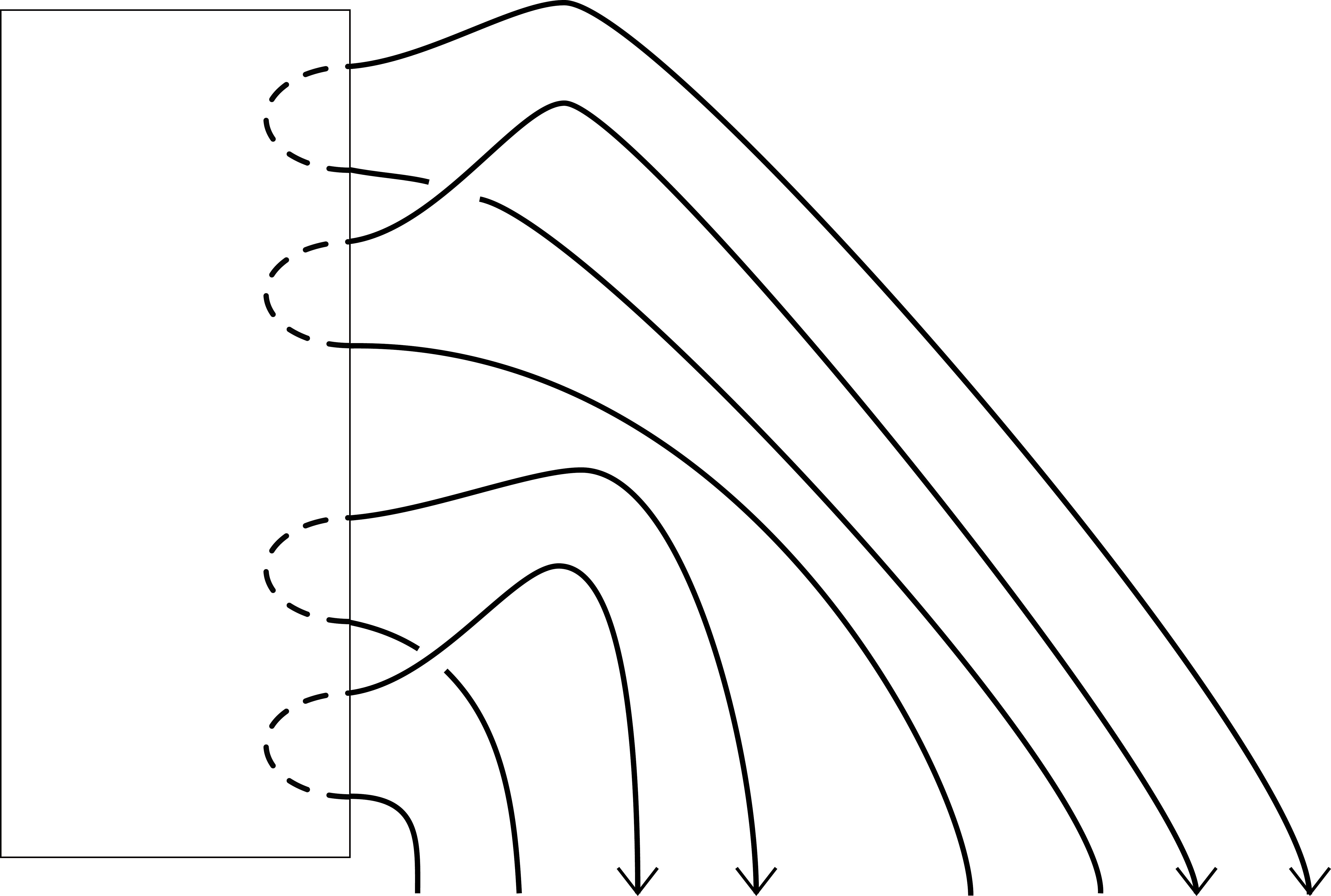}
$$

Next double every strand of $L^{(1)}$, that is, apply  $\check{\Delta}_\ell$ to every component $\ell$ of $L^{(1)}$ to obtain a $4g$-component tangle in $\hom{\mathcal{T}}{(+^4 -^4)^g}{\emptyset}$, denoted $L^{(2)}$. Further reverse the orientation of every component labelled with an odd integer applying $\check{S}_\ell$ to such every component, to obtain a tangle in $\hom{\mathcal{T}}{(+-)^{2g}}{\emptyset}$ as shown below,
$$
\labellist \small \hair 2pt
\pinlabel{$L^{(3)} =$}  at -210 640
\pinlabel{$\vdots$}  at 400 645
\pinlabel{$\cdots$}  at 1150 20
\endlabellist
\centering
\includegraphics[width=0.45\textwidth]{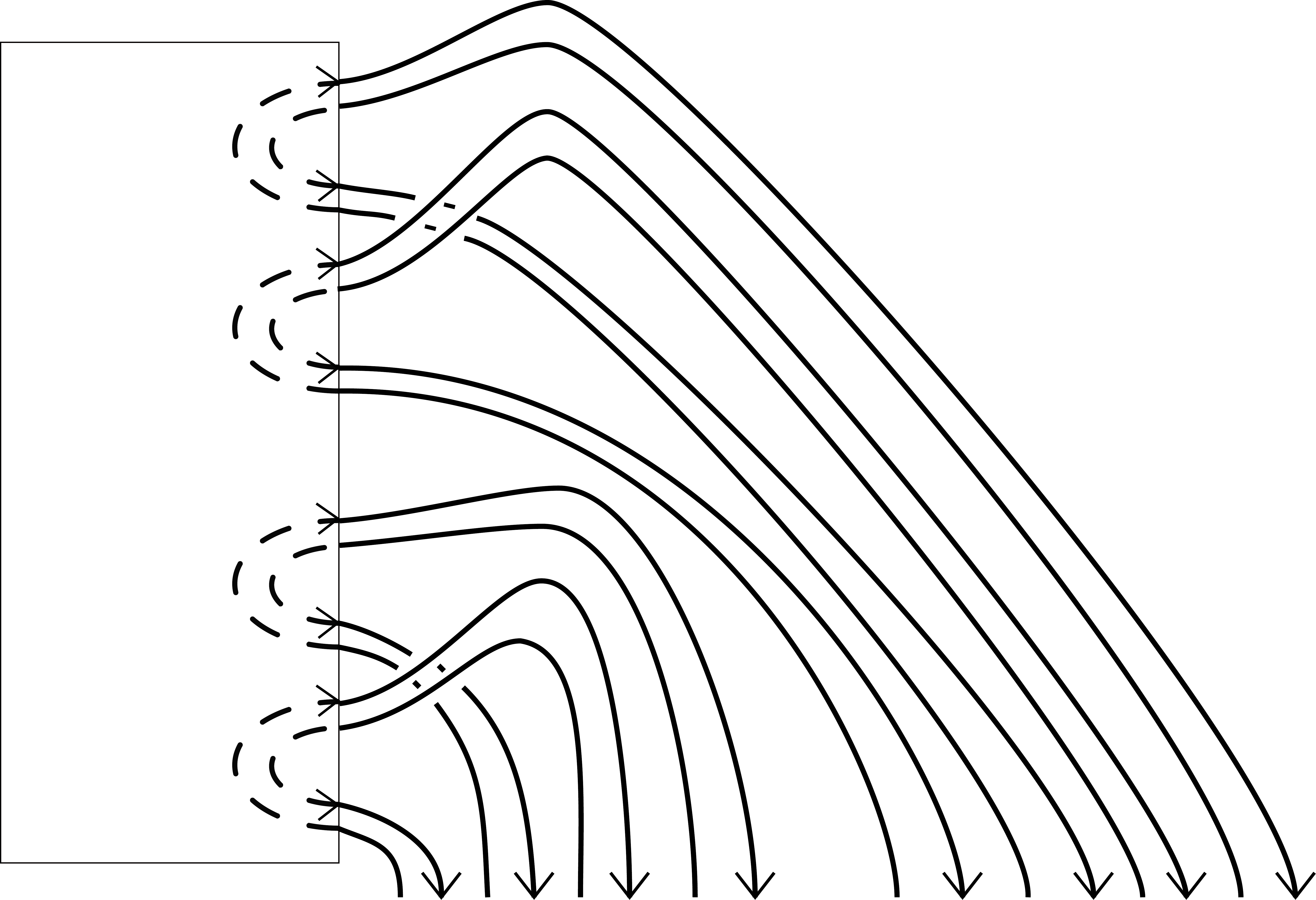}
$$
where the order of the components have been also indicated. Then the knot $\widecheck{Th}(L)$ is obtained from $L^{(3)}$ by attaching $4g$ cups to the bottom, that is $$\widecheck{Th}(L)= (L^{(3)}\otimes \uparrow_+) \circ (\uparrow_+ \otimes \cupr \otimes \overset{4g}{\cdots} \otimes \cupr).   $$
If $\T^{(i)}$ denotes the set of all such $L^{(i)}$ built out of vertical bottom tangles, then the previous description realises the geometric thickening map as top row composite  of the diagram \eqref{eq comm diagram thickening}.

The universal invariant $Z_A (\widecheck{Th}(L))$ can be then computed using the naturality properties of \cref{prop naturality Z_A}. Indeed if we set $\widetilde{Cr}(x \otimeshat y) := [Cr (x \otimeshat y)] \cdot (\kappa^{-1} \otimeshat \kappa^{-1})$  by definition we have that $$Z_A (L^{(1)}) = \widetilde{Cr}^{\otimeshat g} (Z_A(L)) .$$

Now \cref{prop naturality Z_A} implies at once the commutativity of the three leftmost squares of \eqref{eq comm diagram thickening}, where we put $S_{\capr} (x) = \kappa^{-1}S(x)$. Now  for $n \geq 2$ write 
 $$(\mu^{\cupr} \,  )^{[n]}(x_1 \otimeshat \cdots \otimeshat x_n) := x_1   \kappa  x_2  \kappa   \cdots   \kappa  x_n \in A,$$ 
and set  $\tau_{(123)}$ is as in \S \ref{subsec relation with classical quantum invariants} for the permutation $(123) \in \Sigma_4$. Then the map $\widetilde{\mu}$ appearing at the bottom right of  \eqref{eq comm diagram thickening} is defined as the following composite:
$$\begin{tikzcd}
[column sep={7em,between origins}]
A^{\otimeshat 4g} \rar{\tau_{(123)}^{\otimeshat g}} & A^{\otimeshat 4g} \rar{\id \otimes \eta} & A^{\otimeshat 4g +1} \rar{(\mu^{\cupr})^{[4g+1]}} & A
\end{tikzcd}$$
It then follows from \cref{prop naturality Z_A} that the rightmost square of \eqref{eq comm diagram thickening} also commutes, keeping in mind that we had to permute the factors to change the order of the components before applying the multiplication.

What is left to check is that the bottom row composite equals the algebraic thickening map $Th$. It can be readily seen that the band map $B: A \otimeshat A \to A$ equals the following composite:
$$\begin{tikzcd}
[column sep={5.5em,between origins}]
A^{\otimeshat 2} \rar{(\cdot \kappa^{-1})^{\otimeshat 2}} & A^{\otimeshat 2} \rar{\Delta \otimeshat \Delta} & A^{\otimeshat 4}  \rar{(S_{\capr} \otimeshat \id )^{\otimeshat 2}} & A^{\otimeshat 4} \rar{\tau_{(123)}} & A^{\otimeshat 4} \rar{\id \otimeshat \eta} & A^{\otimeshat 5} \rar{(\mu^{\cupr} \, )^{[5]}} & A
\end{tikzcd}$$
Therefore the algebraic thickening map equals the composite
$$ \mu^{[g]} \circ  [ (\mu^{\cupr} \, )^{[5]} ]^{\otimeshat g} \circ (\id \otimeshat \eta)^{\otimeshat g} \circ \tau_{(123)}^{\otimeshat g} \circ (S_{\capr} \otimeshat \id )^{\otimeshat 2g} \circ \Delta^{\otimeshat 2g} \circ  (\cdot \kappa^{-1})^{\otimeshat 2g} \circ Cr^{\otimeshat g}. $$
Noting that $\mu^{[g]} \circ  [ (\mu^{\cupr} \, )^{[5]} ]^{\otimeshat g} \circ (\id \otimeshat \eta)^{\otimeshat g} \circ \tau_{(123)}^{\otimeshat g} = \widetilde{\mu}$, the claim follows directly.
\end{proof}

In particular, the previous result claims that to understand the universal invariant of a knot, it suffices to understand the universal invariant of a vertical bottom tangle whose thickening is the knot. We will heavily rely on this fact to derive properties about the invariant $Z_\D$ for knots.

\begin{remark}
Despite we built the commutative diagram \eqref{eq comm diagram thickening} as explained in the proof above, we were not able to define maps between classes of tangles that were precisely the topological counterparts of the three maps whose composite define the algebraic thickening map (hence giving rise to a more natural commutative diagram). One should regard this as a deficiency of the notion of tangle that we are using. One way to fix this is to allow virtual crossings, see \cite{becerra}.
\end{remark}

For the rest of the section, we will be arguing with generating series. The following lemmas will prove to be extremely useful. We write the statements for free $k$-modules for simplicity, but trivially all of them hold for topologically free $k[[h]]$-modules as well.

\begin{lemma}\label{lemma y_i=0}
Let 
$$\begin{tikzcd}
 k[x_1, \ldots , x_n] \rar{f} &  k[y_1, \ldots , y_m] \rar{g} & k[z_1, \ldots , z_p]
\end{tikzcd}$$
be $k$-module maps, and suppose that the generating series $\bar{g}$ of $g$ has no $\eta_i$ for some $i=1, \ldots, m$.  If $\bar{f}_{| y_i =0}$ denotes the evaluation of the generating series $\bar{f}$ of $f$ at $y_i=0$, then we have
\begin{equation}\label{eq in lemma y_i=0}
\langle \bar{f}  \cdot  \bar{g}  \rangle_y = \langle \bar{f}_{| y_i =0}  \cdot  \bar{g} \rangle_y 
\end{equation}

Similarly, if $\bar{f}$ has no $y_i$, then we have 
\begin{equation}\label{eq in lemma eta_i=0}
\langle \bar{f}  \cdot  \bar{g}  \rangle_y = \langle \bar{f} \cdot  \bar{g}_{| \eta_i =0}  \rangle_y 
\end{equation}
(and the same is true for maps between free $k[[h]]$-modules).
\end{lemma}
\begin{proof}
If $\bar{g}$ has no $\eta_i$, then any $y_i$ in $\bar{f}$ will be paired with $1=\eta_i^0$ yielding $\langle y_i^k \eta_i^0 \rangle_{y_i}=0$ for $k>0$. Hence we might just set $y_i=0$ directly in $\bar{f}$. The second statement is similar.
\end{proof}
\begin{proof}[Alternative proof.]
Unravelling definitions, that $\bar{g}$ has no $\eta_i$ is equivalent to saying that the ideal $\mathfrak{a} \subset k[y_1, \ldots , y_m]$ generated by $y_i$ lies in the kernel of $g$. Besides, the $k$-module map corresponding with $\bar{f}_{| y_i =0}$ is the map $\widetilde{f}$ which is identically zero on $f^{-1}(\mathfrak{a})$ and coincides with $f$ elsewhere. Trivially $g \circ f = g \circ \widetilde{f}$ as $k$-module maps, but this is equivalent to \eqref{eq in lemma y_i=0}, which concludes the first part.

For the second part,  $\bar{f}$ having no $y_i$ means that $\mathfrak{a} \cap \im f =0$; and replacing $\bar{g}$ by $\bar{g}_{| \eta_i =0}$ amounts to replace $g$ by $\widetilde{g}$ which is identically zero on $\mathfrak{a}$ and coincides with $g$ elsewhere. The result follows as above.
\end{proof}

\begin{notation}
As stated above, we will usually blur the distinction between a $k$-module map as in the lemma and its generating series, so that we will mostly rewrite  \eqref{eq in lemma y_i=0} and \eqref{eq in lemma eta_i=0} as
$$g \circ f = g \circ f_{|y_i=0} \qquad \mathrm{and} \qquad g \circ f = g_{| \eta_i=0} \circ f,$$
respectively.
\end{notation}

The following lemma is shown using the same argument as before:

\begin{lemma}\label{lemma z_i=0}
Let $f,g $ be $k$-module maps as above. Then $$(g \circ f)_{| z_i=0} = g_{|z_i =0} \circ f$$ and similarly $$(g \circ f)_{| \xi_i=0} = g \circ f_{| \xi_i=0}.$$
\end{lemma}

\begin{lemma}\label{lemma h f g}
Let
$$\begin{tikzcd}
k[w_1, \ldots , w_q] \rar{h} & k[x_1, \ldots , x_n] \rar{f} &  k[y_1, \ldots , y_m] \rar{g} & k[z_1, \ldots , z_p]
\end{tikzcd}$$
be a diagram of $k$-module maps. 
\begin{enumerate}
\item If the generating series of $f$ has no $\xi_i$, neither does the generating series of $g \circ f$.
\item If the generating series of $f$ has no $\xi_i$, and the ideal generated by $w_j$ is mapped under $h$ to the ideal generated by $x_i$, then  the generating series of $f \circ h$  has no $\omega_j$ (the dual variable of $w_j$).
\end{enumerate}
\end{lemma}
\begin{proof}
If $\mathfrak{a}$ denotes the ideal generated by $x_i$, then (1) follows from $\mathfrak{a} \subseteq \ker f \subseteq \ker (g \circ f)$. Similarly, if $\mathfrak{b}$ denotes the ideal generated by $w_j$, then (2) follows from $h(\mathfrak{b}) \subseteq \mathfrak{a} \subseteq \ker f$, that is $\mathfrak{b} \subseteq \ker (f \circ h)$.
\end{proof}

For the topological ribbon $\Qeh$-algebra $\D$, the generating series for the algebraic thickening map has two remarkable properties. 

\begin{proposition}\label{prop B tau free}
The generating series $B_{i,j}^k$ of the band map is $\tau$-free, that is, it has no monomials $\tau_i$ or $\tau_j$. Furthermore, the generating series $Th_{i_1, \ldots , i_{2g}}^{j}$ for the algebraic thickening map is also $\tau$-free.
\end{proposition}
\begin{proof}
For the first part, it suffices to see that the two-sided ideals generated by $\bm{t} \otimeshat 1$ and $1 \otimeshat \bm{t}$ belong to the kernel of $B: \D \otimeshat \D \to \D$.  If $\bm{u}  \otimeshat \bm{v} \in \D \otimeshat \D$, then we have
\begin{align*}
B(\bm{tu}  \otimeshat \bm{v}) &= \sum_{(\bm{t}),(\bm{u}), (\bm{v})} \bm{t}_{(2)} \bm{u}_{(2)} S(\bm{v}_{(1)}) \bm{\kappa} S(\bm{t}_{(1)}\bm{u}_{(1)}) \bm{\kappa} \bm{v}_{(2)} \\
 &= \sum_{(\bm{u}), (\bm{v})} \Big(  \bm{u}_{(2)} S(\bm{v}_{(1)}) \bm{\kappa} S(\bm{t}\bm{u}_{(1)}) \bm{\kappa} \bm{v}_{(2)} +\bm{t} \bm{u}_{(2)} S(\bm{v}_{(1)}) \bm{\kappa} S(\bm{u}_{(1)}) \bm{\kappa} \bm{v}_{(2)}   \Big) \\
 &=  \sum_{(\bm{u}), (\bm{v})} \Big( - \bm{u}_{(2)} S(\bm{v}_{(1)}) \bm{\kappa} S(\bm{u}_{(1)}) \bm{t} \bm{\kappa} \bm{v}_{(2)} +\bm{t} \bm{u}_{(2)} S(\bm{v}_{(1)}) \bm{\kappa} S(\bm{u}_{(1)}) \bm{\kappa} \bm{v}_{(2)}   \Big) \\
 &= 0
\end{align*}
since $\bm{t}$ is central. For the ideal generated by $1 \otimeshat  \bm{t}$ the computation is similar.

For the second part, note that $Cr: \D \otimeshat \D \to \D \otimeshat \D$ trivially maps each of the two ideals to themselves, so the claim follows directly from \cref{lemma h f g}.
\end{proof}

Abusing notation we will usually write $Z_\D (K)$ for $\mathcal{O}^{-1}(Z_\D (K))$. Let us also write $Z_\D (L)_{|t=0}$ for $Z_\D (L)_{|t_{i_1}=\cdots = t_{i_{2g}}= 0}$.

\begin{corollary}\label{cor Th L}
Let $K= \widecheck{Th}(L)$ be a 0-framed knot. Then the value $Z_\D(K)$ is completely determined by $Z_\D(L)_{| t=0}$, $$Z_\D(K) = Th (Z_\D(L)_{| t=0}).$$
\end{corollary}

The previous result turns out to be extremely powerful, since it simplifies significantly the amount of terms involved in computations and the complexity of proofs. 

Here is the second property:

\begin{theorem}
If $n\geq 0$ and $j\geq 1$, then the image of the band map $B$ contains no monomials with $\bm{a}^{n+j}\varepsilon^n$. 

More precisely, $$\mathrm{span}_{\Q[[h]]} \{ \bm{y}^p \bm{t}^q \bm{a}^{n+j} \bm{x}^r \varepsilon^n : p,q,r \geq 0   \} \cap \im B =0.$$
\end{theorem}
\begin{proof}
Let us start by introducing some terminology: say that an elementary  monomial $\bm{y}^{*} \bm{t}^{*} \bm{a}^{p} \bm{x}^{*} \varepsilon^q h^{*}\in \D$ \textit{satisfies the property $(\bigstar)$}  if $p \leq q$, where $*$ represents arbitrary non-negative integers. More generally, for a $k$-fold tensor product, say that $$\bm{y}^{*} \bm{t}^{*} \bm{a}^{p_1} \bm{x}^{*} \otimeshat \cdots \otimeshat \bm{y}^{*} \bm{t}^{*} \bm{a}^{p_k} \bm{x}^{*} \varepsilon^q h^{*} \in \D^{\otimeshat k}$$ satisfies $(\bigstar)$ if $\sum_i p_i \leq q$. Extend the definition of $(\bigstar)$ to arbitrary elements of $\D^{\otimeshat k}$ by $\Q$-linearity. Then the statement of the theorem claims that all elements of $\im B$ satisfy $(\bigstar)$.

Now the key observation is that given elements $\bm{u}_1, \ldots , \bm{u}_n \in \D^{\otimeshat k}$, the product $\bm{u}_1 \cdots \bm{u}_n$ satisfies $(\bigstar)$ if and only if any permuted product $\bm{u}_{\sigma(1)} \cdots \bm{u}_{\sigma (n)}$ satisfies $(\bigstar)$ for any $\sigma \in \Sigma_n$. This is because the relations defining $\D$ \textit{preserve} the property $(\bigstar)$, meaning that in 
$$ \bm{a} \bm{x} = \bm{x} \bm{a} + \bm{x} \quad , \quad   \bm{a} \bm{y} = \bm{y} \bm{a} - \bm{y} \quad , \quad   \bm{x} \bm{y} = \sum_i \lambda_i \bm{y} \bm{x} \varepsilon^i h^i + \sum_{i,j} \mu_{i,j} \bm{t}^i \bm{a}^j \varepsilon^j h^{i+j-1}  $$ both sides of the equations satisfy the property $(\bigstar)$, where here $\lambda_i, \mu_{ij} \in \Q$. Also note that $\bm{u} \in \D$ satisfies $(\bigstar)$ if and only if $\bm{u} \bm{\kappa}$ satisfies $(\bigstar)$, where $\bm{\kappa} = \sum_i \frac{(-1)^i}{i!} \bm{T}^{1/2} \bm{a}^i \varepsilon^i h^i  $ denotes as usual the balancing element. One implication is straightforward since if two elements satisfy $(\bigstar)$, so will their product. For the converse, note that if $\bm{u}$ has a monomial $\bm{y}^* \bm{t}^* \bm{a}^{n+j} \bm{x}^* \varepsilon^n h^*$ with $j \geq 1$, then $\bm{u} \bm{\kappa}$ will contain $\bm{y}^* \bm{t}^* \bm{a}^{n+j+i} \bm{x}^* \varepsilon^{n+i} h^*$.

This being said, we make the following claim: if $\bm{u}\otimeshat \bm{v}, \bm{u}'\otimeshat \bm{v}' \in \D \otimeshat \D$ are such that $B(\bm{u}\otimeshat \bm{v})$ and $B(\bm{u}'\otimeshat \bm{v}')$ satisfy $(\bigstar)$, then so will the image of the product $B(\bm{u} \bm{u}'\otimeshat \bm{v}\bm{v}')$. Indeed by the properties showed above the element
$$ B(\bm{u} \bm{u}'\otimeshat \bm{v}\bm{v}') = \sum_{(\bm{u}),( \bm{u}'),( \bm{v}),(\bm{v}')} \bm{u}_{(2)} \bm{u}'_{(2)} S(\bm{v}'_{(1)}) S(\bm{v}_{(1)}) \bm{\kappa}   S(\bm{u}'_{(1)}) S(\bm{u}_{(1)})  \bm{\kappa} \bm{v}_{(2)} \bm{v}'_{(2)}   $$ satisfies $(\bigstar)$ if and only if 
$$ \sum_{(\bm{u}),( \bm{u}'),( \bm{v}),(\bm{v}')} \bm{u}_{(2)}  S(\bm{v}_{(1)}) \bm{\kappa}    S(\bm{u}_{(1)})  \bm{\kappa} \bm{v}_{(2)}  \bm{u}'_{(2)} S(\bm{v}'_{(1)})S(\bm{u}'_{(1)}) \bm{v}'_{(2)} $$ satisfies $(\bigstar)$ if and only if 
\begin{align*}
\sum_{(\bm{u}),( \bm{u}'),( \bm{v}),(\bm{v}')} \bm{u}_{(2)}  S(\bm{v}_{(1)}) \bm{\kappa}    S(\bm{u}_{(1)})  \bm{\kappa} \bm{v}_{(2)}  \bm{u}'_{(2)} S(\bm{v}'_{(1)}) &  \bm{\kappa} S(\bm{u}'_{(1)})\bm{\kappa}  \bm{v}'_{(2)} \\ &= B(\bm{u}\otimeshat \bm{v})B(\bm{u}'\otimeshat \bm{v}')
\end{align*}
satisfies $(\bigstar)$, which is true, which concludes the proof of the claim.

Let us conclude the proof of the theorem. According to the previous claim, it suffices to check that the images under $B$ of the elements $\bm{y} \otimes 1$, $\bm{a} \otimes 1$, $\bm{x} \otimes 1$, $1 \otimes  \bm{y}$, $1 \otimes  \bm{y}$, $1 \otimes  \bm{a}$, $1 \otimes  \bm{x}$ satisfy $(\bigstar)$. It is readily verified that
$$\Delta (\bm{a} \otimes 1)= \Delta (1 \otimes \bm{a}) = -\bm{a} \bm{\kappa}^2 + \bm{a} \bm{\kappa}^2=0$$ which trivially satisfies $(\bigstar)$. For the other four elements we argue as follows: say that a map $f: \D^{\otimes n} \to \D^{\otimes m}$\textit{ preserves the property} $(\bigstar)$ if $f(\bm{w})$ satisfies $(\bigstar)$ whenever $\bm{w}$ satisfies  $(\bigstar)$. Clearly the composite of  $(\bigstar)$-preserving maps is again  $(\bigstar)$-preserving, so the argument boils down to show that the comultiplication and the antipode preserve $(\bigstar)$. Since they are algebra (anti)homomorphisms, it really suffices to check the generators. The equalities
\begin{align*}
\Delta(\bm{x}) &= \sum_i \pm (1 \otimeshat \bm{x} + \bm{x}\otimeshat \bm{a}^i \varepsilon^i h^i) \qquad , \qquad S(\bm{x})=-\sum_i \bm{x} \bm{a}^i \varepsilon^i h^i\\
\Delta(\bm{y}) &= \sum_i \pm (1 \otimeshat \bm{y} + \bm{y} \otimes \bm{t}^* \bm{a}^i \varepsilon^i h^i) \qquad , \qquad S(\bm{y}) = - \sum_i \bm{y} \bm{t}^* \bm{a}^i \varepsilon^i h^i
\end{align*}
from \cref{prop Hopf structure D} ensure that $\Delta$ and $S$ preserve $(\bigstar)$, which concludes the proof of the theorem.
\end{proof}

\begin{corollary}\label{cor Th has no a epsilon}
If $n\geq 0$ and $j\geq 1$, then the image of the algebraic thickening map $$Th: \D^{\otimeshat 2g} \to \D,$$ $g\geq 1$, contains no monomials with $\bm{a}^{n+j}\varepsilon^n$. 
\end{corollary}
\begin{proof}
Using the terminology of the previous proof, all elements in $\im B$ satisfy $(\bigstar)$ and hence all elements of $\im B^{\otimeshat g}$ will also satisfy $(\bigstar)$, and then so will all elements in $\im (B^{\otimeshat g} \circ Cr^{\otimeshat g})$. Since the multiplication preserves the property $(\bigstar)$, we conclude.
\end{proof}

\subsection{The universal invariant  \texorpdfstring{$Z_\D \pmod {\varepsilon}$}{Z\_D mod eps3}  }


Recall that if $K$ is a closed knot in $S^3$ and $F$ is a genus $g$ Seifert surface for $K$, the \textit{Seifert matrix} of $K$ (associated to $F$) is the $2g \times 2g$ matrix $V=(v_{ij})$ with integer coefficients given by $v_{ij} := \mathrm{lk}(\gamma_i, \gamma_j^+)$, where $\gamma_1, \ldots , \gamma_{2g}$ are simple closed curves such that  $H_1 (F; \Z) \cong \Z \gamma_1 \oplus \cdots \oplus \Z \gamma_{2g}$ and $\gamma_j^+$ represents the positive  push-off of $\gamma_j$. Then the Alexander polynomial of $K$ is given by
\begin{equation}
\Delta_K (t)= \det (t^{-1/2}V-t^{1/2}V^T) \in \Z [t+t^{-1}] \subset  \Z [t,t^{-1}].
\end{equation}
In particular, if $K= \widecheck{Th}(L)$ (so it comes together with a choice of Seifert surface $F$), then there is a canonical set of simple closed curves $\gamma_1, \ldots , \gamma_{2g}$ generating $H_1 (F; \Z)$ that is completely determined by $L$, namely the cores of the bands arising from the thickening. In this situation, it is easy to obtain the Seifert matrix associated to $K$ (and $F$). Namely, let $p_{ij}$ (resp. $n_{ij}$) denote the number of positive (resp. negative) crossings in $L$ in which the overstrand is the $i$-th component and the understrand is the $j$-th component, and put $V'=(v_{ij}')$ with $v_{ij}':= p_{ij}-n_{ij}$. Then it is readily verified that the matrix
\begin{equation}
V:= V' +  \begin{pmatrix} 0 & 1 \\0 & 0 \end{pmatrix} \oplus \overset{g}{\cdots} \oplus  \begin{pmatrix} 0 & 1 \\0 & 0 \end{pmatrix}
\end{equation}
is the desired Seifert matrix for $K$, where $\oplus$ denotes block sum of matrices. Furthermore, the basis $\gamma_1, \ldots , \gamma_{2g}$ is symplectic with respect to the intersection form of $F$, and concretely this equals
\begin{equation}\label{eq V seifert surface}
V-V^T = \begin{pmatrix} 0 & 1 \\-1 & 0 \end{pmatrix} \oplus \overset{g}{\cdots} \oplus  \begin{pmatrix} 0 & 1 \\ -1 & 0 \end{pmatrix} =: Q.
\end{equation}

\begin{proposition}\label{prop exp V'}
Let $L$ be a $2g$ component vertical bottom tangle whose components are labelled by $\{1, \ldots , 2g  \}$, let $K= \widecheck{Th}(L)$  and let  $V'$ be as above. Then $$Z_\D (L)_{|t=0} =  \exp \Big[  yV'x \cdot h \Big] \pmod {\varepsilon},$$ where $y=(y_1, \ldots , y_{2g})$ and $x=(x_1, \ldots, x_{2g})^T$.
\end{proposition}
\begin{proof}
First off, by \cref{lemma vbt rot form} we can think of $L$ as some combination of crossings and spinners (cf. \eqref{eq crossings and spinners}), so that by \eqref{eq forgetful represented} the universal invariant of $L$ can be thought of as a composition 
$$\begin{tikzcd}
[column sep={9em,between origins}]
\Qeh \rar & \D^{\otimeshat N} \rar{\tau_{\sigma}} &  \D^{\otimeshat N} \rar{\mu^{[n_1]} \otimeshat \cdots \otimeshat \mu^{[n_{2g}]}} & \D^{\otimeshat 2g}
\end{tikzcd}$$
where the first arrow maps 1 to a tensor product of copies of $R^{\pm 1}$ and $\kappa^{\pm 1}$ and the second one rearranges the tensor factors, and all arrows are thought $\pmod \varepsilon$. Under the isomorphism $\mathcal{O}$, we can view this as a composite
$$\begin{tikzcd}
\Qeh \rar & \Qe [y_I, t_I, a_I, x_I][[h]]   \rar{\mu_I^J}  & \Qe [y_J, t_J, a_J, x_J][[h]]
\end{tikzcd}$$
where $I=\{ i_1, \ldots i_{N} \}$, $J = \{1, \ldots , 2g  \}$ and $\mu_I^J$ is a composite of generating series of the multiplication map,
$$\mu_I^J = ( \mu_{k_r, i_N}^{2g}  \circ  \cdots ) \circ   \cdots  \circ (\mu_{k_p, i_q}^1 \circ   \cdots   \circ   \mu_{k_1, i_3}^{k_2}   \circ  \mu_{i_1, i_2}^{k_1} ). $$
Since we are only interested in the resulting composite at $t=0$, by \cref{lemma z_i=0} we can replace  $\mu_I^J$ by $(\mu_I^J)_{| t=0}$. Repeating the argument, we can equivalently replace  $\mu_{k_r, i_N}^{2g}$ by $ (\mu_{k_r, i_N}^{2g})_{|t=0}$. Looking at the generating series $\pmod \varepsilon$ of the multiplication from \cref{thm DoPeGDO}, we see that $ (\mu_{k_r, i_N}^{2g})_{|t=0}$ is $\tau$-free, so by \cref{lemma y_i=0} we could replace the rest of the composite  by its evaluation at $t=0$. Iterating the argument, we see that replacing each of the multiplication maps appearing in $\mu_I^J$ by its evaluation at $t=0$ produces $(\mu_I^J)_{| t=0}$.

In particular, the above argument also shows that $(\mu_I^J)_{| t=0}$ is $\tau$-free, so that again by \cref{lemma y_i=0} we can replace all the copies of $R^{\pm 1}$ by its evaluations at $t=0$ and simply ignore the spinners since $(\kappa^{\pm 1})^i_{|t_i=0} =1 \pmod \varepsilon$. Observe then that $$(R^{\pm 1})^{ij}_{|t=0} =  \exp [ \pm (y_i x_j)h  ]$$ and in particular is $a$-free, so that applying recursively \cref{lemma y_i=0}  again we can use instead  $(\mu_{i,j}^k)_{|t=\alpha=0}$ in the composite defining $\mu_I^J$. But this generating series equals the one for the commutative multiplication, see \cref{remark commutative gen series}, for which product means relabelling with the same index. A positive crossings (with overstrand $i$) adds up one for $p_{ij}$, whereas a negative crossing (with overstrand $i$) adds up one for $n_{ij}$.  The claim then follows.
\end{proof}

%
%


The previous Proposition allow us to show that $Z_\D(K)$ (mod $\varepsilon$) is tantamount to the Alexander polynomial.

\begin{proposition}[\cite{barnatanveengaussians}]\label{prop Z_D for eps ^ 0 and 1}
Let $K$ be a $0$-framed knot. Then 
$$Z_\D(K)= \frac{1}{\Delta_K (\bm{T})} \pmod  \varepsilon .$$
\end{proposition}

\begin{proof}
Write $K= \widecheck{Th}(L)$ for some $2g$-com\-po\-nent vertical bottom tangle $L$. We aim to compute $Z_\D (K) = Th (Z_\D (L))$ by studying the composite
$$
\begin{tikzcd}
[column sep={7em,between origins}]
\Qeh \rar{Z_\D (L)} & \D^{\otimeshat 2g} \rar{Cr^{\otimeshat g}} & \D^{\otimeshat 2g} \rar{\mu^{[g]} \circ B^{\otimeshat g}} & \D
\end{tikzcd}
$$ in this fashion.

The map $Cr^{\otimeshat g}$ consists of multiplications and copies of $R$. The same argument as the one given in   \cref{prop exp V'} shows that if $V$ is given by \eqref{eq V seifert surface}, that is, $V$ is the Seifert matrix of $K$ associated to the thickening, then $$Cr^{\otimeshat g} (Z_\D (L))_{|t=0} =  \exp \Big[  yVx \cdot h \Big] \pmod {\varepsilon}.$$

Now since $B$ is $\tau$-free, \cref{lemma h f g} ensures that so is $\mu^{[g]} \circ B^{\otimeshat g}$, hence it suffices to use the evaluation at $t=0$ of $Cr^{\otimeshat g} (Z_\D (L))$ above. But  $Cr^{\otimeshat g} (Z_\D (L))_{| t=0}$ is $a$-free mod $\varepsilon$, which means that we use $(\mu^{[g]} \circ B^{\otimeshat g})_{| \alpha =0}$ instead by \cref{lemma y_i=0}. Using \cref{cor gen funct mu circ B}, we have
\begin{align}
Z_\D (K) &=  (\mu^{[g]} \circ B^{\otimeshat g})_{| \alpha =0} \circ Cr^{\otimeshat g} (Z_\D (L))_{| t=0} \nonumber \\
&=  \left\langle   T_k^g \exp \left[  yVx \cdot h  +  \frac{T_k-1}{h} \cdot \xi Q \eta \right]    \right\rangle_{1, \ldots , 2g} \pmod \varepsilon.\label{eq 1/Alex contraction}
\end{align}
To perform the previous composition, we use the contraction \cref{thm contraction} with
$$ u=v=0 \quad , \quad r= ( y \ \xi ) \quad , \quad s =(\eta \ x) \quad , \quad W=  \begin{pmatrix} 0 & V h \\\frac{T_k -1}{h} Q & 0 \end{pmatrix},$$ which yields
\begin{align*}
Z_\D (K) &= T^g \det (I-W)^{-1} =   T^g \det (I - (T-1)VQ)^{-1} \\
&= T^g \det (-Q^2 -(T-1)VQ)^{-1} =  T^g \det (Q) \det (-Q -(T-1)V)^{-1}\\ &= T^g \det ( -V+V^T - (T-1)V )^{-1} =  T^g \det ( V^T - TV )^{-1} \\ &=  (T^{-1/2})^{-2g} \det ( V - TV^T )^{-1} = \det (T^{-1/2} V - T^{1/2} V^T ) \\ &= \Delta_K (T)^{-1} \pmod \varepsilon
\end{align*}
(we suppressed the index $k$), where we have used that $\det (Q)=1$, $Q^2 = -I$, $V-V^T=Q$ and that for any square matrices $A,B$ of the same order, the equality $$ \det \begin{pmatrix}
I & A \\ B & I
\end{pmatrix} = \det (I - AB)$$ holds.
\end{proof}

\subsection{A normal form for $Z_\D$} We will now focus on the universal invariant of (long) knots. It is well-known (see e.g. \cite[Proposition 8.2]{habiro}) that the universal invariant $Z_A(K)$ of a knot $K$ takes places in the centre $\mathcal{Z}(A)$ of the ribbon Hopf algebra $A$. 

We can fully  determine the centre of the algebra $\D$.    


\begin{proposition}[\cite{barnatanveengaussians}]\label{prop w}
Suppose that $\bm{w}' \in \mathcal{Z}(\D)$ and satisfies $$\bm{w}' = \bm{yx} + g(\bm{a}, \bm{t}) \pmod h$$ for some $g \in \Qe \langle  X,Y \rangle$. Then $\mathcal{Z}(\D)$ is generated, as a topological $\Qe [[h]]$-subalgebra, by the elements $\bm{t}$ and $\bm{w}'$.

In particular, the element $$\bm{w} := \bm{y}\bm{A}^{-1}\bm{x} + \frac{q\bm{A}^{-1} + \bm{TA} - \frac{1}{2} (q+1)(\bm{T}+1)}{h(q-1)}$$ satisfies the hypothesis above.
\end{proposition}

\begin{remark}
The numerator in the statement is indeed multiple of $h(q-1)$. For its image\footnote{In the power series ring $\Q [[x]]$, the ideal generated by $x$ equals the ideal generated by $e^x-1$.} in $\D/ (q-1) \D = \D/ \varepsilon h \D  $ equals $1+\bm{T} - (\bm{T}+1) =0$ as $q$ and $\bm{A}^{\pm 1}$ are mapped to $1$. Applying the L'Hôpital rule to $(q\bm{A}^{-1} + \bm{TA} - \frac{1}{2} (q+1)(\bm{T}+1))/(q-1)$, which is possible since $\bm{t}$ is central, yields the claim.
\end{remark}

What the proposition says is that we should be able to express the universal invariant $Z_\D (K)$ of a knot $K$ in terms of $\bm{t}$ and $\bm{w}$. For knots with framing zero, following fundamental theorem says that we can express $Z_\D (K)$ in a particularly neat form.

For a knot $K$ we write $\Delta_K \in  \Z[t,t^{-1}]$ for the Alexander polynomial of (the closure of) $K$, normalised so that $\Delta_K (t^{-1})=\Delta_K (t)$ and $\Delta_K (1)=1$. 

\begin{theorem}[\cite{barnatanveengaussians}]\label{thm rhos existence}
For any $0$-framed knot $K$, there exist knot polynomial invariants $$\rho_K^{i,j} \in \Q [t,t^{-1}],$$ $i \geq 0, \ 0 \leq j \leq 2 i$, such that
$$ Z_\D (K)= \sum_{i=0}^\infty \left(  \sum_{j=0}^{2i}  h^{i+j} \frac{\rho_K^{i,j}(\bm{T})}{\Delta_K^{2i+1-j} (\bm{T})}   \bm{w}^j     \right) \varepsilon^i  $$
\end{theorem}

Note that \cref{prop Z_D for eps ^ 0 and 1} amounts to saying that $\rho_K^{0,0}=1$.

\subsection{Vanishing of the higher $\rho_K^{i,j}$}

\cref{thm rhos existence} states that studying the universal invariant $Z_\D(K)$ amounts to studying the collection of  Bar-Natan - van der Veen polynomials $\rho_K^{i,j}$ for $i \geq 0$, $0 \leq j \leq 2i$. The natural question then is to ask how strong these knot polynomial invariants are.  The polynomials $\rho_K^{i,0}$ are expected to be separate a large number of knots (see \cref{sect perpectives}). For $j>0$, we expect these polynomials to be either trivial or consequence of the ones for $j=0$ (and the Alexander polynomial). In fact, solving a conjecture in \cite{barnatanveengaussians}, we show below  that in fact the summation only runs through $j=0, \ldots , i$. 

First we need an explicit formula for the central element $\bm{w}$ mod $\varepsilon^2$.

\begin{lemma}\label{lemma w mod eps2}
Let $\bm{w} \in \D$ as in \cref{prop w}. Then we have
$$ \bm{w} = \frac{1-\bm{T}}{h} \left( \bm{a} + \frac{1}{2} \right) + \bm{yx} + \left[ \frac{1}{2} (1+\bm{T})(\bm{a}+ \bm{a}^2) + \bm{yax}   h \right] \varepsilon \pmod {\varepsilon^2}  .$$
\end{lemma}
\begin{proof}
One way to show this is by expanding the expression defining $\bm{w}$ in power series of $\varepsilon$. Alternatively, since $h(q-1) \in \varepsilon \D$ and $\bm{y}\bm{A}^{-1}\bm{x} = \bm{yx} + \bm{yax}h \pmod {\varepsilon^2}$, it suffices to show that 
\begin{align*}
q\bm{A}^{-1} + \bm{TA} - \frac{1}{2} (q+1)(\bm{T}+1) = h (q-1) \Bigg[ \frac{1-\bm{T}}{h} \left( \bm{a} + \frac{1}{2} \right) +\frac{1}{2} ( 1 + &\bm{T}) (  \bm{a}+ \bm{a}^2) \varepsilon   \Bigg]  \\  &\pmod {\varepsilon^3}
\end{align*}
what is readily verified expanding both sides in powers of $\varepsilon$.
\end{proof}

\begin{theorem}\label{thm rho ij =0 for j>i}
For any $0$-framed knot $K$, $$\rho_K^{i,j}(t)=0$$ for $j>i\geq 1$.
\end{theorem}
\begin{proof}

Let us start by noticing that it suffices to show that for a 0-framed knot $K$, the value $Z_\D (K)$ does not contain monomials with $\bm{a}^{n+p}\varepsilon^n$ for $p>0$. Indeed by \cref{lemma w mod eps2} we have that $$\bm{w}^j \varepsilon^i = \left( \frac{1-\bm{T}}{h} \right)^j \bm{a}^j \varepsilon^i + \text{(terms with lower powers of $\bm{a}$)} \pmod {\varepsilon^{i+1}} ,$$ so looking at the expression in \cref{thm rhos existence} we must conclude that $\rho_K^{i,j} =0$ for $j>i>0$.

As usual for a $2g$-component vertical bottom tangle $L$  whose strands are labelled by $1, \ldots , 2g$ let  us write $$(Z_\D (L)_{| t=0})^{1, \ldots , 2g} = e^{y V'x\cdot h}(P_0+P_1 \varepsilon + \cdots+  P_n  \varepsilon^n) \pmod {\varepsilon^{n+1}}$$
and 
$$ Th_{1, \ldots , 2g}^0 = e^{(*)}(P_0'+P'_1 \varepsilon + \cdots+  P'_n  \varepsilon^n) \pmod {\varepsilon^{n+1}} $$ where $(*)$ is an expression that does not contain any $a_0$, by \cref{cor Th has no a epsilon}. We then compute
\begin{align*}
Z_\D (K) &= Th_{1, \ldots , 2g}^0    \circ  (Z_\D (L)_{| t=0})^{1, \ldots , 2g}\\
&=  \left\langle   e^{y V'x\cdot h}(P_0+P_1 \varepsilon + \cdots+  P_n  \varepsilon^n) e^{(*)}(P_0'+P'_1 \varepsilon + \cdots+  P'_n  \varepsilon^n)  \right\rangle_{1, \ldots , 2g}\\
&= \left\langle  e^{y V'x\cdot h + (*)} P_0 P'_0 \right\rangle  + \cdots + \left\langle  e^{y V'x\cdot h + (*)} (P_0 P'_n  + \cdots + P_n P_0') \right\rangle \varepsilon^n\\
&\hspace*{9cm}  \pmod {\varepsilon^{n+1}}.
\end{align*}
To see that the contraction $$ \left\langle  e^{y V'x\cdot h + (*)} (P_0 P'_n  + \cdots + P_n P_0') \right\rangle_{1, \ldots , 2g}$$ does not have any term $a_0^{n+p}$, $p>0$, we argue as follows: any power of $a_0$ after the contraction comes either (1) because it was already in one of the $P_i'$ or (2) it is produced from one $a_\ell$, $\ell =1, \ldots , 2g$ when applying the Contraction \cref{thm contraction}. But if (1), according to \cref{cor Th has no a epsilon}, such a power of $a_0$ must be at most $n$. On the other hand, no $a_0$ can be created when applying the contraction formula because such a term would come from the exponential part and there is no $a_0$ in the exponential $e^{y V'x\cdot h + (*)}$ (cf. \cref{example contraction}), which concludes.
\end{proof}



\subsection{Some concrete values}
For a 0-framed knot $K$, let us put $\mathrm{Coeff}_{a_0 \varepsilon^2} (Z_\D(K))$ for the coefficient in $a_0 \varepsilon^2$ (independent of $a_0$ and $\varepsilon$) of  $Z_\D(K)^0$, where $0$ denotes the only label of $K$, that is, 
$$ Z_\D(K)^0 = \frac{1}{\Delta_K (T_0)} + (\cdots) \varepsilon + \left(\mathrm{Coeff}_{a_0 \varepsilon^2}  (Z_\D(K)) a_0 + \left( \parbox[c]{5em}{\centering  {\small  terms with $a_0^k, k\neq 1$ }} \right) \right)  \varepsilon^2 \pmod {\varepsilon^3} $$
and write $\mathrm{Coeff}_{ \mathbf{1} \varepsilon} (Z_\D(K))$ for the term, multiple of $\varepsilon^1$, that is independent of $y_0, a_0, x_0$.

\begin{proposition}\label{prop coef a esp2 = -2hT coeff 1eps}
    Let $K$ be a 0-framed knot. The we have
    $$ \mathrm{Coeff}_{a_0 \varepsilon^2} (Z_\D(K)) =   -2hT_0 \cdot \partial_{T_0}  \mathrm{Coeff}_{ \mathbf{1} \varepsilon} (Z_\D(K)).  $$
\end{proposition}
\begin{proof}
As usual, let $L$ be a $2g$-component vertical bottom tangle such that $K=  \widecheck{Th} (L)$ and write
$$\Big[ Cr^{\otimeshat g} \circ Z_\D (L))_{| t=0}\Big]^{1, \ldots , 2g} = e^{y Vx\cdot h}(1+\overline{P}_1 \varepsilon +   \overline{P}_2  \varepsilon^2) \pmod {\varepsilon^{3}}$$ for the generating series of the universal invariant of $L$ composed with the $g$-fold cross map and
 $$\Big[  \mu^{[g]} \circ B^{\otimeshat g} \Big]_{1, \ldots , 2g}^0 =e^{G}(T_0^g + P_1^{\alpha \neq 0}\varepsilon + P_2^{\alpha \neq 0} \varepsilon^2)  \pmod {\varepsilon^3}$$
 for the generating series of (part of) the thickening map as in \cref{prop gen ser Th with alpha neq 0}, so that $G_{| \alpha =0} =\frac{T_0-1}{h}\xi Q \eta$, and for the whole expression
  $$\Big[ ( \mu^{[g]} \circ B^{\otimeshat g})_{|\alpha =0} \Big]_{1, \ldots , 2g}^0 =e^{\frac{T_0-1}{h}\xi Q \eta}(T_0^g + P_1 \varepsilon + P_2 \varepsilon^2)  \pmod {\varepsilon^3}$$
  where $P_i = (P_i^{\alpha \neq 0})_{| \alpha=0}$. We then have
\begin{align*}
Z_\D (K)^0 &=  \begin{aligned}[t]
    &\left\langle  T_0^g  e^{y Vx\cdot h + G}   \right\rangle + \left\langle  (T_0^g \overline{P}_1 + P_1^{\alpha \neq 0} )  e^{y Vx\cdot h + G}   \right\rangle \varepsilon\\
    &+ \left\langle  (T_0^g \overline{P}_2 + P_1^{\alpha \neq 0} \overline{P}_1 + P_2^{\alpha \neq 0} )  e^{y Vx\cdot h + G}   \right\rangle \varepsilon^2
\end{aligned}\\
&= \begin{aligned}[t]
    &\left\langle  T_0^g  e^{y Vx\cdot h +\frac{T_0-1}{h}\xi Q \eta}   \right\rangle \\
    &+  \left(  \left\langle  T_0^g \overline{P}_1   e^{y Vx\cdot h + G}   \right\rangle + \left\langle    P_1   e^{y Vx\cdot h +\frac{T_0-1}{h}\xi Q \eta}  \right\rangle \right) \varepsilon\\
    &+ \left( \left\langle  (T_0^g \overline{P}_2 + P_1^{\alpha \neq 0} \overline{P}_1  )  e^{y Vx\cdot h + G}   \right\rangle  + \left\langle   P_2  e^{y Vx\cdot h + \frac{T_0-1}{h}\xi Q \eta}   \right\rangle
  \right) \varepsilon^2 \pmod {\varepsilon^{3}}
\end{aligned}
\end{align*}
where we put $\left\langle  -   \right\rangle= \left\langle  -   \right\rangle_{1, \ldots , 2g} $ and replaced $G$ and $P_i^{\alpha \neq 0}$ by their evaluations at $\alpha=0$ in those contractions where no $a_i$, $1\leq i \leq 2g$ appear. From this it follows that
\begin{equation}\label{eq coeff 1 eps}
     \mathrm{Coeff}_{ \mathbf{1} \varepsilon} (Z_\D(K)) = \left\langle  T_0^g \overline{P}_1   e^{y Vx\cdot h + G}   \right\rangle + \left\langle  \mathrm{Coeff}_{\mathbf{1}}(P_1)   e^{y Vx\cdot h +\frac{T_0-1}{h}\xi Q \eta}  \right\rangle  
\end{equation}
and
\begin{equation}\label{eq coeff a eps2}
    \mathrm{Coeff}_{ a_0 \varepsilon^2} (Z_\D(K)) = \left\langle   \mathrm{Coeff}_{a_0}(P_1^{\alpha \neq 0}) \overline{P}_1    e^{y Vx\cdot h + G}   \right\rangle  + \left\langle   \mathrm{Coeff}_{a_0}(P_2)  e^{y Vx\cdot h + \frac{T_0-1}{h}\xi Q \eta}   \right\rangle.
\end{equation}

On the other hand, \cref{lemma 77} implies that
 $$ -2 h T_0 \cdot \partial_{T_0} \Big(  \mathrm{Coeff}_{\mathbf{1}}(P_1)  e^{ \frac{T_0-1}{h}\xi Q \eta}  \Big) =  \mathrm{Coeff}_{a_0}(P_2) e^{ \frac{T_0-1}{h}\xi Q \eta} , $$ which in turn means that  
 $$ -2 h T_0 \cdot \partial_{T_0} \left\langle    \mathrm{Coeff}_{\mathbf{1}}(P_1)  e^{y Vx\cdot h+ \frac{T_0-1}{h}\xi Q \eta}  \right\rangle   =  \left\langle  \mathrm{Coeff}_{a_0}(P_2) e^{ y Vx\cdot h+\frac{T_0-1}{h}\xi Q \eta} \right\rangle $$ by the linearity of the pairing and the derivative.  Now multiplying by $\overline{P}_1 e^{ \frac{T_0-1}{h}\xi Q \eta}$ on both sides of  the first part of \cref{prop gen ser Th with alpha neq 0}  and contracting we get

 $$ -2 h T_0 \cdot \partial_{T_0} \left\langle T_0^g \overline{P}_1 e^{y Vx\cdot h+G} \right\rangle  =  \left\langle \mathrm{Coeff}_{a_0}(P_1^{\alpha \neq 0})\overline{P}_1 e^{y Vx\cdot h+G} \right\rangle  $$
 and adding the two last equations and comparing with \eqref{eq coeff 1 eps} and \eqref{eq coeff a eps2} we get the desired result.
\end{proof}

\begin{theorem}\label{thm rho 2 sth}
For any 0-framed knot $K$ we have
%
\vspace*{0.3cm}
\begin{enumerate}
\setlength\itemsep{0.7em}
\item $\displaystyle \rho_K^{1,1}(t)= \frac{2t}{1-t} \Delta_K ' (t),  $
\item $\displaystyle \rho_K^{2,1}(t)= \frac{2t \Delta_K^4(t)}{t-1} \left( \frac{\rho_K^{1,0}}{ \Delta_K^3}  \right) ' (t),  $
\item $\displaystyle \rho_K^{2,2}(t)= t \left( \frac{\Delta_K(t)}{1-t} \right)^3 \left( (3-t) (\Delta_K^{-1})'(t) + 2t(1-t) (\Delta_K^{-1})''(t)  \right).$
\end{enumerate}
\vspace*{0.1cm}
\end{theorem}
\begin{proof}
(1) Let us denote by $\mathrm{Coeff}_{a_0 \varepsilon} (Z_\D(K))$ the coefficient in $a_0 \varepsilon$ (independent of $a_0$ and $\varepsilon$) of  $Z_\D(K)^0$, where $0$ denotes the only label, that is
$$ Z_\D(K)^0 = \frac{1}{\Delta_K (T_0)} + \Big(\mathrm{Coeff}_{a_0 \varepsilon}  (Z_\D(K)) a_0 +  (\text{terms indep. of $a_0$}) \Big)  \varepsilon \pmod {\varepsilon^2}.$$
Comparing this with \cref{thm rhos existence} and \cref{lemma w mod eps2} yields
\begin{equation}\label{eq rho11}
\mathrm{Coeff}_{a_0 \varepsilon} (Z_\D(K)) = h \frac{(1-T_0)\rho^{1,1}_K(T_0) }{\Delta_K^2 (T_0)} .
\end{equation}
Let $L$ be a vertical bottom tangle such that $K= \widecheck{Th}(L)$, so that 
$$Cr^{\otimeshat g} (Z_\D (L))_{|t=0} =  e^{ yVx \cdot h } \pmod {\varepsilon}$$
as in the proof of \cref{prop Z_D for eps ^ 0 and 1}, where $V$ is a Seifert matrix for $K$. Another application of \cref{cor Th L}, using the expression of $\mathrm{Coeff}_{a_0} (P_1)$ for the first perturbation term $P_1$ of $( \mu^{[g]} \circ B^{\otimeshat g})_{|\alpha =0}$ derived in   \cref{prop a a2 for mu circ B}, yields 
\begin{align*}
    \mathrm{Coeff}_{a_0 \varepsilon} (Z_\D(K)) &= \left\langle  \mathrm{Coeff}_{a_0} (P_1) e^{\frac{T_0-1}{h}\xi Q \eta}   e^{   yVx \cdot h  }   \right\rangle_{1, \ldots , 2g}\\
    &= \left\langle  -2 h T_0 \cdot \partial_{T_0} \left( T_0^g e^{\frac{T_0-1}{h}\xi Q \eta} \right)    e^{   yVx \cdot h  }   \right\rangle_{1, \ldots , 2g}\\
    &= -2 h T_0 \cdot \partial_{T_0}\left\langle   T_0^g e^{\frac{T_0-1}{h}\xi Q \eta + yVx \cdot h }     \right\rangle_{1, \ldots , 2g}\\
    &= -2 h T_0 \cdot \partial_{T_0} \left( \frac{1}{\Delta_K (T_0)}\right),
\end{align*}
where we have used the linearity of the pairing and the derivative and the fact that \eqref{eq 1/Alex contraction} equals the inverse of the Alexander polynomial as showed in  \cref{prop Z_D for eps ^ 0 and 1}. This together with \eqref{eq rho11} imply that 
\begin{align*}
    \rho^{1,1}_K(T_0)&= \frac{\mathrm{Coeff}_{a_0 \varepsilon} (Z_\D(K)) \cdot \Delta_K (T_0)}{h (1-T_0)}\\
    &=\frac{-2 h T_0 \cdot \partial_{T_0} \left( \frac{1}{\Delta_K (T_0)}\right) \Delta_K^2(T_0)}{h(1-T_0)}\\
    &= \frac{2T_0}{1-T_0} \partial_{T_0}\Delta_K (T_0).
\end{align*}

(3) Let us define now $\mathrm{Coeff}_{a_0^2\varepsilon^2} (Z_\D(K))$ in a similar way. Again combining \cref{thm rhos existence} and \cref{lemma w mod eps2} we obtain
\begin{equation}\label{eq rho22}
    \mathrm{Coeff}_{a_0^2\varepsilon^2} (Z_\D(K)) =  \frac{h^2 (1+T_0) \rho^{1,1}_K(T_0) }{2\Delta_K^2 (T_0)} + \frac{h^2 (1-T_0)^2 \rho^{2,2}_K(T_0)}{\Delta_K^3}.
\end{equation}
On the other hand, the term $\mathrm{Coeff}_{a_0^2\varepsilon^2} (Z_\D(K))$ can be obtained from \cref{cor Th L} using the second part of  \cref{prop a a2 for mu circ B} as we did before:
\begin{align*}
    \mathrm{Coeff}_{a_0^2 \varepsilon^2} (Z_\D(K)) &= \left\langle  \mathrm{Coeff}_{a_0^2} (P_2) e^{\frac{T_0-1}{h}\xi Q \eta}   e^{   yVx \cdot h  }   \right\rangle_{1, \ldots , 2g}\\
    &= \left\langle  2 h^2 T_0 \left[ (\partial_{T_0} + T_0 \cdot \partial_{T_0, T_0} ) \left( T_0^g e^{\frac{T_0-1}{h}\xi Q \eta} \right) \right]   e^{   yVx \cdot h  }   \right\rangle_{1, \ldots , 2g}\\
    &= 2 h^2 T_0 (\partial_{T_0} + T_0 \cdot \partial_{T_0, T_0} )\left\langle    T_0^g e^{\frac{T_0-1}{h}\xi Q \eta + yVx \cdot h }   \right\rangle_{1, \ldots , 2g}\\
    &= 2h^2 T_0 \cdot \partial_{T_0} \left( \frac{1}{\Delta_K (T_0)} \right) +  2h^2 T_0^2 \cdot \partial_{T_0, T_0} \left( \frac{1}{\Delta_K (T_0)} \right).
\end{align*}
From \eqref{eq rho22} and (1) we obtain
\begin{align*}
    \rho^{2,2}_K(T_0)&= \frac{\Delta_K^3 (T_0)}{(1-T_0)^2} \left( \frac{ \mathrm{Coeff}_{a_0^2 \varepsilon^2} (Z_\D(K))}{h^2}   -  \frac{T_0(1+T_0)  }{(1-T_0) } \partial_{T_0}\Delta_K^{-1}(T_0) \right)\\
    &= \begin{aligned}[t]
     \frac{\Delta_K^3 (T_0)}{(1-T_0)^2} [  2(1-T_0)T_0 \partial_{T_0}\Delta_K^{-1}(T_0)  +  2(1-T_0) &T_0^2 \partial_{T_0, T_0}\Delta_K^{-1}(T_0) \\ &+T_0(T_0 +1) \partial_{T_0}\Delta_K^{-1}(T_0)   ]
\end{aligned}\\
&=T_0 \frac{\Delta_K^3 (T_0)}{(1-T_0)^2} [  (3-T_0) \partial_{T_0}\Delta_K^{-1}(T_0) + 2T_0 (1-T_0) \partial_{T_0, T_0}\Delta_K^{-1}(T_0) ].
\end{align*}

(2) Once again \cref{thm rhos existence} and \cref{lemma w mod eps2} imply that
\begin{equation}\label{eq another one}
\mathrm{Coeff}_{a_0 \varepsilon^2} (Z_\D(K)) = 
 \frac{h^2 \rho_K^{1,1} (T_0) (1+T_0)^2 }{2\Delta_K^2(T_0) } + \frac{h^2 \rho_K^{2,1} (T_0) (1-T_0) }{\Delta_K^4(T_0)} + \frac{h^2 \rho_K^{2,2} (T_0) (1-T_0)^2}{\Delta_K^3(T_0)} 
\end{equation}
and 
$$  \mathrm{Coeff}_{\mathbf{1}} (Z_\D(K)) =  \frac{h \rho_K^{1,0} (T_0)  }{\Delta_K^3(T_0) } + \frac{h \rho_K^{1,1} (T_0) (1-T_0) }{2\Delta_K^2(T_0) } ,$$
and by (1) 
\begin{align*}
    \partial_{T_0} \mathrm{Coeff}_{\mathbf{1}} (Z_\D(K)) &= h \cdot \partial_{T_0} \left(   \frac{\rho_K^{1,0}   }{\Delta_K^3 } \right)(T_0) -h \cdot  \partial_{T_0} \left(   
 T_0 \cdot \partial_{T_0} \Delta_K^{-1}(T_0) \right) \\
 &= h \cdot \partial_{T_0} \left(   \frac{\rho_K^{1,0}   }{\Delta_K^3 } \right)(T_0) -h \left( \partial_{T_0} \Delta_K^{-1}(T_0) + T_0 \cdot \partial_{T_0, T_0} \Delta_K^{-1}(T_0)    \right).
\end{align*}
This expression together with \eqref{eq another one} yield
\begin{align*}
    \rho^{2,1}_K(T_0) &= \frac{\Delta_K^4}{h^2(1-T_0)}  \left(  \mathrm{Coeff}_{a_0 \varepsilon^2} (Z_\D(K)) -   \frac{h^2 \rho_K^{1,1} (T_0) (1+T_0)^2 }{2\Delta_K^2(T_0) } - \frac{h^2 \rho_K^{2,2} (T_0) (1-T_0)^2}{\Delta_K^3(T_0)} \right)\\
    &= \frac{\Delta_K^4}{(1-T_0)} \left[ 2T_0  \cdot \partial_{T_0} \Delta_K^{-1}(T_0)  + 2T_0^2\cdot \partial_{T_0, T_0} \Delta_K^{-1}(T_0) - 2T_0  \cdot \partial_{T_0} \left(   \frac{\rho_K^{1,0}   }{\Delta_K^3 } \right)(T_0)  - \frac{ \rho_K^{2,2} (T_0) (1-T_0)^2}{\Delta_K^3(T_0)} \right]\\
    &= \begin{aligned}[t]
 \frac{\Delta_K^4}{(1-T_0)} &\Bigg[ \left(  2T_0 - \frac{T_0 }{1-T_0}(3-T_0) + \frac{T_0 (1+T_0)}{1-T_0} \right)\cdot \partial_{T_0} \Delta_K^{-1}(T_0) \\ &+   \left( 2T_0^2 - \frac{T_0}{1-T_0} 2T_0 (1-T_0)  \right)\cdot \partial_{T_0, T_0} \Delta_K^{-1}(T_0)  - 2T_0  \cdot \partial_{T_0} \left(   \frac{\rho_K^{1,0}   }{\Delta_K^3 } \right)(T_0) \Bigg]   
\end{aligned}\\
    &=\frac{2T_0 \Delta_K^4 (T_0)}{T_0-1} \cdot \partial_{T_0} \left(   \frac{\rho_K^{1,0}   }{\Delta_K^3 } \right)(T_0)
\end{align*}
where we have used  (1) and \cref{prop coef a esp2 = -2hT coeff 1eps} in the second equality and (2) in the third equality.
\end{proof}

\subsection{Behaviour under connected sums}

We would also like study the behaviour of $Z_\D (K)$ with respect to connected sums. Recall that for the Alexander polynomial, for knots $K, K'$, we have that
\begin{equation}\label{eq alexander K + K'}
\Delta_{K \# K'}= \Delta_{K }\Delta_{ K'}.
\end{equation}

\begin{theorem}\label{thm connected sum}
Let $K, K'$ be 0-framed knots, and let $n \geq 0$ and $0 \leq r \leq n$ be integers. Then
$$
 \rho_{K \# K'}^{n,r} = \sum_{i=0}^n \sum_{j=\max (0,r+i-n)}^{\min (i,r)}  \rho_{K }^{i,j} \rho_{K' }^{n-i,r-j}  \Delta^{2(n-i)-(r-j)}_{K} \Delta^{2i-j}_{K'}.
$$
\end{theorem}
\begin{proof}
First observe that, by definition, we have that $$Z_\D (K \# K') = Z_\D (K ) Z_\D ( K').$$ Now, by \cref{thm rhos existence},  the coefficient multiplying $\bm{w}^r \varepsilon^n$ in the left-hand side of the equation must be $$\frac{ \rho_{K \# K'}^{n,r} }{ \Delta_{K \# K'}^{2n+1-r} } = \frac{ \rho_{K \# K'}^{n,r} }{ \Delta_{K}^{2n+1-r} \Delta_{K'}^{2n+1-r}}, $$ whereas such a coefficient for the right-hand side must be $$ \sum_{i+u=n} \sum_{\substack{j+v=r\\ j \leq i \ , \ v \leq u}}  \frac{ \rho_{K }^{i,j} }{\Delta^{2i+1-j}_{K}} \frac{ \rho_{K' }^{u,v} }{\Delta^{2u+1-v}_{K'}} .$$
Here the $\rho^{i,j}$ and the Alexander polynomials are evaluated at $\bm{T}$ and we have suppressed $h$ by convenience. Setting $u:= n-i$ and $v:= r-j$, we must have $j \geq r+i-n$ in the second summation.  Rearranging terms we conclude the proof.
\end{proof}

\begin{remark}
Note that the argument of the previous theorem reproves \eqref{eq alexander K + K'}.
\end{remark}

\begin{corollary}
Let $K, K'$ be 0-framed knots and  $n \geq 0$. Then
$$
 \rho_{K \# K'}^{n,0} = \sum_{i=0}^n  \rho_{K }^{i,0} \rho_{K' }^{n-i,0} \Delta^{2(n-i)}_{K} \Delta^{2i}_{K'}.
$$
\end{corollary}
\begin{remark}
  The case $n=1$ was already showed in \cite{quarles}, but the argument used there involves a cumbersome, lengthy computation.  
\end{remark}

For convenience, let us write down explicitly the first values of  $\rho_{K \# K'}^{n,r} $:

\vspace*{0.5em}
\begin{enumerate}
\itemsep0.5em 
\item $\displaystyle   \rho_{K \# K'}^{1,0} = \rho_{K }^{1,0} \Delta^2_{K'}+ \rho_{ K'}^{1,0}\Delta^2_{K},$
\item $\displaystyle   \rho_{K \# K'}^{1,1} = \rho_{K }^{1,1} \Delta_{K'}+ \rho_{ K'}^{1,1}\Delta_{K},$
\item $\displaystyle  \rho_{K \# K'}^{2,0} = \rho_{K }^{2,0} \Delta^4_{K'}+ \rho_{ K'}^{2,0}\Delta^4_{K} + \rho_{K }^{1,0} \rho_{K' }^{1,0}\Delta^2_{K}\Delta^2_{K'}, $
\item $\displaystyle  \rho_{K \# K'}^{2,1} = \rho_{K }^{2,1} \Delta^3_{K'}+ \rho_{ K'}^{2,1}\Delta^3_{K} + \rho_{K }^{1,1} \rho_{K' }^{1,0}\Delta^2_{K}\Delta_{K'}+\rho_{K' }^{1,1} \rho_{K }^{1,0}\Delta^2_{K'}\Delta_{K},$
\item $\displaystyle  \rho_{K \# K'}^{2,2} = \rho_{K }^{2,2} \Delta^2_{K'}+ \rho_{ K'}^{2,2}\Delta^2_{K'} + \rho_{K }^{1,1} \rho_{K' }^{1,1}\Delta_{K}\Delta_{K'}.$
\end{enumerate}

\section{Properties of the band map} \label{sect band map}

In this section we derive properties of the generating series of the  band map that we use several times in  \cref{sect ZD}. For some of these results, which involve a large number of monomials, we have performed the computation with the help of a Mathematica code that we include in \cref{sect appendix}.

\begin{proposition}\label{prop exp band map} We have
$$(B_{|\alpha =0})_{ij}^k = T_k \exp \left[   \frac{T_k-1}{h}  \cdot \begin{pmatrix} \xi_i & \xi_j  \end{pmatrix}  \begin{pmatrix} 0 & 1 \\-1 & 0 \end{pmatrix}    \begin{pmatrix} \eta_i  \\ \eta_j  \end{pmatrix} \right]  \pmod \varepsilon $$
\end{proposition}
\begin{proof}
One way to show this is by  letting the computer perform the calculation by applying the contraction formula seven times. However, for the sake of exposition, we would like to include an explicit proof here.

Since $B$ is $\tau$-free by \cref{prop B tau free} and we are setting $\alpha =0$ for $B$, all monomials with $\bm{t}$ or $\bm{a}$ are mapped to zero and hence by \cref{lemma extension hat} $(B_{|\alpha =0})_{ij}^k$ equals
\begin{align}\label{eq yr xs yp xq 1}
\begin{split}
(\widehat{(\mathcal{O}^{\otimeshat 2})^{-1}} \circ \widehat{B} &\circ \widehat{\mathcal{O}} ) (e^{y_i \eta_i +  x_i \xi_i +y_j \eta_j +  x_j \xi_j  })\\ &= (\widehat{(\mathcal{O}^{\otimeshat 2})^{-1}} \circ \widehat{B}) \left( \sum_{r,s,p,q \geq 0} \frac{1}{r! s! p! q!} \bm{y}^r \bm{x}^s \otimeshat \bm{y}^p \bm{x}^q \eta_i^r \xi_i^s \eta_j^p x_j^q  \right)
\end{split}
\end{align}
Now the key observation is that it suffices to show that, modulo $\varepsilon$,
\begin{equation}\label{eq B ( yr xs yp xq ) }
B( \bm{y}^r \bm{x}^s \otimeshat \bm{y}^p \bm{x}^q ) = \begin{cases} \bm{T} \left( \frac{1-\bm{T}}{h} \right)^{r+s} r!s! (-1)^s, & \text{if } r=q \text{ and } s=p,\\ 0 & \text{else.} \end{cases}
\end{equation}
For \eqref{eq yr xs yp xq 1} becomes
\begin{align*}
\widehat{(\mathcal{O}^{\otimeshat 2})^{-1}}  &\left( \sum_{r,s \geq 0} \frac{1}{r! s! s! r!} \bm{T} \left( \frac{1-\bm{T}}{h} \right)^{r+s} r!s! (-1)^s \eta_i^r \xi_i^s \eta_j^p x_j^q    \right)\\ 
&= \widehat{(\mathcal{O}^{\otimeshat 2})^{-1}}   \left( \bm{T}  \sum_{r,s \geq 0} \frac{1}{r! s!}  \left( \frac{1-\bm{T}}{h} \right)^{r+s} (\eta_i \xi_j)^r (-\eta_j \xi_i)^s \right)\\ 
&= \widehat{(\mathcal{O}^{\otimeshat 2})^{-1}}   \left( \bm{T}  \sum_{r,s \geq 0} \frac{1}{r! s!}  \left( \frac{\bm{T}-1}{h} \right)^{r+s} (-\eta_i \xi_j)^r (\eta_j \xi_i)^s \right)\\
&= T_k \sum_{r,s \geq 0} \left( \frac{1}{r! s!}   \left( \frac{T_k -1}{h} \right)^{r+s} (-\eta_i \xi_j)^r (\eta_j \xi_i)^s \right) \pmod \varepsilon
\end{align*}
which equals the expression in the statement.

In order to show \eqref{eq B ( yr xs yp xq ) }, we start by noting the following equalities for the elementary monomials $\bm{y}^n \bm{x}^m$, which are routine calculations to show:
\begin{align*}
\Delta (\bm{y}^n \bm{x}^m) &= \sum_{i=0}^n \sum_{j=0}^m {n \choose i} {m \choose j} \bm{y}^i \bm{x}^j \otimeshat \bm{T}^i \bm{y}^{n-i} \bm{x}^{m-j} \pmod \varepsilon,\\
S(\bm{y}^n \bm{x}^m)  &= (-1)^{n+m} \bm{T}^{-n} \bm{x}^m \bm{y}^n \pmod \varepsilon,\\
\bm{x}^n \bm{y}^m &= \sum_{i=0}^m {n \choose i} {m \choose i} i! \bm{\lambda}^i \bm{y}^{m-i}  \bm{x}^{n-i} \pmod \varepsilon,
\end{align*}
where we put $\bm{\lambda} :=  \frac{1-\bm{T}}{h} $. These expressions imply that
\begin{align*}
(S \otimeshat &\id \otimeshat S \otimeshat\id ) (\Delta \otimeshat \Delta)(\bm{y}^r \bm{x}^s \otimeshat \bm{y}^p \bm{x}^q)\\
&= \sum_{i=0}^r \sum_{j=0}^s \sum_{u=0}^p \sum_{v=0}^q {r \choose i} {s \choose j} {p \choose u} {q \choose v} (-1)^{i+j+u+v} \bm{T}^{-i} \bm{x}^{j} \bm{y}^{i} \otimeshat \bm{T}^{i} \bm{y}^{r-i} \bm{x}^{s-j} \\ &\phantom{jajajajajajjajajajajajjajajaj} \otimeshat \bm{T}^{-u} \bm{x}^{v} \bm{y}^{u} \otimeshat \bm{T}^{u} \bm{y}^{p-u} \bm{x}^{q-v} \pmod \varepsilon
\end{align*}
and therefore $B(\bm{y}^r \bm{x}^s \otimeshat \bm{y}^p \bm{x}^q)$ equals
\begin{align*}
\bm{T} \sum_{i=0}^r \sum_{j=0}^s \sum_{u=0}^p \sum_{v=0}^q {r \choose i} {s \choose j} {p \choose u} {q \choose v} (-1)^{i+j+u+v} \bm{y}^{r-i} \bm{x}^{s-j+v}\bm{y}^{u} \bm{x}^{j} \bm{y}^{p-u+i} \bm{x}^{q-v}
\end{align*}
modulo $\varepsilon$. Moreover we have
\begin{align*}
\bm{y}^{r-i} \bm{x}^{s-j+v}\bm{y}^{u} &\bm{x}^{j} \bm{y}^{p-u+i} \bm{x}^{q-v} \\ &= \sum_{k=0}^u {s-j+v \choose k} {u \choose k} k! \bm{\lambda}^{k} \bm{y}^{r+u-i-k} \bm{x}^{s+v-k} \bm{y}^{p+i-u} \bm{x}^{q-v} \\
&= \sum_{k=0}^u \sum_{\ell =0}^{p+i-u} {s-j+v \choose k} {u \choose k} {s+v-k \choose\ell} {p-u + i \choose \ell} k! \ell ! \bm{\lambda}^{k+ \ell } \cdot \\ &\hspace*{6cm} \bm{y}^{(r+p)-(k+ \ell)} \bm{x}^{(s+q)-(k+ \ell)}.
\end{align*}
Putting all together and permuting summations, we have
\begin{align}\label{eq B ( yr xs yp xq ) 2}
\begin{split}
B(\bm{y}^r \bm{x}^s \otimeshat \bm{y}^p \bm{x}^q) = &\bm{T}  \sum_{k=0}^p \sum_{\ell = 0}^{p+r-k} k! \ell ! \bm{\lambda}^{k+ \ell}  \bm{y}^{(r+p)-(k+ \ell)} \bm{x}^{(s+q)-(k+ \ell)} \cdot \\ &\sum_{v=0}^q (-1)^v {q \choose v}{s+v-k \choose\ell} \sum_{j=0}^s (-1)^j {s \choose j} {s  +v-j \choose k} \cdot \\  &\hspace*{-.6cm} \sum_{u=k}^{\min (p,p+r-\ell)} (-1)^u  {p \choose u} {u \choose k} \sum_{i= \max (\ell + u-p,0)}^r  (-1)^i{r \choose i} {p-u+i \choose \ell}.
\end{split}
\end{align}
Now we claim that for every $0 \leq k \leq p$, the rest of the sum is zero for $\ell < p+r-k$. To this end, it suffices to see that  $$\sum_{u=k}^{\min (p,p+r-\ell)} (-1)^u  {p \choose u} {u \choose k} \sum_{i= \max (\ell + u-p,0)}^r  {r \choose i} {p-u+i \choose \ell}=0$$ for $\ell < p+r-k$. To begin with, note that  $ \max (\ell + u-p,0)$ can be replace by $0$ as if $0 \leq i < \ell + u - p$ then ${p-u+i \choose \ell}=0$. Now if $N,M,R \geq 0$ are non-negative integers then the following  equality, which can be showed by elementary methods, holds:
\begin{equation}\label{eq binomial 1}
\sum_{i=0}^N (-1)^i {N \choose i} {R+i \choose M} = (-1)^N {R \choose M-N},
\end{equation}
where ${R \choose M-N}:=0$ if $M<N$. In particular for $R=0$ it yields
\begin{equation}\label{eq binomial 2}
\sum_{i=0}^N (-1)^i {N \choose i} {i \choose M} = (-1)^N \delta_{N,M}
\end{equation}
with $\delta_{N,M}$ the Kroneker delta. Using these formulas we finally compute
\begin{align*}
\sum_{u=k}^{\min (p,p+r-\ell)} (-1)^u  {p \choose u} {u \choose k} &\sum_{i= 0}^r  {r \choose i} {p-u+i \choose \ell} \\ &= (-1)^r \sum_{u=0}^{p-(\ell - r)} (-1)^u  {p \choose u} {u \choose k} {p-u \choose \ell - r} \\ &=  (-1)^r {p \choose \ell - r} \sum_{u=0}^{p-(\ell - r)} (-1)^u  {u \choose k} {p - (\ell -r) \choose u} \\  &= (-1)^r {p \choose \ell - r} \delta_{p-(\ell - r), k} \\ &= 0
\end{align*}
where we have assumed that $\ell >r$ (otherwise the sum is automatically zero by \eqref{eq binomial 1}). This shows the claim.

Therefore, we can put $\ell = p+r-k$ in \eqref{eq B ( yr xs yp xq ) 2}, so that this is rewritten as 
\begin{align}\label{eq B ( yr xs yp xq ) 3}
\begin{split}
B(\bm{y}^r \bm{x}^s \otimeshat \bm{y}^p \bm{x}^q) =  &\bm{T} \bm{\lambda}^{p+r}   \bm{x}^{(s+q)-(p+r)} \sum_{k=0}^p (-1)^{k+r}  k! (p+r-k) ! {p \choose k}  \cdot \\ &\sum_{v=0}^q (-1)^v {q \choose v}{s+v-k \choose p+r-k} \sum_{j=0}^s (-1)^j {s \choose j} {s  +v-j \choose k}.
\end{split}
\end{align}

Next we claim that $$\sum_{v=0}^q (-1)^v {q \choose v}{s+v-k \choose p+r-k} \sum_{j=0}^s (-1)^j {s \choose j} {s  +v-j \choose k}$$ equals $(-1)^q {p+r-s \choose k-s} \delta_{p+r,q+s}.$ For this, we will use the following expression which easily follows from \eqref{eq binomial 1}: if $N,M,R \geq 0$ are non-negative integers and $R \geq N$, then we have
\begin{equation}\label{eq binomial 3}
\sum_{i=0}^N (-1)^i {N \choose i} {R-i \choose M} =  {R-N \choose M-N}.
\end{equation}
 Using this we compute
\begin{align*}
\sum_{v=0}^q (-1)^v {q \choose v}{s+v-k \choose p+r-k} &\sum_{j=0}^s (-1)^j {s \choose j} {s  +v-j \choose k}\\
&=  \sum_{v=0}^q (-1)^v {q \choose v}{s+v-k \choose p+r-k} {v \choose k-s} \\ &= {p + r-s \choose k-s}  \sum_{v=0}^q (-1)^v {q \choose v} {v \choose p+r-s} \\ &= (-1)^q {p + r-s \choose k-s}  \delta_{p+r-s,q}
\end{align*}
which demonstrates the claim.

This shows that $B(\bm{y}^r \bm{x}^s \otimeshat \bm{y}^p \bm{x}^q) = 0$ if $p+r \neq q+s$ and hence  \eqref{eq B ( yr xs yp xq ) 3} is rewritten as
$$
B(\bm{y}^r \bm{x}^s \otimeshat \bm{y}^p \bm{x}^q) = \bm{T}  \bm{\lambda}^{p+r} (-1)^{r+q} \delta_{p+r,q+s} \sum_{k=s}^p (-1)^{k}  k! (p+r-k) ! {p \choose k} {p + r-s \choose k-s} 
$$
and noting that
\begin{align*}
\sum_{k=s}^p (-1)^{k}  k! (p+ &r -k) ! {p \choose k} {p + r-s \choose k-s} \\
&= (-1)^s \sum_{\widetilde{k}=0}^{p-s} (-1)^{\widetilde{k}} {p \choose \widetilde{k} +s } {p+r-s \choose \widetilde{k}} (\widetilde{k} +s)! (p+r-s-\widetilde{k})! \\
&= (-1)^s \frac{p! (p+r-s)!}{(p-s)!} \sum_{\widetilde{k}=0}^{p-s}   (-1)^{\widetilde{k}} {p-s \choose \widetilde{k}}\\
&= (-1)^s \frac{p! (p+r-s)!}{(p-s)!} \delta_{p,s}\\
&= (-1)^s p! s! \delta_{p,s}
\end{align*}
we conclude that
\begin{align*}
B(\bm{y}^r \bm{x}^s \otimeshat \bm{y}^p \bm{x}^q) &= \bm{T} \bm{\lambda}^{p+r} (-1)^{r+q+s}   p! r! \delta_{p,s}\delta_{p+r,q+s} \\ &= \bm{T} \bm{\lambda}^{p+r}  (-1)^{s} r!s!  \delta_{p,s} \delta_{q,r}
\end{align*}
as claimed in \eqref{eq B ( yr xs yp xq ) }.
\end{proof}

\begin{corollary}\label{cor gen funct mu circ B}
If $g\geq 1$ then
$$\Big[ ( \mu^{[g]} \circ B^{\otimeshat g})_{|\alpha =0} \Big]_{1, \ldots , 2g}^0 = T_0^g \exp \left[ \frac{T_0-1}{h} \cdot \xi Q \eta \right]  \pmod \varepsilon $$
where $Q$ is the intersection form matrix from \eqref{eq V seifert surface}, $\xi= (\xi_1, \ldots , \xi_{2g})$ and $\eta= (\eta_1, \ldots , \eta_{2g})^T$.
\end{corollary}
\begin{proof}
We have $$( \mu^{[g]} \circ B^{\otimeshat g})_{|\alpha =0} =  \mu^{[g]} \circ B^{\otimeshat g}_{|\alpha =0} = \mu^{[g]}_{| \alpha = t =0} \circ (B_{|\alpha =0})^{\otimeshat g}  \pmod \varepsilon ,$$ where the first equality is consequence of \cref{lemma z_i=0} and the second of \cref{lemma y_i=0} given that that $B_{| \alpha =0} \pmod \varepsilon$ is both $\tau$ and $a$-free. Since $\mu^{[g]}_{| \alpha = t =0} \pmod \varepsilon$ is the $g$-fold commutative multiplication, the claim follows.
\end{proof}

We will also need to know some pieces of the perturbation part of the band map. The following result can be showed by a repeated application of the Contraction \cref{thm contraction}. We check this by computer in \cref{sect appendix}.

\begin{lemma}\label{lemma P1 P2 B}
Write
$$(B_{|\alpha =0})_{ij}^k =  \exp \left[   \frac{T_k-1}{h}  \cdot \xi Q \eta \right] (T_k + P_1 \varepsilon + P_2 \varepsilon^2)  \pmod {\varepsilon^3} $$  for the generating series of the band map, and put $\mathrm{Coeff}_{a_k}(P_1)$ for the term (independent of $a_k$)   that multiplies $a_k$ in the expression of $P_1$. Then $$\mathrm{Coeff}_{a_k}(P_1) = -2T_k (h+T_k (\xi_i \eta_j - \xi_j \eta_i)).$$ If $\mathrm{Coeff}_{a_k^2}(P_2)$ is defined analogously for $a_k^2$ and $P_2$, then $$\mathrm{Coeff}_{a_k^2}(P_2) = 2h^2 T_k + 6h T_k^2(\xi_i \eta_j - \xi_j \eta_i) + 2T_k^3(\xi_i \eta_j - \xi_j \eta_i)^2. $$
\end{lemma}

\begin{proposition}\label{prop a a2 for mu circ B} Let $g \geq 1$ and write
 $$\Big[ ( \mu^{[g]} \circ B^{\otimeshat g})_{|\alpha =0} \Big]_{1, \ldots , 2g}^0 =e^{\frac{T_0-1}{h}\xi Q \eta}(T_0^g + P_1 \varepsilon + P_2 \varepsilon^2)  \pmod {\varepsilon^3}.$$ If  $\mathrm{Coeff}_{a_0}(P_1)$ and
 $\mathrm{Coeff}_{a_0^2}(P_2)$ are defined as above, then we have
 $$  \mathrm{Coeff}_{a_0}(P_1) e^{\frac{T_0-1}{h}\xi Q \eta} = -2 h T_0 \cdot \partial_{T_0} \left( T_0^g e^{\frac{T_0-1}{h}\xi Q \eta} \right)  $$
 and
 $$  \mathrm{Coeff}_{a_0^2}(P_2) e^{\frac{T_0-1}{h}\xi Q \eta} = 2 h^2 T_0 \left[ (\partial_{T_0} + T_0 \cdot \partial_{T_0, T_0} ) \left( T_0^g e^{\frac{T_0-1}{h}\xi Q \eta} \right) \right]. $$ 
\end{proposition}
\begin{proof}
The main observation is that we still can argue as in \cref{cor gen funct mu circ B}, meaning that when performing the composite $$( \mu^{[g]} \circ B^{\otimeshat g})_{|\alpha =0} =  \mu^{[g]} \circ B^{\otimeshat g}_{|\alpha =0},$$ we can treat the $g$-fold multiplication as commutative multiplication. This is because the coefficients displayed \cref{lemma P1 P2 B}, which are the only ones that play a role in determining the coefficients required in the statement, only depend on $T$ as a Latin variable, which is central.

It readily follows, using \cref{lemma P1 P2 B}, that
\begin{align*}
\mathrm{Coeff}_{a_0}(P_1) &= T_0^{g-1} \sum_{i=1}^g (-2T_0 (h+T_0 (\xi_{2i-1} \eta_{2i} - \xi_{2i} \eta_{2i-1})))\\
&= -2g h T_0^g - 2T_0^{g+1} \xi Q \eta \\
&= -2  T_0 (g h T_0^{g-1} + T_0^g  \xi Q \eta)
\end{align*}
and hence 
\begin{align*}
\mathrm{Coeff}_{a_0}(P_1) e^{\frac{T_0-1}{h}\xi Q \eta} &= -2  T_0 (g h T_0^{g-1} + T_0^g  \xi Q \eta)e^{\frac{T_0-1}{h}\xi Q \eta}\\
&= -2 h T_0 \cdot \partial_{T_0} \left( T_0^g e^{\frac{T_0-1}{h}\xi Q \eta} \right).
\end{align*}

For the coefficient in $a_0^2$, the argument is similar but slightly more involved. Namely, we have
\begin{align*}
 (B_{|\alpha =0})_{ij}^k =  e^{(*)} [ T_k + ( \mathrm{Coeff}_{a_k}(P'_1)_{ij}^k a_k + (*) ) \varepsilon +  ( \mathrm{Coeff}_{a_k^2}(P'_2)_{ij}^k a^2_k + &(*) ) \varepsilon^2]  \\ &\pmod {\varepsilon^3}   
\end{align*}
where as usual $(*)$ indicates some irrelevant terms and $P_1', P_2'$ denote the perturbation terms of $B$. Using \cref{lemma P1 P2 B}, this means that
\begin{align*}
 \mathrm{Coeff}_{a_0^2}(P_2)   &= 
 \begin{aligned}[t] &T_0^{g-1} \sum_{i=1}^g  \mathrm{Coeff}_{a^2_0}(P'_2)_{2i-1,2i}^0 \\ &+ T_0^{g-2} \sum_{1 \leq i < j \leq g} \mathrm{Coeff}_{a_0}(P'_1)_{2i-1,2i}^0 \mathrm{Coeff}_{a_0}(P'_1)_{2j-1,2j}^0 \end{aligned}\\
 &=  \begin{aligned}[t]
     &2g h^2 T_0^g + 6 h T_0^{g+1} \xi Q \eta + 2 T_0^{g+2} \sum_{i=1}^g (\xi_{2i-1} \eta_{2i} - \xi_{2i} \eta_{2i-1})^2 \\
     &+ 2g(g-1) h^2 T_0^g + 4(g-1) h T_0^{g+1} \xi Q \eta\\ &+ 4 T_0^{g+2}  \sum_{1 \leq i < j \leq g} (\xi_{2i-1} \eta_{2i} - \xi_{2i} \eta_{2i-1})(\xi_{2j-1} \eta_{2j} - \xi_{2j} \eta_{2j-1})
 \end{aligned}\\
 &= 2g h^2 T_0^g  + 2g(g-1) h^2 T_0^g + (4g+2) h T^{g+1} \xi Q \eta + 2T_0^{g+2}(\xi Q \eta)^2\\
 &= \begin{aligned}[t]
&2T_0(g h^2 T_0^{g-1} + h T_0^g \xi Q \eta )+2T_0^2 (g(g-1)h^2 T_0^2 + 2g h T_0^{g-1}\xi Q \eta + T_0^g (\xi Q \eta)^2),
  \end{aligned}
\end{align*}
where we have used that
$$(\xi Q \eta)^2= \sum_{i=1}^g (\xi_{2i-1} \eta_{2i} - \xi_{2i} \eta_{2i-1})^2 + 2 \sum_{ i < j } (\xi_{2i-1} \eta_{2i} - \xi_{2i} \eta_{2i-1})(\xi_{2j-1} \eta_{2j} - \xi_{2j} \eta_{2j-1})$$ in the third equality. From this it follows that
\begin{align*}
\mathrm{Coeff}_{a_0^2}(P_2) e^{\frac{T_0-1}{h}\xi Q \eta} &= \begin{aligned}[t]
&2T_0(g h^2 T_0^{g-1} + h T_0^g \xi Q \eta )e^{\frac{T_0-1}{h}\xi Q \eta} \\ &+2T_0^2 (g(g-1)h^2 T_0^2 + 2g h T_0^{g-1}\xi Q \eta + T_0^g (\xi Q \eta)^2)e^{\frac{T_0-1}{h}\xi Q \eta}  \end{aligned}\\
&=  2 h^2 T_0 \cdot \partial_{T_0} \left( T_0^g e^{\frac{T_0-1}{h}\xi Q \eta} \right) + 2 h^2 T_0^2 \cdot  \partial_{T_0, T_0} \left( T_0^g e^{\frac{T_0-1}{h}\xi Q \eta} \right)
\end{align*}
which concludes.
\end{proof}

In the previous results we have set $\alpha =0$, as it is our main case of interest. However, it is further true that the two equalities of \cref{prop a a2 for mu circ B} hold in general:

\begin{proposition}\label{prop gen ser Th with alpha neq 0}
 Let $g \geq 1$ and write
 $$\Big[  \mu^{[g]} \circ B^{\otimeshat g} \Big]_{1, \ldots , 2g}^0 =e^{G}(T_0^g + P_1^{\alpha \neq 0}\varepsilon + P_2^{\alpha \neq 0} \varepsilon^2)  \pmod {\varepsilon^3},$$ so that $G_{| \alpha =0} =\frac{T_0-1}{h}\xi Q \eta$.  If  $\mathrm{Coeff}_{a_0}(P_1^{\alpha \neq 0})$ and
 $\mathrm{Coeff}_{a_0^2}(P_2^{\alpha \neq 0})$ are defined as above, then we have
 $$  \mathrm{Coeff}_{a_0}(P_1^{\alpha \neq 0}) e^{G} = -2 h T_0 \cdot \partial_{T_0} \left( T_0^g e^{G} \right)  $$
 and
 $$  \mathrm{Coeff}_{a_0^2}(P_2^{\alpha \neq 0}) e^{G} = 2 h^2 T_0 \left[ (\partial_{T_0} + T_0 \cdot \partial_{T_0, T_0} ) \left( T_0^g e^{G} \right) \right]. $$ 
\end{proposition}
\begin{proof}[Proof sketch]
Let us set $g=1$ for simplicity. For the first equality, it sufficies to show that
$$   \mathrm{Coeff}_{a_0}(P_1^{\alpha \neq 0})  = -2h T_0 (1+T_0 \cdot \partial_{T_0}G)   $$
and for the second one that
$$ \mathrm{Coeff}_{a_0^2}(P_2^{\alpha \neq 0}) = 2h^2 T_0 (1+3T_0 \partial_{T_0}G + T_0^2 [\partial_{T_0, T_0}G + (\partial_{T_0}G)^2 ]) .  $$ Then these equalities are checked best by computer (the term $\mathrm{Coeff}_{a_0^2}(P_2^{\alpha \neq 0})$ is a polynomial with 40 monimials), and we do this in \cref{sect appendix}. The general case for $g>1$ requires an extra step as the one from \cref{lemma P1 P2 B} to \cref{prop a a2 for mu circ B}.
\end{proof}

\begin{remark}
Even for $g=1$, the full expression of the generating series of the band map is tremendously large for $\alpha\neq 0$: only the term $P_2^{\alpha \neq 0}$ has over 8,000 monomials. This obliges us to perform these calculations by computer (which is equally rigorous).
\end{remark}

\begin{lemma}\label{lemma 77}
 Let $g \geq 1$ and write
 $$\Big[ ( \mu^{[g]} \circ B^{\otimeshat g})_{|\alpha =0} \Big]_{1, \ldots , 2g}^0 =e^{\frac{T_0-1}{h}\xi Q \eta}(T_0^g + P_1 \varepsilon + P_2 \varepsilon^2)  \pmod {\varepsilon^3}$$ as in the previous Proposition. Let $\mathrm{Coeff}_{a_0}(P_2)$ denote the term that multiplies $a_0$ in $P_2$ and let $\mathrm{Coeff}_{\mathbf{1}}(P_1)$ be the term independent of $y_0, a_0, x_0$ in $P_1$ (that is, the evaluation of $P_1$ at $y_0 = a_0 = x_0 =0$). Then the following equality holds:
 $$ -2 T_0 \Big( \xi Q \eta \cdot \mathrm{Coeff}_{\mathbf{1}}(P_1) + h \cdot \partial_{T_0}  \mathrm{Coeff}_{\mathbf{1}}(P_1)   \Big) =  \mathrm{Coeff}_{a_0}(P_2) . $$
\end{lemma}
\begin{proof}[Proof sketch]
For simplicity we put $g=1$, so we directly look at the generating series of the band map $(B_{|\alpha =0})_{1,2}^0$ . By direct computation using the Contraction \cref{thm contraction}, we find that
\begin{align*}
    \mathrm{Coeff}_{\mathbf{1}}(P_1) &= 2T_0 (1-T_0) (\xi_1 \eta_2 -\xi_1 \eta_1 -\xi_2 \eta_2) - \frac{T_0(1-T_0)}{h} (\xi_1^2 \eta_1 \eta_2 + \xi_1 \xi_2 \eta_2^2)\\
    &+ \frac{T_0 (1-T_0)(3-T_0)}{4h} (\xi_1^2 \eta_2^2 - \xi_2^2\eta_1^2 + \xi_1 \xi_2 \eta_1 \eta_2)
\end{align*}
and similarly
\begin{align*}
    \mathrm{Coeff}_{a_0}(P_2) &= 4hT_0(1-2T_0)(\xi_1 \eta_1 + \xi_2 \eta_2 - \xi_1 \eta_2) + 2T_0(1-2T_0^2)(\xi_1^2 \eta_1 \eta_2 + \xi_1 \xi_2 \eta_2^2)\\
    &+ \frac{T_0(-3+5T^2_0)}{2}\xi_1^2 \eta_2^2 + \frac{T_0(3-8T_0+3T^2_0)}{2}\xi_2^2 \eta_1^2\\
    &+  \frac{2T_0^2(1-T_0)}{h} (\xi_1^3\eta_1 \eta_2^2 -2h \xi_1 \xi_2 \eta_1^2 - \xi_1^2 \xi_2 \eta_1^2 \eta_2 + \xi_1^2\xi_2 \eta_2^3 - 2h \xi_2^2\eta_1 \eta_2 - \xi_1 \xi_2^2 \eta_1 \eta_2^2)\\
    &+ \frac{T_0^2 (1-T_0)(T_0-3)}{2h} (\xi_1^3 \eta_2^3 + \xi_2^3\eta_1^3 + 3\xi_1^2 \xi_2 \eta_1 \eta_2^2 - 5 \xi_1 \xi_2^2 \eta_1^2 \eta_2 )\\
    &+ 2T_0 (-3+10T_0 - 5T_0^2) \xi_1 \xi_2 \eta_1 \eta_2.
\end{align*}
Then the equality in the statement is a mere checking, that we do by computer in \cref{sect appendix}. The general case for $g>1$ requires an extra step similar to the passage from \cref{lemma P1 P2 B} to \cref{prop a a2 for mu circ B}, which is a tedious computation left to the reader.
\end{proof}

\section{Perspectives}  \label{sect perpectives}

We would like to close off stating some open problems and future directions related to the universal invariant $Z_\D$.

\subsection{The polynomials $\rho_K^{n,0}$}

Aside the concrete formulas, a consequence of \cref{thm rho 2 sth} (and \cref{thm rho ij =0 for j>i}) is that for a 0-framed knot $K$, the value of $Z_\D(K)$ is completely determined by the triple of polynomials ($\Delta_K, \rho_K^{1,0}, \rho_K^{2,0})$. More precisely, these result claim that the rest of values $\rho^{i,j}_K$ for $i=1,2$ and $j>0$ are either trivial or are expressed in terms of derivatives of $\Delta_K$ and $ \rho_K^{1,0}$.

We expect this to be the situation for higher order terms as well:

\begin{conjecture}[see also \cite{barnatanveengaussians}]\label{conj rho i0}
If $K$ is a 0-framed knot, the value of $Z_\D(K)$ is completely determined by the Alexander polynomial $\Delta_K$ and the family of polynomials $$\rho_K^{i,0} \in \Q [t,t^{-1}] \qquad , \qquad i>0.$$

More precisely, if  $j>0$, the value $\rho_K^{i,j}$ can be written in terms of the Alexander polynomial $\Delta_K$, the polynomials $\rho_K^{i,0}$, $i<j$, and derivatives of these.
\end{conjecture}

For the cases $\rho_K^{1,1}$ and $\rho_K^{2,2}$, we showed in \cref{thm rho 2 sth}  that they only depend on the Alexander polynomial. We expect this to be always the case for $i=j$:

\begin{conjecture}
If $K$ is a 0-framed knot, the  family of polynomials $\rho_K^{i,i}$, $i>0$, is completely determined by the Alexander polynomial of $K$.
\end{conjecture}

\subsection{Relation with knot topological properties}

When studying quantum invariants of knots, it is natural to ask between the interplay of the invariant $Z_\D$ with the topological nature of the knot. For a 0-framed knot, \cref{prop Z_D for eps ^ 0 and 1} ensures that $Z_\D (K)$ (mod $\varepsilon$) is equivalent to the Alexander polynomial of the knot, which has a well-understood topological interpretation in terms of the first homology group of the universal abelian cover of the knot complement \cite{lickorish}.  If we view the rest of knot polynomial invariants $\rho_K^{i,0}$ as ``perturbed'' versions of the Alexander polynomial, it is sensible to ask for topological interpretations for these as well.

\begin{problem}
Give a topological construction  for each of  the knot polynomial invariants $\rho_K^{i,0}$.
\end{problem}

By a topological construction we mean one in terms of the knot complement. For quantum invariants this is generally a very hard problem, remarkably no topological interpretation is known even for the Jones polynomial \cite{bigelow}. However, Bar-Natan and van der Veen have recently shown that the knot polynomial $\rho_K^{1,0}$ can be computed from a quadratic expression whose variables are evaluated at the entries of the inverse of the matrix $A$ which arises from applying the Fox calculus to the Wirtinger presentation of the fundamental group of the knot complement \cite{barnatanveenAPAI}. Since the determinant of the matrix $A$ is precisely the Alexander polynomial, we believe that this approach might lead to the desired topological interpretation of $\rho_K^{1,0}$. 

Another fact that leads us to think that $\rho_K^{1,0}$ should have a purely topological interpretation is that it provides a knot genus bound, $g(K) \geq \frac{1}{4} \mathrm{breadth} \ \rho_K^{1,0}$, sometimes sharper than the Alexander bound \cite{barnatanveengaussians}. Therefore, given that the Alexander polynomial give necessary conditions for a knot to be fibred or slice, we can ask whether similar properties exist for $\rho_K^{1,0}$.

\begin{problem}
Give necessary conditions for a knot $K$ to be slice or fibred in terms of $\rho_K^{1,0}$.
\end{problem}

\subsection{Relation with the coloured Jones polynomials} \label{subsec coloured Jones}

Let us explain how the collection of polynomials $\rho_K^{i,0}$ might relate with the so-called rational expansion of the coloured Jones polynomials. 

For $n \geq 2$ and $K$ a 0-framed knot, let $J_{K}^n (q) \in \Z [q,q^{-1}]$ denote the coloured Jones polynomial of $K$, normalised so that $J_O^n (q)=1$. That is,
\begin{equation}
J_K^n (q) = \frac{q^{- c_n  }}{ \{n \}_q }  RT(K; V_n)
\end{equation}
where $q=e^h$ and  $ \{n \}_q := q^{-n+1} [n]_{q^2} = RT(O; V_n)$. Here $V_n$ stands for the rank $n$ irreducible representation of $U_h (\mathfrak{sl}_2)$,  and $c_n := (n-1)(n+1)/2$ is the Casimir element of $V_n$.

Based on the work by Melvin-Morton  on the expansion of the coloured Jones polynomials \cite{MM}, Rozansky conjectured (and later showed) a rational expansion of the coloured Jones polynomials with denominators powers of the Alexander polynomial.

\begin{theorem}[\cite{rozanskyconjectures, rozanskyRmatrix}]\label{thm rozansky}
Let $K$ be a 0-framed knot. Then there exist symmetric Laurent polynomials $$P_K^{i} \in \Z [t,t^{-1}] \qquad , \qquad i \geq 0$$ such that $$J_K^n (q)= \sum_{i=0}^\infty \frac{P_K^i (q^n)}{\Delta_K^{2i+ 1} (q^n)} (q-1)^i\in \Q [[q-1]].$$
Furthermore, $P_K^0 (t) =1$ .
\end{theorem}

In the previous statement, both sides of the equality are to be understood as their expansions in powers of $q-1$ (note that $\Delta_K(1)=1$). Besides, the property ${P_K^0 (t) =1}$ is tantamount to the so-called \textit{Melvin-Morton-Rozansky conjecture}, proven in \cite{BNG}. 

\begin{remark}
The collection of polynomials $P_K^i$ in the Theorem should be understood as a repackaging of the coloured Jones polynomials (see e.g. \cite{analytic}). Indeed it can be shown that, given a 0-framed knot $K$,  there exist rational numbers $a_{ij}(K)$  such that
$$J_K^n (e^h)= \sum_{i,j \geq 0} a_{ij}(K) n^j h^i $$
for any $n \geq 2$, each $a_{ij}(K)$ being a degre $i$ finite type invariant of $K$. The key property is that $a_{ij}(K)=0$ when $j>i$, which allows us to sum over diagonals and rewrite the previous expression as
\begin{equation}\label{eq J_K MM}
J_K^n (e^h)=\sum_{\ell =0}^\infty R_K^\ell (nh)  h^\ell
\end{equation}
for $R_K^\ell(x):= \sum_{j=0}^\infty a_{j+\ell, \ell }x^j \in \Q[[x]]$. Rozansky then showed that each  $R_K^\ell$ can be written as a rational function, $$R_K^\ell (x)= \frac{\widetilde{P}_K^\ell (e^x)}{\Delta_K^{2\ell+1}(e^x)}  $$ for some symmetric Laurent polynomials $\widetilde{P}_K^\ell \in \Q[t,t^{-1}]$. After the change of variables $h= \log ( (q-1)+1)$ the formula of \cref{thm rozansky} follows.
\end{remark}

Now recall from \cref{subsec relation with classical quantum invariants} that the coloured Jones polynomials can be recovered from the universal invariant $Z_{U_h(\mathfrak{sl}_2)}$. Habiro \cite{habiro_WRT} proved that the converse is also true, and together with the previous Remark we have 

\begin{theorem}\label{thm equiv pieces of data}
Let $K$ be a 0-framed knot. The following pieces of data are all equivalent:
\begin{enumerate}
\item The universal invariant $Z_{U_h(\mathfrak{sl}_2)}(K)$,
\item The coloured Jones polynomials $ \{ J_K^n (K) \}_{n \geq 2}$,
\item The Rozansky polynomials $ \{ P_K^n (K) \}_{n \geq 1}$.
\end{enumerate}
\end{theorem}

By \cref{cor ZD recovers}, the universal invariant $Z_\D(K)$ of a 0-framed knot $K$ determines the first (and hence all) item of the previous list.   We conjecture that the universal invariant $Z_\D$ contains no more information than these three pieces of data:

\begin{conjecture}
For any 0-framed knot $K$, the datum of the universal invariant $Z_\D(K)$ is equivalent to any of the three pieces of data of the previous Theorem.
\end{conjecture}

By \cref{conj rho i0}, a rephrase of the previous statement is that the collection of Bar-Natan - van der Veen polynomials $\{ \rho_K^{n,0} \}_{n \geq 1}$ is equivalent to each of the pieces of data of \cref{thm equiv pieces of data}.

We would like to be more precise about the relation between $Z_\D$ and the Rozansky polynomials. Computational evidence strongly suggests

\begin{conjecture}\label{conj rho10 =P1}
For any 0-framed knot $K$, $$\rho_K^{1,0} = P_K^1.$$
\end{conjecture}

Actual tables supporting this Conjecture can be found in \cite{barnatanveengaussians} and \cite{rozansky12theta} (after the change of variables $z=t^{1/2}-t^{-1/2}$). For the rest of polynomials, we would like to stress the similarity between the rational expansions of \cref{thm rhos existence} and \cref{thm rozansky}, each of the $\rho_K^{i,0}$ having as a denominator the same power of Alexander as its ``counterpart'' $P_K^i$. Even more, we have $$\rho_K^{0,0} =1 = P_K^0$$ (cf. \cref{prop Z_D for eps ^ 0 and 1} and \cref{thm rozansky}). However for $i>0$ the values of $\rho_K^{i,0}$ and $P_K^i$ seem to be different. This is simply because of the choice of the central element $\bm{w} \in \D$ in \cref{prop w}, that was convenient but not canonical. Hence we have the following

\begin{problem}
Find a central element $\bm{w}' \in \D$ such that, if $\tilde{\rho}_K^{i,j}$ denote the corresponding polynomials in the rational expansion of $Z_\D$ for this central element, then $$\tilde{\rho}_K^{i,0} =P_K^i.$$
\end{problem}

\subsection{Relation with the 2-loop polynomial}\label{sect relation 2-loop}

The \textit{Kontsevich invariant} (or \textit{Kontsevich integral} due to its classical formulation) \cite{kontsevich} is a  tangle invariant that depends on a choice of \textit{Drinfeld associator}, a certain element  $\varphi \in \mathbb{Q} \langle \langle X,Y \rangle \rangle $, the ring of formal power series in two noncommutative variables $X,Y$ (for knots the invariant is independent of this choice). The Kontsevich invariant is  the strongest knot invariant we know; in fact it is a conjecture that it distinguishes all  isotopy classes of (oriented) knots \cite{ohtsukiproblems}. Moreover, it is known that it dominates the Reshtikhin-Turaev invariants for representations of the quantum groups $U_h (\mathfrak{g})$ for any semisimple complex Lie algebra $\mathfrak{g}$.

The Kontsevich invariant  can be arranged as a strong monoidal functor \cite{lemurakami, ohtsukibook, HabiroMassuyeau} $$Z= Z_\varphi: \mathcal{T}_q \to \hat{\mathcal{A}}$$ where $\mathcal{T}_q$ is the non-strictification of the category $\mathcal{T}$ of framed, oriented tangles in the cube and $\hat{\mathcal{A}}$ is the category of Jacobi diagrams in polarised 1-manifolds. 

The main point is that the value of the Kontsevich invariant $Z(K)$ of a knot $K$ can be viewed as a infinite formal linear combination of unitrivalent graphs. For instance for the unknot $O$, it takes the form 

\begin{equation}\label{eq ZO}
Z(O)  = \exp_\amalg \left( \sum_{m \geq 1} b_{2m} \begin{array}{c}
\labellist
\small\hair 2pt
 \pinlabel{$\cdots$}   at 83 83
  \pinlabel{\scriptsize $2m$}   at 83 105
\endlabellist
\centering
\includegraphics[scale=0.3]{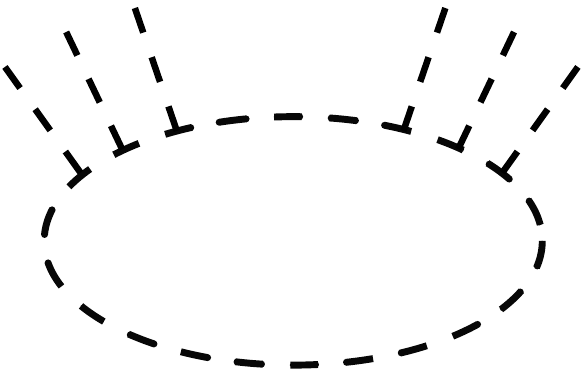} \end{array} \right)
\end{equation}
where $\exp_\amalg$ refers to the exponential power series with respect to the disjoint union of graphs and the coefficients $b_{2m}$ are the \textit{modified Bernoulli numbers}, $$\sum_{m \geq 1} b_{2m} x^{2m} := \frac{1}{2} \log \left(   \frac{\sinh (x/2)}{x/2}  \right) \in \Q [[x]],$$ so $b_2 = \frac{1}{48}, b_4 = - \frac{1}{5760}$, etc. 

The expression \eqref{eq ZO} has two remarkable features: it is expressed as an exponential and the graphs carry \textit{hairs}. It turns out that these two features are common to all knots: solving a conjecture by Rozansky \cite{rozansky12theta}, Kricker \cite{kricker}  (see also \cite{krikergaroufalidis})  showed that the Kontsevich invariant of any 0-framed knot $K$ admits a \textit{loop expansion}
\begin{equation*}
\log_\amalg (Z(K)) = \sum_i \lambda_i  \begin{array}{c}
\labellist
\small\hair 2pt
\pinlabel{$\cdots$}   at 90 95
\endlabellist
\centering
\includegraphics[scale=0.3]{sketch_figures/hairyball} \end{array} + \sum_i \mu_i  \begin{array}{c}
\labellist
\small\hair 2pt
\pinlabel{$\cdots$}   at 140 135
\pinlabel{$\cdots$}   at 140 50
\pinlabel{$\cdots$}   at 140 218
\endlabellist
\centering
\includegraphics[scale=0.3]{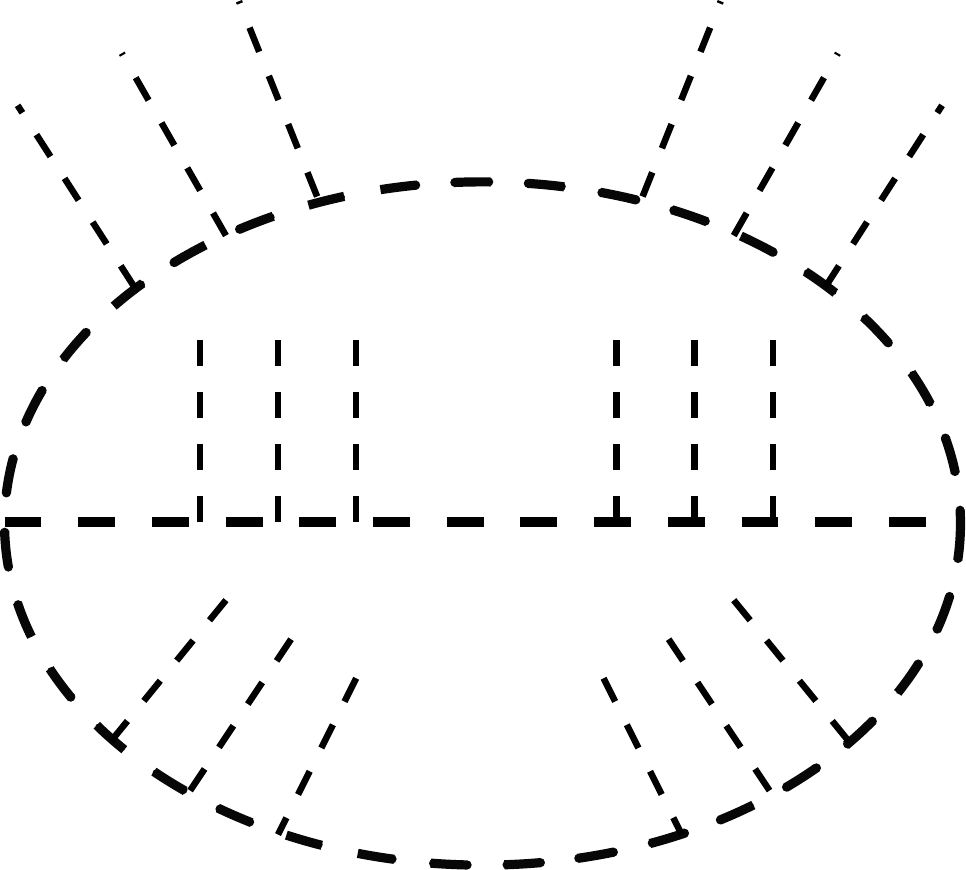} \end{array} + \left( \parbox[c]{4em}{\centering                       {\small  $n$-loop terms, $n>2$ }}  \right).
\end{equation*}
The first summand of the above expression, called the \textit{1-loop part}, can be shown to be tantamount to the Alexander polynomial $\Delta_K$ of $K$. The second summand, called the\textit{ 2-loop part}, is equivalent to a two-variable polynomial
\begin{equation*}
\Theta_K (t_1, t_2)  \in \Q [t^{\pm 1}_1, t^{\pm 1}_2]
\end{equation*}
called the \textit{2-loop polynomial} of $K$  \cite{ohtsuki2loop}. This is a strong knot polynomial invariant, even able to distinguish mutation.

Let us put $$ \hat{\Theta}_K(t):= \Theta_K(t,1) \in \Q [t,t^{-1}].$$ Using the $\mathfrak{sl}_2$-weight system, it is shown in  \cite{ rozansky12theta, ohtsuki_perturbative} that 
$$ \hat{\Theta}_K= P_K^1,$$
so that \cref{conj rho10 =P1} is equivalent to

\begin{conjecture}\label{conj rho10 =Theta}
For any 0-framed knot $K$, $$\rho_K^{1,0} =  \hat{\Theta}_K.$$
\end{conjecture}

In a future publication \cite{becerra2loop}, we will prove \cref{conj rho10 =Theta} for knots of genus less or equal to one.

\appendix

\section{Computations for the band map} \label{sect appendix}

In this appendix we include the Mathematica \cite{mathematica} implementation for the computations of the band map $B$ that we used in \cref{sect band map}.  For this we use the code written by Bar-Natan and van der Veen for \cite{barnatanveengaussians}, which computes the generating series of the ribbon algebra structure maps of $\D$ and performs composites using the Contraction \cref{thm contraction}. This is freely available on van der Veen's website \href{http://www.rolandvdv.nl/PG/}{\texttt{ http://www.rolandvdv.nl/PG/}}.

First, we define the generating series $B_{ij}^k$ of the band map as in \cref{sect thickening map}. As usual in Mathematica language, double slash stands for composition from left to right (or pairing of generating series).

\vspace*{0.3cm}
\noindent\includegraphics[scale=0.65,page=1]{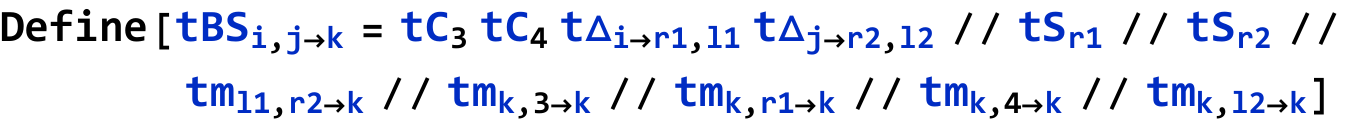}
\vspace*{0.3cm}

Here \textcolor{blue}{\texttt{tm}}, \textcolor{blue}{\texttt{t$\Delta$}}, \textcolor{blue}{\texttt{tS}} and \textcolor{blue}{\texttt{tC}} stand for the multiplication, comultiplication, antipode and balancing element, respectively. The result is a list-like type whose second entry is the exponential part and whose third entry is the perturbation part of the generating series. From this we can compute the coefficients of \cref{lemma P1 P2 B}:

\vspace*{0.3cm}
\noindent\includegraphics[scale=0.65,page=1]{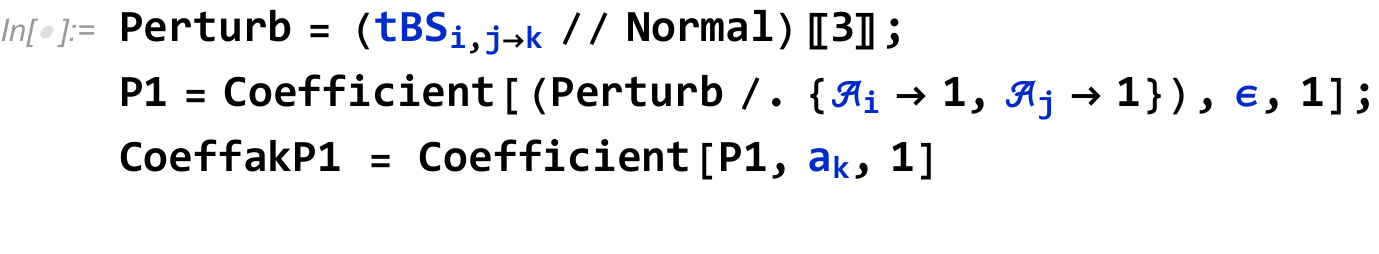}

\noindent\includegraphics[scale=0.65,page=1]{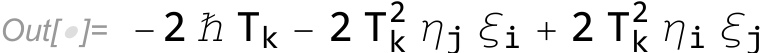}

\vspace*{0.3cm}
\noindent\includegraphics[scale=0.65,page=1]{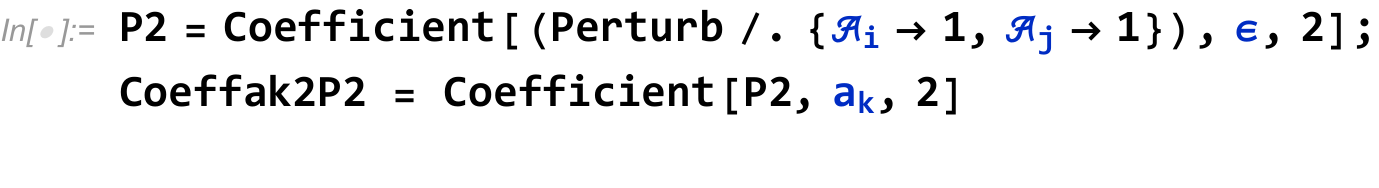}

\noindent\includegraphics[scale=0.65,page=1]{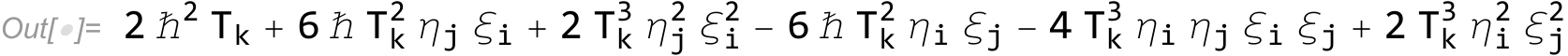}
\vspace*{0.3cm}

Next we check the two equalities that we claimed in the proof of \cref{prop gen ser Th with alpha neq 0}. 

\vspace*{0.3cm}
\noindent\includegraphics[scale=0.65,page=1]{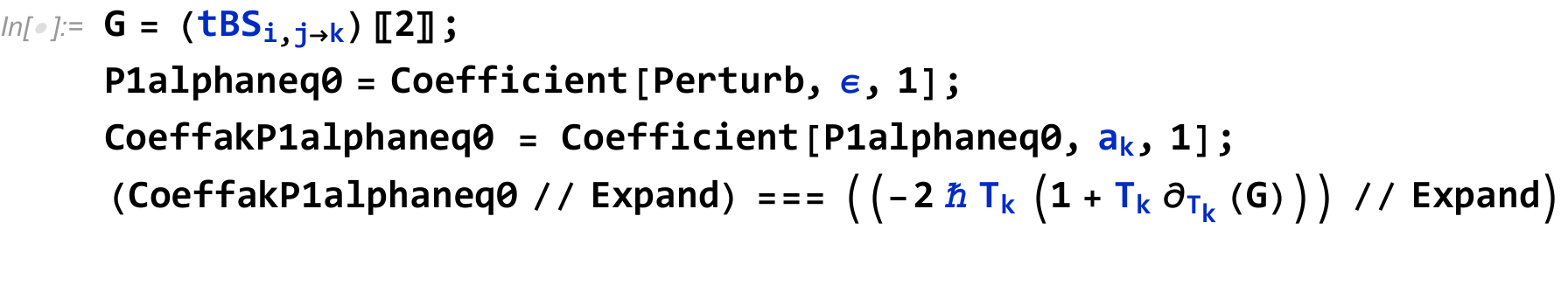}

\noindent\includegraphics[scale=0.65,page=1]{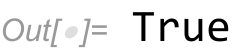}

\vspace*{0.3cm}
\noindent\includegraphics[scale=0.65,page=1]{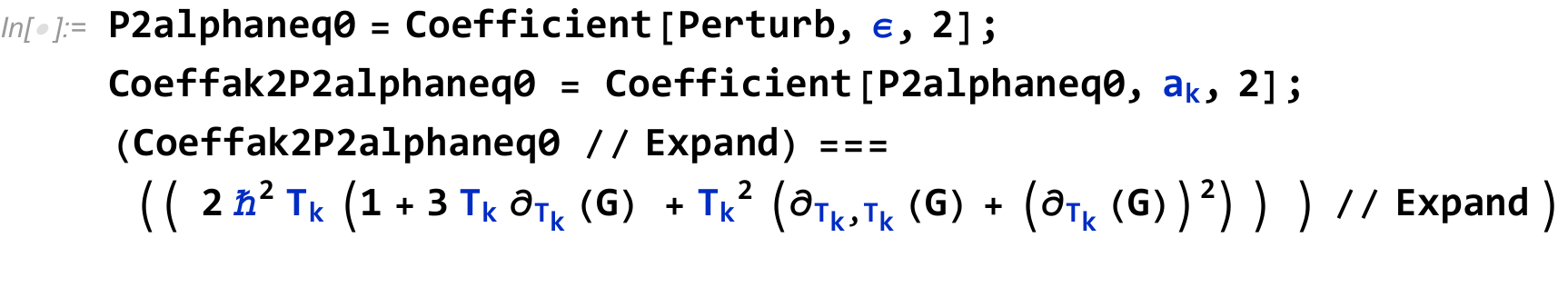}

\noindent\includegraphics[scale=0.65,page=1]{folder_code/code007.pdf}
\vspace*{0.3cm}

Finally we directly check the statement of \cref{lemma 77} (for $g=1$). 
\vspace*{0.3cm}

\noindent\includegraphics[scale=0.65,page=1]{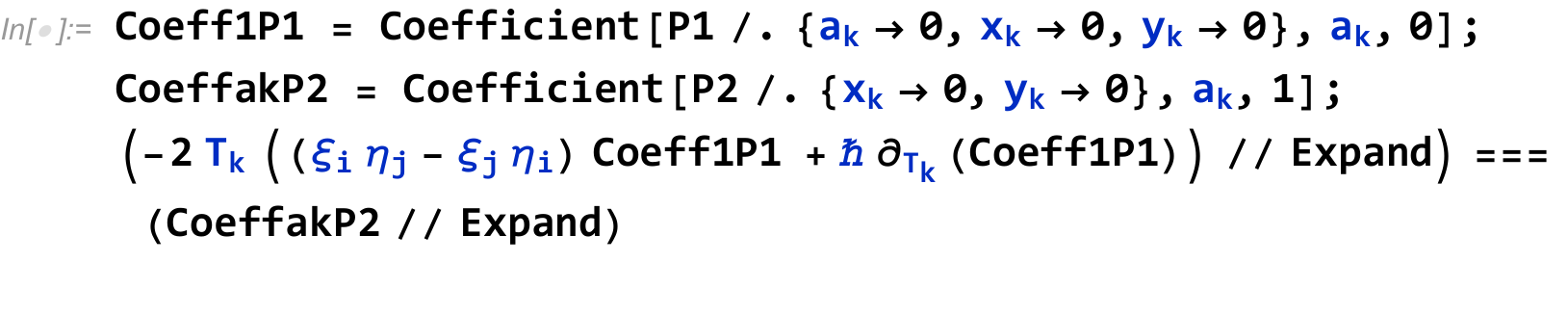}

\noindent\includegraphics[scale=0.65,page=1]{folder_code/code007.pdf}
\vspace*{0.3cm}

\bibliographystyle{halpha-abbrv}
\bibliography{bibliografia}

\begin{thebibliography}{Oht02b}
\expandafter\ifx\csname url\endcsname\relax
  \def\url#1{\texttt{#1}}\fi
\expandafter\ifx\csname doi\endcsname\relax
  \def\doi#1{\burlalt{doi:#1}{http://dx.doi.org/#1}}\fi
\expandafter\ifx\csname urlprefix\endcsname\relax\def\urlprefix{URL }\fi
\expandafter\ifx\csname href\endcsname\relax
  \def\href#1#2{#2}\fi
\expandafter\ifx\csname burlalt\endcsname\relax
  \def\burlalt#1#2{\href{#2}{#1}}\fi

\bibitem[BBG18]{BBG}
A.~Beliakova, C.~Blanchet, and N.~Geer.
\newblock {Logarithmic Hennings invariants for restricted quantum
  ${\mathfrak{sl}}(2)$}.
\newblock {\em Algebraic \& Geometric Topology}, 18(7):4329 -- 4358, 2018.
\newblock \doi{10.2140/agt.2018.18.4329}.

\bibitem[Bec23a]{becerra2loop}
J.~Becerra.
\newblock A {H}opf algebraic construction of the 2-loop polynomial of genus one
  knots.
\newblock In preparation, 2023+.

\bibitem[Bec23b]{becerra}
J.~Becerra.
\newblock On the naturality of the universal tangle invariant.
\newblock In preparation, 2023+.

\bibitem[Big02]{bigelow}
S.~Bigelow.
\newblock A homological definition of the {J}ones polynomial.
\newblock In {\em Invariants of knots and 3-manifolds ({K}yoto, 2001)},
  volume~4 of {\em Geom. Topol. Monogr.}, pages 29--41. Geom. Topol. Publ.,
  Coventry, 2002.
\newblock \doi{10.2140/gtm.2002.4.29}.

\bibitem[BNG96]{BNG}
D.~Bar-Natan and S.~Garoufalidis.
\newblock On the {M}elvin-{M}orton-{R}ozansky conjecture.
\newblock {\em Invent. Math.}, 125(1):103--133, 1996.
\newblock \doi{10.1007/s002220050070}.

\bibitem[BV18]{barnatanveenpolytime}
D.~Bar-Natan and R.~van~der Veen.
\newblock A polynomial time knot polynomial.
\newblock {\em Proc. Amer. Math. Soc.}, 147(1):377--397, 2019.
\newblock \doi{10.1090/proc/14166}.

\bibitem[BV21]{barnatanveengaussians}
D.~Bar-Natan and R.~van~der Veen.
\newblock Perturbed {G}aussian generating functions for universal knot
  invariants, 2021.
\newblock \doi{10.48550/ARXIV.2109.02057}.

\bibitem[BV22]{barnatanveenAPAI}
D.~Bar-Natan and R.~van~der Veen.
\newblock A {P}erturbed-{A}lexander {I}nvariant, 2022.

\bibitem[CP94]{chari_pressley}
V.~Chari and A.~Pressley.
\newblock {\em A guide to quantum groups}.
\newblock Cambridge University Press, Cambridge, 1994.

\bibitem[Dri85]{drinfeld85}
V.~G. Drinfeld.
\newblock Hopf algebras and the quantum {Y}ang-{B}axter equation.
\newblock {\em Dokl. Akad. Nauk SSSR}, 283(5):1060--1064, 1985.

\bibitem[GK04]{krikergaroufalidis}
S.~Garoufalidis and A.~Kricker.
\newblock A rational noncommutative invariant of boundary links.
\newblock {\em Geom. Topol.}, 8:115--204, 2004.
\newblock \doi{10.2140/gt.2004.8.115}.

\bibitem[GL05]{analytic}
S.~Garoufalidis and T.~T.~Q. Le.
\newblock An analytic version of the melvin-morton-rozansky conjecture, 2005.
\newblock \doi{10.48550/ARXIV.MATH/0503641}.

\bibitem[Hab02]{habiro_sl2}
K.~Habiro.
\newblock On the quantum {$\rm sl_2$} invariants of knots and integral homology
  spheres.
\newblock In {\em Invariants of knots and 3-manifolds ({K}yoto, 2001)},
  volume~4 of {\em Geom. Topol. Monogr.}, pages 55--68. Geom. Topol. Publ.,
  Coventry, 2002.
\newblock \doi{10.2140/gtm.2002.4.55}.

\bibitem[Hab06]{habiro}
K.~Habiro.
\newblock Bottom tangles and universal invariants.
\newblock {\em Algebr. Geom. Topol.}, 6:1113--1214, 2006.
\newblock \doi{10.2140/agt.2006.6.1113}.

\bibitem[Hab08]{habiro_WRT}
K.~Habiro.
\newblock A unified {W}itten-{R}eshetikhin-{T}uraev invariant for integral
  homology spheres.
\newblock {\em Invent. Math.}, 171(1):1--81, 2008.
\newblock \doi{10.1007/s00222-007-0071-0}.

\bibitem[HM21]{HabiroMassuyeau}
K.~Habiro and G.~Massuyeau.
\newblock The {K}ontsevich integral for bottom tangles in handlebodies.
\newblock {\em Quantum Topol.}, 12(4):593--703, 2021.
\newblock \doi{10.4171/qt/155}.

\bibitem[Jim85]{jimbo85}
M.~Jimbo.
\newblock A {$q$}-difference analogue of {$U(\mathfrak{g})$} and the
  {Y}ang-{B}axter equation.
\newblock {\em Lett. Math. Phys.}, 10(1):63--69, 1985.
\newblock \doi{10.1007/BF00704588}.

\bibitem[Jon85]{jones}
V.~F.~R. Jones.
\newblock A polynomial invariant for knots via von {N}eumann algebras.
\newblock {\em Bull. Amer. Math. Soc. (N.S.)}, 12(1):103--111, 1985.
\newblock \doi{10.1090/S0273-0979-1985-15304-2}.

\bibitem[Kas95]{kassel}
C.~Kassel.
\newblock {\em Quantum groups}, volume 155 of {\em Graduate Texts in
  Mathematics}.
\newblock Springer-Verlag, New York, 1995.
\newblock \doi{10.1007/978-1-4612-0783-2}.

\bibitem[Kon93]{kontsevich}
M.~Kontsevich.
\newblock Vassiliev's knot invariants.
\newblock In {\em I. {M}. {G}el'fand {S}eminar}, volume~16 of {\em Adv. Soviet
  Math.}, pages 137--150. Amer. Math. Soc., Providence, RI, 1993.

\bibitem[KR93]{kauffman_radford}
L.~H. Kauffman and D.~E. Radford.
\newblock A necessary and sufficient condition for a finite-dimensional
  {D}rinfeld double to be a ribbon {H}opf algebra.
\newblock {\em J. Algebra}, 159(1):98--114, 1993.
\newblock \doi{10.1006/jabr.1993.1148}.

\bibitem[Kri00]{kricker}
A.~Kricker.
\newblock The lines of the kontsevich integral and rozansky's rationality
  conjecture, 2000.
\newblock \doi{10.48550/ARXIV.MATH/0005284}.

\bibitem[Law89]{lawrence}
R.~J. Lawrence.
\newblock A universal link invariant using quantum groups.
\newblock In {\em Differential geometric methods in theoretical physics
  ({C}hester, 1988)}, pages 55--63. World Sci. Publ., Teaneck, NJ, 1989.

\bibitem[Lic97]{lickorish}
W.~B.~R. Lickorish.
\newblock {\em An introduction to knot theory}, volume 175 of {\em Graduate
  Texts in Mathematics}.
\newblock Springer-Verlag, New York, 1997.
\newblock \doi{10.1007/978-1-4612-0691-0}.

\bibitem[LM96]{lemurakami}
T.~Q.~T. Le and J.~Murakami.
\newblock The universal {V}assiliev-{K}ontsevich invariant for framed oriented
  links.
\newblock {\em Compositio Math.}, 102(1):41--64, 1996.
\newblock \urlprefix\url{http://www.numdam.org/item?id=CM_1996__102_1_41_0}.

\bibitem[MM95]{MM}
P.~M. Melvin and H.~R. Morton.
\newblock The coloured {J}ones function.
\newblock {\em Comm. Math. Phys.}, 169(3):501--520, 1995.
\newblock \urlprefix\url{http://projecteuclid.org/euclid.cmp/1104272852}.

\bibitem[MS16]{suzuki16}
J.-B. Meilhan and S.~Suzuki.
\newblock The universal {$sl_2$} invariant and {M}ilnor invariants.
\newblock {\em Internat. J. Math.}, 27(11):1650090, 37, 2016.
\newblock \doi{10.1142/S0129167X16500907}.

\bibitem[Oht02a]{ohtsukiproblems}
T.~Ohtsuki.
\newblock Problems on invariants of knots and 3-manifolds.
\newblock In {\em Invariants of knots and 3-manifolds ({K}yoto, 2001)},
  volume~4 of {\em Geom. Topol. Monogr.}, pages i--iv, 377--572. Geom. Topol.
  Publ., Coventry, 2002.
\newblock \doi{10.2140/gtm.2002.4}.
\newblock With an introduction by J. Roberts.

\bibitem[Oht02b]{ohtsukibook}
T.~Ohtsuki.
\newblock {\em Quantum invariants}, volume~29 of {\em Series on Knots and
  Everything}.
\newblock World Scientific Publishing Co., Inc., River Edge, NJ, 2002.
\newblock A study of knots, 3-manifolds, and their sets.

\bibitem[Oht07]{ohtsuki2loop}
T.~Ohtsuki.
\newblock On the 2-loop polynomial of knots.
\newblock {\em Geom. Topol.}, 11:1357--1475, 2007.
\newblock \doi{10.2140/gt.2007.11.1357}.

\bibitem[Oht10]{ohtsuki_perturbative}
T.~Ohtsuki.
\newblock Perturbative invariants of 3-manifolds with the first {B}etti number
  1.
\newblock {\em Geom. Topol.}, 14(4):1993--2045, 2010.
\newblock \doi{10.2140/gt.2010.14.1993}.

\bibitem[Pol10]{polyak10}
M.~Polyak.
\newblock Minimal generating sets of {R}eidemeister moves.
\newblock {\em Quantum Topol.}, 1(4):399--411, 2010.
\newblock \doi{10.4171/QT/10}.

\bibitem[Qua22]{quarles}
R.~J. Quarles.
\newblock {\em A {N}ew {P}erspective on a {P}olynomial {T}ime {K}not
  {P}olynomial}.
\newblock Pro\-Quest LLC, Ann Arbor, MI, 2022.
\newblock \doi{10.31390/gradschool{\_}dissertations.5806}.
\newblock Thesis (Ph.D.)--Louisiana State University and Agricultural \&
  Mechanical College.

\bibitem[Rie14]{riehl2014}
E.~Riehl.
\newblock {\em Categorical homotopy theory}, volume~24 of {\em New Mathematical
  Monographs}.
\newblock Cambridge University Press, Cambridge, 2014.
\newblock \doi{10.1017/CBO9781107261457}.

\bibitem[Roz97]{rozanskyconjectures}
L.~Rozansky.
\newblock Higher order terms in the {M}elvin-{M}orton expansion of the colored
  {J}ones polynomial.
\newblock {\em Comm. Math. Phys.}, 183(2):291--306, 1997.
\newblock \doi{10.1007/BF02506408}.

\bibitem[Roz98]{rozanskyRmatrix}
L.~Rozansky.
\newblock The universal ${R}$-matrix, {B}urau representation, and the {M}elvin
  - {M}orton expansion of the colored {J}ones polynomial.
\newblock {\em Adv. Math.}, 134(1):1--31, 1998.
\newblock \doi{10.1006/aima.1997.1661}.

\bibitem[Roz03]{rozansky12theta}
L.~Rozansky.
\newblock A rationality conjecture about {K}ontsevich integral of knots and its
  implications to the structure of the colored {J}ones polynomial.
\newblock In {\em Proceedings of the {P}acific {I}nstitute for the
  {M}athematical {S}ciences {W}orkshop ``{I}nvariants of {T}hree-{M}anifolds''
  ({C}algary, {AB}, 1999)}, volume 127, pages 47--76, 2003.
\newblock \doi{10.1016/S0166-8641(02)00053-6}.

\bibitem[Suz10]{suzuki10}
S.~Suzuki.
\newblock On the universal {${\rm sl}_2$} invariant of ribbon bottom tangles.
\newblock {\em Algebr. Geom. Topol.}, 10(2):1027--1061, 2010.
\newblock \doi{10.2140/agt.2010.10.1027}.

\bibitem[Suz13]{suzuki13}
S.~Suzuki.
\newblock On the universal {$sl_2$} invariant of {B}runnian bottom tangles.
\newblock {\em Math. Proc. Cambridge Philos. Soc.}, 154(1):127--143, 2013.
\newblock \doi{10.1017/S0305004112000503}.

\bibitem[Suz18]{suzuki18}
S.~Suzuki.
\newblock The universal quantum invariant and colored ideal triangulations.
\newblock {\em Algebr. Geom. Topol.}, 18(6):3363--3402, 2018.
\newblock \doi{10.2140/agt.2018.18.3363}.

\bibitem[Tur16]{turaev}
V.~G. Turaev.
\newblock {\em Quantum invariants of knots and 3-manifolds}, volume~18 of {\em
  De Gruyter Studies in Mathematics}.
\newblock De Gruyter, Berlin, 2016.
\newblock \doi{10.1515/9783110435221}.
\newblock Third edition [of MR1292673].

\bibitem[Wol]{mathematica}
Wolfram.
\newblock Mathematica, {V}ersion 13.2.
\newblock \urlprefix\url{https://www.wolfram.com/mathematica}.
\newblock Champaign, IL, 2022.

\end{thebibliography}

\end{document}